\documentclass[11pt]{preprint}
\usepackage[full]{textcomp}
\usepackage[osf]{newtxtext} 
\usepackage[cal=boondoxo]{mathalfa}
\usepackage{colortbl}

\usepackage{comment}
\usepackage[usenames,dvipsnames]{xcolor}

\newcommand{\pxyo}{\pi_{xy}^{(\mathbf{0})}}

\newcommand{\mb}[1]{\mathbb{#1}}

\usepackage{amssymb}
\usepackage{lmodern}
\usepackage{mathtools}

\usepackage{hyperref}
\usepackage{breakurl}
\usepackage{aligned-overset}
\usepackage{mhenvs}
\usepackage{mhequ} 
\newcommand{\be}{\begin{equation*}}
\newcommand{\ee}{\end{equation*}}
\usepackage{mhsymb}
\usepackage{booktabs}
\usepackage{tikz}
\usepackage{tcolorbox}
\usepackage{mathrsfs}
\usepackage[utf8]{inputenc}
\usepackage{longtable}
\usepackage{wrapfig}
\usepackage{subcaption}
\usepackage{mathrsfs}
\usepackage{epsfig}
\usepackage{microtype}
\usepackage{comment}
\usepackage{wasysym}
\usepackage{centernot}
\usepackage{enumitem}
\usepackage{bm}
\usepackage{stackrel}
\usepackage{graphicx}

\makeatletter
\newcommand{\globalcolor}[1]{%
  \color{#1}\global\let\default@color\current@color
}
\makeatother

\usetikzlibrary{calc}
\usetikzlibrary{decorations}
\usetikzlibrary{positioning}
\usetikzlibrary{shapes}
\usetikzlibrary{external}

\definecolor{blush}{rgb}{0.87, 0.36, 0.51}
	\definecolor{brightcerulean}{rgb}{0.11, 0.67, 0.84}
	\definecolor{greenryb}{rgb}{0.4, 0.69, 0.2}

\newif\ifdark
\darkfalse

\ifdark
\definecolor{darkred}{rgb}{0.9,0.2,0.2}
\definecolor{darkblue}{rgb}{0.7,0.3,1}
\definecolor{darkgreen}{rgb}{0.1,0.9,0.1}
\definecolor{franck}{rgb}{0,0.8,1}
\definecolor{pagebackground}{rgb}{.15,.21,.18}
\definecolor{pageforeground}{rgb}{.84,.84,.85}
\pagecolor{pagebackground}
\AtBeginDocument{\globalcolor{pageforeground}}
\definecolor{symbols}{rgb}{0,0.7,1}
\colorlet{connection}{red!80!black}
\colorlet{boxcolor}{blue!50}

\else

\definecolor{darkred}{rgb}{0.7,0.1,0.1}
\definecolor{darkblue}{rgb}{0.4,0.1,0.8}
\definecolor{darkgreen}{rgb}{0.1,0.7,0.1}
\definecolor{franck}{rgb}{0,0,1}
\definecolor{pagebackground}{rgb}{1,1,1}
\definecolor{pageforeground}{rgb}{0,0,0}
\colorlet{symbols}{blue!90!black}
\colorlet{connection}{red!30!black}
\colorlet{boxcolor}{blue!50!black}

\fi

\def\slash{\leavevmode\unskip\kern0.18em/\penalty\exhyphenpenalty\kern0.18em}
\def\dash{\leavevmode\unskip\kern0.18em--\penalty\exhyphenpenalty\kern0.18em}

\DeclareMathAlphabet{\mathbbm}{U}{bbm}{m}{n}

\DeclareFontFamily{U}{BOONDOX-calo}{\skewchar\font=45 }
\DeclareFontShape{U}{BOONDOX-calo}{m}{n}{
  <-> s*[1.05] BOONDOX-r-calo}{}
\DeclareFontShape{U}{BOONDOX-calo}{b}{n}{
  <-> s*[1.05] BOONDOX-b-calo}{}
\DeclareMathAlphabet{\mcb}{U}{BOONDOX-calo}{m}{n}
\SetMathAlphabet{\mcb}{bold}{U}{BOONDOX-calo}{b}{n}

\setlist{noitemsep,topsep=4pt,leftmargin=1.5em}

\DeclareMathAlphabet{\mathbbm}{U}{bbm}{m}{n}

\DeclareMathAlphabet{\mcb}{U}{BOONDOX-calo}{m}{n}
\SetMathAlphabet{\mcb}{bold}{U}{BOONDOX-calo}{b}{n}
\DeclareFontFamily{U}{mathx}{\hyphenchar\font45}
\DeclareFontShape{U}{mathx}{m}{n}{
      <5> <6> <7> <8> <9> <10>
      <10.95> <12> <14.4> <17.28> <20.74> <24.88>
      mathx10
      }{}
\DeclareSymbolFont{mathx}{U}{mathx}{m}{n}
\DeclareMathSymbol{\bigtimes}{1}{mathx}{"91}

\def\s{\mathfrak{s}}

\providecommand{\figures}{false}
{ \ifthenelse{\equal{\figures}{false}} {#1}{\[ {\rm Figure \ missing !} \]} }{}
\def\id{\mathrm{id}}
\def\Id{\mathrm{Id}}

\usepackage{mathtools}

\def\CA{\mathcal{A}}

\def\CQ{\mathcal{Q}}

\def\CM{\mathcal{M}}
\def\CT{\mathcal{T}}

\tikzstyle{tinydots}=[dash pattern=on \pgflinewidth off \pgflinewidth]
\tikzstyle{superdense}=[dash pattern=on 4pt off 1pt]

\newcommand{\mcM}{\mathcal{M}}

\newcommand{\mcI}{\mathcal{I}}

\newcommand{\gi}[2]{\gamma^{R,#1}_{xy}}

\newcommand{\beq}{\begin{equation}}
\newcommand{\eeq}{\end{equation}}

\usepackage{empheq}

\newcommand{\mbn}{\mathbf{n}}

\def\Labe{\mathfrak{e}}

\def\Labn{\mathfrak{n}}

\def\${|\!|\!|}

\newcounter{theorem}

\newtheorem{ex}[theorem]{Example}

\newenvironment{DIFnomarkup}{}{} 

\newtheorem{assumption}{Assumption}
 
\theorembodyfont{\rmfamily}

\newfont{\indic}{bbmss12}

\def\PPi{\boldsymbol{\Pi}}

\def\Nabla_#1{\nabla_{\!#1}}

\makeatletter
\pgfdeclareshape{crosscircle}
{
  \inheritsavedanchors[from=circle] 
  \inheritanchorborder[from=circle]
  \inheritanchor[from=circle]{north}
  \inheritanchor[from=circle]{north west}
  \inheritanchor[from=circle]{north east}
  \inheritanchor[from=circle]{center}
  \inheritanchor[from=circle]{west}
  \inheritanchor[from=circle]{east}
  \inheritanchor[from=circle]{mid}
  \inheritanchor[from=circle]{mid west}
  \inheritanchor[from=circle]{mid east}
  \inheritanchor[from=circle]{base}
  \inheritanchor[from=circle]{base west}
  \inheritanchor[from=circle]{base east}
  \inheritanchor[from=circle]{south}
  \inheritanchor[from=circle]{south west}
  \inheritanchor[from=circle]{south east}
  \inheritbackgroundpath[from=circle]
  \foregroundpath{
    \centerpoint%
    \pgf@xc=\pgf@x%
    \pgf@yc=\pgf@y%
    \pgfutil@tempdima=\radius%
    \pgfmathsetlength{\pgf@xb}{\pgfkeysvalueof{/pgf/outer xsep}}%
    \pgfmathsetlength{\pgf@yb}{\pgfkeysvalueof{/pgf/outer ysep}}%
    \ifdim\pgf@xb<\pgf@yb%
      \advance\pgfutil@tempdima by-\pgf@yb%
    \else%
      \advance\pgfutil@tempdima by-\pgf@xb%
    \fi%
    \pgfpathmoveto{\pgfpointadd{\pgfqpoint{\pgf@xc}{\pgf@yc}}{\pgfqpoint{-0.707107\pgfutil@tempdima}{-0.707107\pgfutil@tempdima}}}
    \pgfpathlineto{\pgfpointadd{\pgfqpoint{\pgf@xc}{\pgf@yc}}{\pgfqpoint{0.707107\pgfutil@tempdima}{0.707107\pgfutil@tempdima}}}
    \pgfpathmoveto{\pgfpointadd{\pgfqpoint{\pgf@xc}{\pgf@yc}}{\pgfqpoint{-0.707107\pgfutil@tempdima}{0.707107\pgfutil@tempdima}}}
    \pgfpathlineto{\pgfpointadd{\pgfqpoint{\pgf@xc}{\pgf@yc}}{\pgfqpoint{0.707107\pgfutil@tempdima}{-0.707107\pgfutil@tempdima}}}
  }
}
\makeatother

\def\symbol#1{\textcolor{symbols}{#1}}

\def\decorate#1#2{
        \ifnum#2>0
    		\foreach \count in {1,...,#2}{
	       	let
				\p1 = (sourcenode.center),
                \p2 = (sourcenode.east),
				\n1 = {\x2-\x1},
				\n2 = {1mm},
				\n3 = {(1.3+0.6*(\count-1))*\n1},
				\n4 = {0.7*\n1}
			in 
        		node[rectangle,fill=symbols,rotate=30,inner sep=0pt,minimum width=0.2*\n2,minimum height=\n2] at ($(sourcenode.center) + (\n3,\n4)$) {}
				}
		\fi
        \ifnum#1>0
    		\foreach \count in {1,...,#1}{
	       	let
				\p1 = (sourcenode.center),
                \p2 = (sourcenode.east),
				\n1 = {\x2-\x1},
				\n2 = {1mm},
				\n3 = {(1.3+0.6*(\count-1))*\n1},
				\n4 = {0.7*\n1}
			in 
        		node[rectangle,fill=symbols,rotate=-30,inner sep=0pt,minimum width=0.2*\n2,minimum height=\n2] at ($(sourcenode.center) + (-\n3,\n4)$) {}
				}
		\fi
}

\tikzset{
    dectriangle/.style 2 args={
        triangle,
        alias=sourcenode,
        append after command={\decorate{#1}{#2}}
    },
    dectriangle/.default={0}{0},
}

\tikzset{
	cross/.style={path picture={ 
  		\draw[symbols]
			(path picture bounding box.south east) -- (path picture bounding box.north west) (path picture bounding box.south west) -- (path picture bounding box.north east);
		}},
root/.style={circle,fill=green!50!black,inner sep=0pt, minimum size=1.2mm},
        dot/.style={circle,fill=pageforeground,inner sep=0pt, minimum size=1mm},
        dotred/.style={circle,fill=pageforeground!50!pagebackground,inner sep=0pt, minimum size=2mm},
        var/.style={circle,fill=pageforeground!10!pagebackground,draw=pageforeground,inner sep=0pt, minimum size=3mm},
        kernel/.style={semithick,shorten >=2pt,shorten <=2pt},
        kernels/.style={snake=zigzag,shorten >=2pt,shorten <=2pt,segment amplitude=1pt,segment length=4pt,line before snake=2pt,line after snake=5pt,},
        rho/.style={densely dashed,semithick,shorten >=2pt,shorten <=2pt},
           testfcn/.style={dotted,semithick,shorten >=2pt,shorten <=2pt},
        renorm/.style={shape=circle,fill=pagebackground,inner sep=1pt},
        labl/.style={shape=rectangle,fill=pagebackground,inner sep=1pt},
        xic/.style={very thin,circle,draw=symbols,fill=symbols,inner sep=0pt,minimum size=1.2mm},
        g/.style={very thin,rectangle,draw=symbols,fill=symbols!10!pagebackground,inner sep=0pt,minimum width=2.5mm,minimum height=1.2mm},
        xi/.style={very thin,circle,draw=symbols,fill=symbols!10!pagebackground,inner sep=0pt,minimum size=1.2mm},
	xies/.style={very thin,rectangle,fill=green!50!black!25,draw=symbols,inner sep=0pt,minimum size=1.1mm},
	xiesf/.style={very thin,rectangle,fill=green!50!black,draw=symbols,inner sep=0pt,minimum size=1.1mm},
        xix/.style={very thin,crosscircle,fill=symbols!10!pagebackground,draw=symbols,inner sep=0pt,minimum size=1.2mm},
        X/.style={very thin,cross,rectangle,fill=pagebackground,draw=symbols,inner sep=0pt,minimum size=1.2mm},
	xib/.style={thin,circle,fill=symbols!10!pagebackground,draw=symbols,inner sep=0pt,minimum size=1.6mm},
	xie/.style={thin,circle,fill=green!50!black,draw=symbols,inner sep=0pt,minimum size=1.6mm},
	xid/.style={thin,circle,fill=symbols,draw=symbols,inner sep=0pt,minimum size=1.6mm},
	xibx/.style={thin,crosscircle,fill=symbols!10!pagebackground,draw=symbols,inner sep=0pt,minimum size=1.6mm},
	kernels2/.style={very thick,draw=connection,segment length=12pt},
	keps/.style={thin,draw=symbols,->},
	kepspr/.style={thick,draw=connection,->},
	krho/.style={thin,draw=symbols,superdense,->},
	krhopr/.style={thick,draw=connection,superdense},
	triangle/.style = { regular polygon, regular polygon sides=3},
	not/.style={thin,circle,draw=connection,fill=connection,inner sep=0pt,minimum size=0.5mm},
	diff/.style = {very thin,draw=symbols,triangle,fill=red!50!black,inner sep=0pt,minimum size=1.6mm},
	diff1/.style = {very thin,dectriangle={1}{0},fill=red!50!black,draw=symbols,inner sep=0pt,minimum size=1.6mm},
	diff2/.style = {very thin,dectriangle={1}{1},fill=red!50!black,draw=symbols,inner sep=0pt,minimum size=1.6mm},
		diffmini/.style = {very thin,rectangle,fill=black,draw=black,inner sep=0pt,minimum size=0.75mm},
	 kernelsmod/.style={very thick,draw=connection,segment length=12pt},
	 rec/.style = {very thin,rectangle,fill=black,draw=black,inner sep=0pt,minimum size=2mm},
	cerc/.style={very thin,circle,draw=black,fill=symbols,inner sep=0pt,minimum size=2mm},
	stars/.style={very thin,star,star points=6,star point ratio=0.5, draw=black,fill=red,inner sep=0pt,minimum size=0.7mm},
	>=stealth,
        }
        \tikzset{
root/.style={circle,fill=black!50,inner sep=0pt, minimum size=3mm},
        circ/.style={circle,fill=white,draw=black,very thin,inner sep=.5pt, minimum size=1.2mm},
        round1/.style={fill=white,outer sep = 0,inner sep=2pt,rounded corners=1mm,draw,text=black,thin,minimum size=1.2mm},
          circ1/.style={circle,fill=red!10,draw=red,very thin,inner sep=.5pt, minimum size=1.2mm},
        rect/.style={fill=white,outer sep = 0,inner sep=2pt,rectangle,draw,text=black,thin,minimum size=1.2mm},
        rect1/.style={fill=white,outer sep = 0,inner sep=2pt,rectangle,draw,text=black,thin,minimum size=1.2mm},
        round2/.style={fill=red!10,outer sep = 0,inner sep=2pt,rounded corners=1mm,draw,text=black,thin,minimum size=1.2mm},
       round3/.style={fill=blue!10,outer sep = 0,inner sep=2pt,rounded corners=1mm,draw,text=black,thin,minimum size=1.2mm}, 
        rect2/.style={fill=black!10,outer sep = 0,inner sep=2pt,rectangle,draw,text=black,thin,minimum size=1.2mm},
        dot/.style={circle,fill=black,inner sep=0pt, minimum size=1.2mm},
        dotred/.style={circle,fill=black!50,inner sep=0pt, minimum size=2mm},
        var/.style={circle,fill=black!10,draw=black,inner sep=0pt, minimum size=3mm},
        kernel/.style={semithick,shorten >=2pt,shorten <=2pt},
         diag/.style={thin,shorten >=4pt,shorten <=4pt},
        kernel1/.style={thick},
        kernels/.style={snake=zigzag,shorten >=2pt,shorten <=2pt,segment amplitude=1pt,segment length=4pt,line before snake=2pt,line after snake=5pt,},
		kernels1/.style={snake=zigzag,segment amplitude=0.5pt,segment length=2pt},
		rho1/.style={densely dotted,semithick},
        rho/.style={densely dashed,semithick,shorten >=2pt,shorten <=2pt},
           testfcn/.style={dotted,semithick,shorten >=2pt,shorten <=2pt},
           visible/.style={draw, circle, fill, inner sep=0.25ex},
        renorm/.style={shape=circle,fill=white,inner sep=1pt},
        labl/.style={shape=rectangle,fill=white,inner sep=1pt},
        xic/.style={very thin,circle,fill=symbols,draw=black,inner sep=0pt,minimum size=1.2mm},
        xi/.style={very thin,circle,fill=blue!10,draw=black,inner sep=0pt,minimum size=1.2mm},
	xib/.style={very thin,circle,fill=blue!10,draw=black,inner sep=0pt,minimum size=1.6mm},
	xie/.style={very thin,circle,fill=green!50!black,draw=black,inner sep=0pt,minimum size=1mm},
	xid/.style={very thin,circle,fill=symbols,draw=black,inner sep=0pt,minimum size=1.6mm},
	edgetype/.style={very thin,circle,draw=black,inner sep=0pt,minimum size=5mm},
	nodetype/.style={very thick,circle,draw=black,inner sep=0pt,minimum size=5mm},
	kernels2/.style={very thick,draw=connection,segment length=12pt},
clean/.style={thin,circle,fill=black,inner sep=0pt,minimum size=1mm},	not/.style={thin,circle,fill=symbols,draw=connection,fill=connection,inner sep=0pt,minimum size=0.8mm},
	>=stealth,
        }

\makeatletter
\def\DeclareSymbol#1#2#3{%
	\expandafter\gdef\csname MH@symb@#1\endcsname{\tikzsetnextfilename{symbol#1}%
	\tikz[baseline=#2,scale=0.15,draw=symbols,line join=round]{#3}}%
	\expandafter\gdef\csname MH@symb@#1s\endcsname{\scalebox{0.75}{\tikzsetnextfilename{symbol#1}%
	\tikz[baseline=#2,scale=0.15,draw=symbols,line join=round]{#3}}}%
	\expandafter\gdef\csname MH@symb@#1ss\endcsname{\scalebox{0.65}{\tikzsetnextfilename{symbol#1}%
	\tikz[baseline=#2,scale=0.15,draw=symbols,line join=round]{#3}}}%
	}
\def\<#1>{\ifthenelse{\boolean{mmode}}{\mathchoice{\csname MH@symb@#1\endcsname}{\csname MH@symb@#1\endcsname}{\csname MH@symb@#1s\endcsname}{\csname MH@symb@#1ss\endcsname}}{\csname MH@symb@#1\endcsname}}
\makeatother

\DeclareSymbol{Xi22}{0.5}{\draw (0,0) node[xi] {} -- (-1,1) node[not] {} -- (0,2) node[xi] {};} 
\DeclareSymbol{Xi2}{-2}{\draw (-1,-0.25) node[xi] {} -- (0,1) node[xi] {};} 
\DeclareSymbol{Xi2b}{-2}{\draw (-1,-0.25) node[xic] {} -- (0,1) node[xic] {};} 
\DeclareSymbol{Xi2g}{-2}{\draw (-1,-0.25) node[xies] {} -- (0,1) node[xi] {};} 
\DeclareSymbol{Xi2g2}{-2}{\draw (-1,-0.25) node[xi] {} -- (0,1) node[xies] {};} 
\DeclareSymbol{cXi2}{-2}{\draw (0,-0.25) node[xi] {} -- (-1,1) node[xic] {};}
\DeclareSymbol{Xi3}{0}{\draw (0,0) node[xi] {} -- (-1,1) node[xi] {} -- (0,2) node[xi] {};}
\DeclareSymbol{XiIIXi}{0}{\draw (0,0) node[xi] {} -- (-1,1); \draw[kernels2] (-1,1) node[not] {} -- (0,2) node[xi] {};}

\DeclareSymbol{Xi4}{2}{\draw (0,0) node[xi] {} -- (-1,1) node[xi] {} -- (0,2) node[xi] {} -- (-1,3) node[xi] {};}
\DeclareSymbol{Xi4_1}{2}{\draw (0,0) node[xic] {} -- (-1,1) node[xic] {} -- (0,2) node[xi] {} -- (-1,3) node[xi] {};}
\DeclareSymbol{Xi4_2}{2}{\draw (0,0) node[xic] {} -- (-1,1) node[xi] {} -- (0,2) node[xi] {} -- (-1,3) node[xic] {};}
\DeclareSymbol{Xi2X}{-2}{\draw (0,-0.25) node[xi] {} -- (-1,1) node[xix] {};}
\DeclareSymbol{XXi2}{-2}{\draw (0,-0.25) node[xix] {} -- (-1,1) node[xi] {};}
\DeclareSymbol{IIXi}{0}{\draw (0,-0.25) node[not] {} -- (-1,1) node[xi] {} -- (0,2) node[xi] {};}
\DeclareSymbol{IXi^2}{-1}{\draw (-1,1) node[xi] {} -- (0,0) node[not] {} -- (1,1) node[xi] {};}
\DeclareSymbol{IIXi^2}{-4}{\draw (0,-1.5) node[not] {} -- (0,0);
\draw[kernels2] (-1,1) node[xi] {} -- (0,0) node[not] {} -- (1,1) node[xi] {};}
\DeclareSymbol{XiX}{-2.8}{\node[xibx] {};}
\DeclareSymbol{tauX}{-2.8}{ \node[X] {};}
\DeclareSymbol{Xi}{-2.8}{\node[xib] {};}

\DeclareSymbol{IXiX}{-1}{\draw (0,-0.25) node[not] {} -- (-1,1) node[xix] {};}
\DeclareSymbol{IXi3}{2}{\draw (0,-0.25) node[not] {} -- (-1,1) node[xi] {} -- (0,2) node[xi] {} -- (-1,3) node[xi] {};}
\DeclareSymbol{IXi}{-2}{\draw (0,-0.25) node[not] {} -- (-1,1) node[xi] {};}
\DeclareSymbol{XiI}{-2}{\draw (0,-0.25) node[xi] {} -- (-1,1) node[not] {};}

\DeclareSymbol{Xi4b}{0}{\draw(0,1.5) node[xi] {} -- (0,0); \draw (-1,1) node[xi] {} -- (0,0) node[xi] {} -- (1,1) node[xi] {};}
\DeclareSymbol{Xi4b'}{0}{\draw(0,1.5) node[xi] {} -- (0,-0.2); \draw (-1,1) node[xi] {} -- (0,-0.2) node[not] {} -- (1,1) node[xi] {};}
\DeclareSymbol{Xi4c}{0}{\draw (0,1) -- (0.8,2.2) node[xi] {};\draw (0,-0.25) node[xi] {} -- (0,1) node[xi] {} -- (-0.8,2.2) node[xi] {};}
\DeclareSymbol{Xi4d}{-4.5}{\draw (0,-1.5) node[not] {} -- (0,0); \draw (-1,1) node[xi] {} -- (0,0) node[xi] {} -- (1,1) node[xi] {};}
\DeclareSymbol{Xi4e}{0}{\draw (0,2) node[xi] {} -- (-1,1) node[xi] {} -- (0,0) node[xi] {} -- (1,1) node[xi] {};}
\DeclareSymbol{Xi4e'}{0}{\draw (0,2) node[xi] {} -- (-1,1) node[xi] {} -- (0,-0.2) node[not] {} -- (1,1) node[xi] {};}

\DeclareSymbol{Xitwo}
{0}{\draw[kernels2] (0,0) node[not] {} -- (-1,1) node[not] {}
-- (-2,2) node[not]{} -- (-3,3) node[xi]  {};
\draw[kernels2] (0,0) -- (1,1) node[xi] {};
\draw[kernels2] (-1,1) -- (0,2) node[xi] {};
\draw[kernels2] (-2,2) -- (-1,3) node[xi] {};}

\DeclareSymbol{IXitwo}
{0}{\draw (-.7,1.2) node[xi] {} -- (0,-0.2) -- (.7,1.2) node[xi] {};}
\DeclareSymbol{I1Xitwo}
{0}{\draw[kernels2] (0,0) node[not] {} -- (-1,1) node[xi] {};
\draw[kernels2] (0,0) -- (1,1) node[xi] {};}

\DeclareSymbol{I1Xitwobis}
{0}{\draw[kernels2] (0,0) node[not] {} -- (-1,1) node[xies] {};
\draw[kernels2] (0,0) -- (1,1) node[xies] {};}

\DeclareSymbol{I1Xitwog}
{0}{\draw[kernels2] (0,0) node[not] {} -- (-1,1) node[xies] {};
\draw[kernels2] (0,0) -- (1,1) node[xi] {};}

\DeclareSymbol{cI1Xitwo}
{0}{\draw[kernels2] (0,0) node[not] {} -- (-1,1) node[xic] {};
\draw[kernels2] (0,0) -- (1,1) node[xi] {};}

\DeclareSymbol{I1IXi3}{0}{\draw (0,0) node[xi] {} -- (-1,1) ; 
\draw[kernels2] (-1,1) node[not] {} -- (0,2) node[xi] {};
\draw[kernels2] (-1,1) node[not] {} -- (-2,2) node[xi] {};}

\DeclareSymbol{I1Xi3c}{-1}{\draw[kernels2](0,1.5) node[xi] {} -- (0,0) node[not] {}; \draw (-1,1) node[xi] {} -- (0,0) ; \draw[kernels2] (0,0) -- (1,1) node[xi] {};}

\DeclareSymbol{I1Xi3cbis}{-1}{\draw[kernels2](0,1.5) node[xies] {} -- (0,0) node[not] {}; \draw (-1,1) node[xies] {} -- (0,0) ; \draw[kernels2] (0,0) -- (1,1) node[xies] {};}

\DeclareSymbol{I1IXi3b}{0}{\draw[kernels2] (0,0) node[not] {} -- (-1,1) ; \draw[kernels2] (0,0)   -- (1,1) node[xi] {} ;
\draw (-1,1) node[xi] {} -- (0,2) node[xi] {};
}

\DeclareSymbol{I1IXi3c}{0}{\draw[kernels2] (0,0) node[not] {} -- (-1,1) ; \draw[kernels2] (0,0)   -- (1,1) node[xi] {} ;
\draw[kernels2] (-1,1) node[not] {} -- (0,2) node[xi] {};
\draw[kernels2] (-1,1) node[not] {} -- (-2,2) node[xi] {};}

\DeclareSymbol{I1IXi3cbis}{0}{\draw[kernels2] (0,0) node[not] {} -- (-1,1) ; \draw[kernels2] (0,0)   -- (1,1) node[xies] {} ;
\draw[kernels2] (-1,1) node[not] {} -- (0,2) node[xies] {};
\draw[kernels2] (-1,1) node[not] {} -- (-2,2) node[xies] {};}

\DeclareSymbol{I1Xi}{0}{\draw[kernels2] (0,0) node[not] {} -- (-1,1)  node[xi] {} ;}

\DeclareSymbol{I1Xi4a}{2}{\draw[kernels2] (0,0) node[not] {} -- (-1,1) ; \draw[kernels2] (0,0) node[not] {} -- (1,1) node[xi] {} ;
\draw (-1,1) node[xi] {} -- (0,2) node[xi] {} -- (-1,3) node[xi] {};}

\DeclareSymbol{cI1Xi4a}{2}{\draw[kernels2] (0,0) node[not] {} -- (-1,1) ; \draw[kernels2] (0,0) node[not] {} -- (1,1) node[xic] {} ;
\draw (-1,1) node[xic] {} -- (0,2) node[xi] {} -- (-1,3) node[xi] {};}

\DeclareSymbol{I1Xi4b}{2}{\draw (0,0) node[xi] {} -- (-1,1) node[xi] {} -- (0,2) ; \draw[kernels2] (0,2) node[not] {} -- (-1,3) node[xi] {};\draw[kernels2] (0,2)  -- (1,3) node[xi] {};
}

\DeclareSymbol{cI1Xi4b}{2}{\draw (0,0) node[xic] {} -- (-1,1) node[xic] {} -- (0,2) ; \draw[kernels2] (0,2) node[not] {} -- (-1,3) node[xi] {};\draw[kernels2] (0,2)  -- (1,3) node[xi] {};
}

\DeclareSymbol{I1Xi4c}{2}{\draw (0,0) node[xi] {} -- (-1,1) node[not] {}; \draw[kernels2] (-1,1) -- (0,2) ; 
\draw[kernels2] (-1,1) -- (-2,2) node[xi] {} ;
\draw (0,2) node[xi] {} -- (-1,3) node[xi] {};}

\DeclareSymbol{cI1Xi4c}{2}{\draw (0,0) node[xic] {} -- (-1,1) node[not] {}; \draw[kernels2] (-1,1) -- (0,2) ; 
\draw[kernels2] (-1,1) -- (-2,2) node[xic] {} ;
\draw (0,2) node[xi] {} -- (-1,3) node[xi] {};}

\DeclareSymbol{I1Xi4ab}{2}{\draw[kernels2] (0,0) node[not] {} -- (-1,1) ; \draw[kernels2] (0,0) node[not] {} -- (1,1) node[xi] {};\draw (-1,1) node[xi] {} -- (0,2) ; \draw[kernels2] (0,2) node[not] {} -- (-1,3) node[xi] {};\draw[kernels2] (0,2)  -- (1,3) node[xi] {}; }

\DeclareSymbol{cI1Xi4ab}{2}{\draw[kernels2] (0,0) node[not] {} -- (-1,1) ; \draw[kernels2] (0,0) node[not] {} -- (1,1) node[xic] {};\draw (-1,1) node[xic] {} -- (0,2) ; \draw[kernels2] (0,2) node[not] {} -- (-1,3) node[xi] {};\draw[kernels2] (0,2)  -- (1,3) node[xi] {}; }

\DeclareSymbol{I1Xi4bc}{2}{\draw (0,0) node[xi] {} -- (-1,1) node[not] {}; \draw[kernels2] (-1,1) -- (0,2) ; 
\draw[kernels2] (-1,1) -- (-2,2) node[xi] {} ; \draw[kernels2] (0,2) node[not] {} -- (-1,3) node[xi] {};\draw[kernels2] (0,2)  -- (1,3) node[xi] {};
}

\DeclareSymbol{cI1Xi4bc}{2}{\draw (0,0) node[xic] {} -- (-1,1) node[not] {}; \draw[kernels2] (-1,1) -- (0,2) ; 
\draw[kernels2] (-1,1) -- (-2,2) node[xic] {} ; \draw[kernels2] (0,2) node[not] {} -- (-1,3) node[xi] {};\draw[kernels2] (0,2)  -- (1,3) node[xi] {};
}

\DeclareSymbol{I1Xi4abcc1}{2}{\draw[kernels2] (0,0) node[not] {} -- (-1,1) node[not] {}
-- (-2,2) node[not]{} -- (-3,3) node[xic]  {};
\draw[kernels2] (0,0) -- (1,1) node[xic] {};
\draw[kernels2] (-1,1) -- (0,2) node[xi] {};
\draw[kernels2] (-2,2) -- (-1,3) node[xi] {};
}

\DeclareSymbol{I1Xi4abcc1b}{2}{\draw[kernels2] (0,0) node[not] {} -- (-1,1) node[not] {}
-- (-2,2) node[not]{} -- (-3,3) node[xi]  {};
\draw[kernels2] (0,0) -- (1,1) node[xic] {};
\draw[kernels2] (-1,1) -- (0,2) node[xic] {};
\draw[kernels2] (-2,2) -- (-1,3) node[xi] {};
}

\DeclareSymbol{I1Xi4abcc2}{2}{\draw[kernels2] (0,0) node[not] {} -- (-1,1) node[not] {}
-- (-2,2) node[not]{} -- (-3,3) node[xic]  {};
\draw[kernels2] (0,0) -- (1,1) node[xi] {};
\draw[kernels2] (-1,1) -- (0,2) node[xi] {};
\draw[kernels2] (-2,2) -- (-1,3) node[xic] {};
}

\DeclareSymbol{I1Xi4ac}{2}{\draw[kernels2] (0,0) node[not] {} -- (-1,1) ; \draw[kernels2] (0,0) node[not] {} -- (1,1) node[xi] {}; 
\draw[kernels2] (-1,1) node[not] {} -- (0,2) ; 
\draw[kernels2] (-1,1) -- (-2,2) node[xi] {} ;
\draw (0,2) node[xi] {} -- (-1,3) node[xi] {};}

\DeclareSymbol{cI1Xi4ac}{2}{\draw[kernels2] (0,0) node[not] {} -- (-1,1) ; \draw[kernels2] (0,0) node[not] {} -- (1,1) node[xic] {}; 
\draw[kernels2] (-1,1) node[not] {} -- (0,2) ; 
\draw[kernels2] (-1,1) -- (-2,2) node[xic] {} ;
\draw (0,2) node[xi] {} -- (-1,3) node[xi] {};}

\DeclareSymbol{I1Xi4acc1}{2}{\draw[kernels2] (0,0) node[not] {} -- (-1,1) ; \draw[kernels2] (0,0) node[not] {} -- (1,1) node[xic] {}; 
\draw[kernels2] (-1,1) node[not] {} -- (0,2) ; 
\draw[kernels2] (-1,1) -- (-2,2) node[xi] {} ;
\draw (0,2) node[xic] {} -- (-1,3) node[xi] {};}

\DeclareSymbol{I1Xi4acc2}{2}{\draw[kernels2] (0,0) node[not] {} -- (-1,1) ; \draw[kernels2] (0,0) node[not] {} -- (1,1) node[xic] {}; 
\draw[kernels2] (-1,1) node[not] {} -- (0,2) ; 
\draw[kernels2] (-1,1) -- (-2,2) node[xi] {} ;
\draw (0,2) node[xi] {} -- (-1,3) node[xic] {};}

\DeclareSymbol{2I1Xi4}{2}{\draw[kernels2] (0,0) node[not] {} -- (-1,1) node[not] {};
\draw[kernels2] (0,0) -- (1,1) node[not] {};
\draw[kernels2] (-1,1) -- (-1.5,2.5) node[xi] {};
\draw[kernels2] (-1,1) -- (-0.5,2.5) node[xi] {};
\draw[kernels2] (1,1) -- (0.5,2.5) node[xi] {};
\draw[kernels2] (1,1) -- (1.5,2.5) node[xi] {};
}

\DeclareSymbol{2I1Xi4dis}{2}{\draw[kernels2] (0,0) node[not] {} -- (-1,1) node[not] {};
\draw[kernels2] (0,0) -- (1,1) node[not] {};
\draw[kernels2] (-1,1) -- (-1.5,2.5) node[xies] {};
\draw[kernels2] (-1,1) -- (-0.5,2.5) node[xies] {};
\draw[kernels2] (1,1) -- (0.5,2.5) node[xies] {};
\draw[kernels2] (1,1) -- (1.5,2.5) node[xies] {};
}

\DeclareSymbol{2I1Xi4c1}{2}{\draw[kernels2] (0,0) node[not] {} -- (-1,1) node[not] {};
\draw[kernels2] (0,0) -- (1,1) node[not] {};
\draw[kernels2] (-1,1) -- (-1.5,2.5) node[xic] {};
\draw[kernels2] (-1,1) -- (-0.5,2.5) node[xi] {};
\draw[kernels2] (1,1) -- (0.5,2.5) node[xic] {};
\draw[kernels2] (1,1) -- (1.5,2.5) node[xi] {};
}

\DeclareSymbol{2I1Xi4c2}{2}{\draw[kernels2] (0,0) node[not] {} -- (-1,1) node[not] {};
\draw[kernels2] (0,0) -- (1,1) node[not] {};
\draw[kernels2] (-1,1) -- (-1.5,2.5) node[xic] {};
\draw[kernels2] (-1,1) -- (-0.5,2.5) node[xic] {};
\draw[kernels2] (1,1) -- (0.5,2.5) node[xi] {};
\draw[kernels2] (1,1) -- (1.5,2.5) node[xi] {};
}

\DeclareSymbol{2I1Xi4b}{2}{\draw[kernels2] (0,0) node[not] {} -- (-1,1) ;
\draw[kernels2] (0,0) -- (1,1);
\draw (-1,1) node[xi] {} -- (-1,2.5) node[xi] {};
\draw (1,1)  node[xi] {} -- (1,2.5) node[xi] {};
}

\DeclareSymbol{2I1Xi4bb}{2}{\draw[kernels2] (0,0) node[not] {} -- (-1,1) ;
\draw[kernels2] (0,0) -- (1,1);
\draw (-1,1) node[xi] {} -- (-1,2.5) node[xiesf] {};
\draw (1,1)  node[xi] {} -- (1,2.5) node[xic] {};
}

\DeclareSymbol{2I1Xi4c}{2}{\draw[kernels2] (0,0) node[not] {} -- (-1,1);
\draw[kernels2] (0,0) -- (1,1) node[not] {};
\draw (-1,1)  node[xi] {} -- (-1,2.5) node[xi] {};
\draw[kernels2] (1,1) -- (0.4,2.5) node[xi] {};
\draw[kernels2] (1,1) -- (1.6,2.5) node[xi] {};
}

\DeclareSymbol{2I1Xi4cc1}{2}{\draw[kernels2] (0,0) node[not] {} -- (-1,1);
\draw[kernels2] (0,0) -- (1,1) node[not] {};
\draw (-1,1)  node[xic] {} -- (-1,2.5) node[xi] {};
\draw[kernels2] (1,1) -- (0.4,2.5) node[xic] {};
\draw[kernels2] (1,1) -- (1.6,2.5) node[xi] {};
}

\DeclareSymbol{2I1Xi4cc2}{2}{\draw[kernels2] (0,0) node[not] {} -- (-1,1);
\draw[kernels2] (0,0) -- (1,1) node[not] {};
\draw (-1,1)  node[xic] {} -- (-1,2.5) node[xic] {};
\draw[kernels2] (1,1) -- (0.4,2.5) node[xi] {};
\draw[kernels2] (1,1) -- (1.6,2.5) node[xi] {};
}

\DeclareSymbol{Xi4ba}{0}{\draw(-0.5,1.5) node[xi] {} -- (0,0); \draw (-1.5,1) node[xi] {} -- (0,0) node[not] {}; \draw[kernels2] (0,0) -- (1.5,1) node[xi] {};
\draw[kernels2] (0,0) -- (0.5,1.5) node[xi] {} ;}

\DeclareSymbol{Xi4badis}{0}{\draw(-0.5,1.5) node[xies] {} -- (0,0); \draw (-1.5,1) node[xies] {} -- (0,0) node[not] {}; \draw[kernels2] (0,0) -- (1.5,1) node[xies] {};
\draw[kernels2] (0,0) -- (0.5,1.5) node[xies] {} ;}

\DeclareSymbol{Xi4ba1}{0}{\draw(-0.5,1.5) node[xi] {} -- (0,0); \draw (-1.5,1) node[xi] {} -- (0,0) node[not] {}; \draw[kernels2] (0,0) -- (1.5,1) node[xic] {};
\draw[kernels2] (0,0) -- (0.5,1.5) node[xic] {} ;}

\DeclareSymbol{Xi4ba1b}{0}{\draw(-0.5,1.5) node[xic] {} -- (0,0); \draw (-1.5,1) node[xic] {} -- (0,0) node[not] {}; \draw[kernels2] (0,0) -- (1.5,1) node[xi] {};
\draw[kernels2] (0,0) -- (0.5,1.5) node[xi] {} ;}

\DeclareSymbol{Xi4ba1bdiff}{0}{\draw(-0.5,1.5) node[xic] {} -- (0,0); \draw (-1.5,1) node[xic] {} -- (0,0) node[not] {}; \draw (0,0) -- (1.5,1) node[xi] {};
\draw (0,0) -- (0.5,1.5) node[xi] {};
\draw(0,0) node[diff] {};}

\DeclareSymbol{Xi4ba1bb}{0}{\draw(-0.5,1.5) node[xic] {} -- (0,0); \draw (-1.5,1) node[xiesf] {} -- (0,0) node[not] {}; \draw[kernels2] (0,0) -- (1.5,1) node[xi] {};
\draw[kernels2] (0,0) -- (0.5,1.5) node[xi] {} ;}

\DeclareSymbol{Xi4ba2}{0}{\draw(-0.5,1.5) node[xi] {} -- (0,0); \draw (-1.5,1) node[xic] {} -- (0,0) node[not] {}; \draw[kernels2] (0,0) -- (1.5,1) node[xi] {};
\draw[kernels2] (0,0) -- (0.5,1.5) node[xic] {} ;}

\DeclareSymbol{Xi4ba2b}{0}{\draw(-0.5,1.5) node[xi] {} -- (0,0); \draw (-1.5,1) node[xic] {} -- (0,0) node[not] {}; \draw[kernels2] (0,0) -- (1.5,1) node[xi] {};
\draw[kernels2] (0,0) -- (0.5,1.5) node[xiesf] {} ;}

\DeclareSymbol{Xi4ca}{0}{\draw (0,1) -- (-1,2.2) node[xi] {};\draw (0,-0.25) node[xi] {} -- (0,1) ; \draw[kernels2] (0,1) node[not] {} -- (1,2.2) node[xi] {};
\draw[kernels2] (0,1) {} -- (0,2.7) node[xi] {};
}

\DeclareSymbol{Xi4cb}{0}{\draw (-1,1) -- (-2,2) node[xi] {};\draw[kernels2] (0,0)  -- (-1,1) node[xi] {} ; \draw[kernels2] (0,0) node[not] {} -- (1,1) node[xi] {} ; 
\draw (-1,1) node[xi] {} -- (0,2) node[xi] {};}

\DeclareSymbol{Xi4cbb}{0}{\draw (-1,1) -- (-2,2) node[xiesf] {};\draw[kernels2] (0,0)  -- (-1,1) node[xi] {} ; \draw[kernels2] (0,0) node[not] {} -- (1,1) node[xi] {} ; 
\draw (-1,1) node[xi] {} -- (0,2) node[xic] {};}

\DeclareSymbol{Xi4cbc1}{0}{\draw (-1,1) -- (-2,2) node[xic] {};\draw[kernels2] (0,0)  -- (-1,1) node[xic] {} ; \draw[kernels2] (0,0) node[not] {} -- (1,1) node[xi] {} ; 
\draw (-1,1) node[xic] {} -- (0,2) node[xi] {};}

\DeclareSymbol{Xi4cbc2}{0}{\draw (-1,1) -- (-2,2) node[xi] {};\draw[kernels2] (0,0)  -- (-1,1) node[xi] {} ; \draw[kernels2] (0,0) node[not] {} -- (1,1) node[xic] {} ; 
\draw (-1,1) node[xic] {} -- (0,2) node[xi] {};}

\DeclareSymbol{Xi4cab}{0}{\draw (-1,1) -- (-2,2) node[xi] {};\draw[kernels2] (0,0)  -- (-1,1); \draw[kernels2] (0,0) node[not] {} -- (1,1) node[xi] {} ; 
\draw[kernels2] (-1,1)  {} -- (0,2) node[xi] {};
\draw[kernels2] (-1,1) node[not] {} -- (-1,2.5) node[xi] {};
}

\DeclareSymbol{Xi4cabdis}{0}{\draw (-1,1) -- (-2,2) node[xies] {};\draw[kernels2] (0,0)  -- (-1,1); \draw[kernels2] (0,0) node[not] {} -- (1,1) node[xies] {} ; 
\draw[kernels2] (-1,1)  {} -- (0,2) node[xies] {};
\draw[kernels2] (-1,1) node[not] {} -- (-1,2.5) node[xies] {};
}

\DeclareSymbol{Xi4cabc1}{0}{\draw (-1,1) -- (-2,2) node[xi] {};\draw[kernels2] (0,0)  -- (-1,1); \draw[kernels2] (0,0) node[not] {} -- (1,1) node[xic] {} ; 
\draw[kernels2] (-1,1)  {} -- (0,2) node[xic] {};
\draw[kernels2] (-1,1) node[not] {} -- (-1,2.5) node[xi] {};
}

\DeclareSymbol{Xi4cabc2}{0}{\draw (-1,1) -- (-2,2) node[xic] {};\draw[kernels2] (0,0)  -- (-1,1); \draw[kernels2] (0,0) node[not] {} -- (1,1) node[xic] {} ; 
\draw[kernels2] (-1,1)  {} -- (0,2) node[xi] {};
\draw[kernels2] (-1,1) node[not] {} -- (-1,2.5) node[xi] {};
}

\DeclareSymbol{Xi4ea}{1.5}{\draw (-1,2.5) node[xi] {} -- (-1,1) node[xi] {} -- (0,0); 
 \draw[kernels2] (0,0)  -- (1,1) node[xi] {};
\draw[kernels2] (0,0) node[not] {} -- (0,1.5) node[xi] {}; }

\DeclareSymbol{Xi4eac1}{1.5}{\draw (-1,2.5) node[xic] {} -- (-1,1) node[xi] {} -- (0,0); 
 \draw[kernels2] (0,0)  -- (1,1) node[xic] {};
\draw[kernels2] (0,0) node[not] {} -- (0,1.5) node[xi] {}; }

\DeclareSymbol{Xi4eac1b}{1.5}{\draw (-1,2.5) node[xic] {} -- (-1,1) node[xi] {} -- (0,0); 
 \draw[kernels2] (0,0)  -- (1,1) node[xiesf] {};
\draw[kernels2] (0,0) node[not] {} -- (0,1.5) node[xi] {}; }

\DeclareSymbol{Xi4eac2}{1.5}{\draw (-1,2.5) node[xic] {} -- (-1,1) node[xic] {} -- (0,0); 
 \draw[kernels2] (0,0)  -- (1,1) node[xi] {};
\draw[kernels2] (0,0) node[not] {} -- (0,1.5) node[xi] {}; }

\DeclareSymbol{Xi4eact1}{1.5}{\draw (-1,2.5) node[xic] {} -- (-1,1) node[xi] {} -- (0,0); 
 \draw (0,0)  -- (1,1) node[xic] {};
\draw[rho] (0,0) node[not] {} -- (0,1.5) node[xi] {}; }

\DeclareSymbol{Xi4eact2}{1.5}{\draw[rho] (-1,2.5) node[xic] {} -- (-1,1) node[xi] {} -- (0,0); 
 \draw (0,0)  -- (1,1) node[xic] {};
\draw (0,0) node[not] {} -- (0,1.5) node[xi] {}; }

\DeclareSymbol{Xi4eabis}{1.5}{\draw (-1,2.5) node[xi] {} -- (-1,1) ; \draw[kernels2] (-1,1) node[xi] {} -- (0,0); 
 \draw (0,0)  -- (1,1) node[xi] {};
\draw[kernels2] (0,0) node[not] {} -- (0,1.5) node[xi] {}; }

\DeclareSymbol{Xi4eabisc1}{1.5}{\draw (-1,2.5) node[xic] {} -- (-1,1) ; \draw[kernels2] (-1,1) node[xi] {} -- (0,0); 
 \draw (0,0)  -- (1,1) node[xi] {};
\draw[kernels2] (0,0) node[not] {} -- (0,1.5) node[xic] {}; }

\DeclareSymbol{Xi4eabisc1b}{1.5}{\draw (-1,2.5) node[xic] {} -- (-1,1) ; \draw[kernels2] (-1,1) node[xi] {} -- (0,0); 
 \draw (0,0)  -- (1,1) node[xi] {};
\draw[kernels2] (0,0) node[not] {} -- (0,1.5) node[xiesf] {}; }

\DeclareSymbol{Xi4eabisc1bis}{1.5}{\draw (-1,2.5) node[xi] {} -- (-1,1) ; \draw[kernels2] (-1,1) node[xi] {} -- (0,0); 
 \draw (0,0)  -- (1,1) node[xi] {};
\draw[kernels2] (0,0) node[not] {} -- (0,1.5) node[xi] {};
\draw (-2,1) node[] {\tiny{$i$}};
\draw (-2,2.5) node[] {\tiny{$\ell$}};
\draw (2,1) node[] {\tiny{$k$}};
\draw (0,2.5) node[] {\tiny{$j$}};
 }

\DeclareSymbol{Xi4eabisc1tris}{1.5}{\draw (-1,2.5) node[xi] {} -- (-1,1) ; \draw[kernels2] (-1,1) node[xi] {} -- (0,0); 
 \draw (0,0)  -- (1,1) node[xi] {};
\draw[kernels2] (0,0) node[not] {} -- (0,1.5) node[xi] {};
\draw (-2,1) node[] {\tiny{i}};
\draw (-2,2.5) node[] {\tiny{j}};
\draw (2,1) node[] {\tiny{j}};
\draw (0,2.5) node[] {\tiny{i}};
 }

\DeclareSymbol{Xi4eabisc1quater}{1.5}{\draw (-1,2.5) node[xic] {} -- (-1,1) ; \draw[kernels2] (-1,1) node[xi] {} -- (0,0); 
 \draw (0,0)  -- (1,1) node[xic] {};
\draw[kernels2] (0,0) node[not] {} -- (0,1.5) node[xi] {};
 }

\DeclareSymbol{Xi4eabisc2}{1.5}{\draw (-1,2.5) node[xic] {} -- (-1,1) ; \draw[kernels2] (-1,1) node[xi] {} -- (0,0); 
 \draw (0,0)  -- (1,1) node[xic] {};
\draw[kernels2] (0,0) node[not] {} -- (0,1.5) node[xi] {}; }

\DeclareSymbol{Xi4eabisc2l}{1.5}{\draw (-1,2.5) node[xiesf] {} -- (-1,1) ; \draw[kernels2] (-1,1) node[xi] {} -- (0,0); 
 \draw (0,0)  -- (1,1) node[xic] {};
\draw[kernels2] (0,0) node[not] {} -- (0,1.5) node[xi] {}; }

\DeclareSymbol{Xi4eabisc2r}{1.5}{\draw (-1,2.5) node[xic] {} -- (-1,1) ; \draw[kernels2] (-1,1) node[xi] {} -- (0,0); 
 \draw (0,0)  -- (1,1) node[xiesf] {};
\draw[kernels2] (0,0) node[not] {} -- (0,1.5) node[xi] {}; }

\DeclareSymbol{Xi4eabisc3}{1.5}{\draw (-1,2.5) node[xic] {} -- (-1,1) ; \draw[kernels2] (-1,1) node[xic] {} -- (0,0); 
 \draw (0,0)  -- (1,1) node[xi] {};
\draw[kernels2] (0,0) node[not] {} -- (0,1.5) node[xi] {}; }

\DeclareSymbol{Xi4eb}{0}{
\draw[kernels2] (0,2) node[xi] {} -- (-1,1) ; \draw[kernels2] (-2,2)  node[xi] {} -- (-1,1) ; \draw (-1,1)  node[not] {} -- (0,0); 
 \draw (0,0) node[xi] {}  -- (1,1) node[xi] {};
}

\DeclareSymbol{Xi4eab}{1.5}{\draw[kernels2] (-1,2.5) node[xi] {} -- (-1,1) ; \draw[kernels2] (-2,2)  node[xi] {} -- (-1,1) ; \draw (-1,1)  node[not] {} -- (0,0); 
 \draw[kernels2] (0,0)  -- (1,1) node[xi] {};
\draw[kernels2] (0,0) node[not] {} -- (0,1.5) node[xi] {}; 
}

\DeclareSymbol{Xi4eabdis}{1.5}{\draw[kernels2] (-1,2.5) node[xies] {} -- (-1,1) ; \draw[kernels2] (-2,2)  node[xies] {} -- (-1,1) ; \draw (-1,1)  node[not] {} -- (0,0); 
 \draw[kernels2] (0,0)  -- (1,1) node[xies] {};
\draw[kernels2] (0,0) node[not] {} -- (0,1.5) node[xies] {}; 
}

\DeclareSymbol{Xi4eabc1}{1.5}{\draw[kernels2] (-1,2.5) node[xic] {} -- (-1,1) ; \draw[kernels2] (-2,2)  node[xi] {} -- (-1,1) ; \draw (-1,1)  node[not] {} -- (0,0); 
 \draw[kernels2] (0,0)  -- (1,1) node[xic] {};
\draw[kernels2] (0,0) node[not] {} -- (0,1.5) node[xi] {}; 
}

\DeclareSymbol{Xi4eabc2}{1.5}{\draw[kernels2] (-1,2.5) node[xi] {} -- (-1,1) ; \draw[kernels2] (-2,2)  node[xi] {} -- (-1,1) ; \draw (-1,1)  node[not] {} -- (0,0); 
 \draw[kernels2] (0,0)  -- (1,1) node[xic] {};
\draw[kernels2] (0,0) node[not] {} -- (0,1.5) node[xic] {}; 
}

\DeclareSymbol{Xi4eabbis}{1.5}{\draw[kernels2] (-1,2.5) node[xi] {} -- (-1,1) ; \draw[kernels2] (-2,2)  node[xi] {} -- (-1,1) ; \draw[kernels2] (-1,1)  node[not] {} -- (0,0); 
 \draw (0,0)  -- (1,1) node[xi] {};
\draw[kernels2] (0,0) node[not] {} -- (0,1.5) node[xi] {}; 
}

\DeclareSymbol{Xi4eabbisc1}{1.5}{\draw[kernels2] (-1,2.5) node[xic] {} -- (-1,1) ; \draw[kernels2] (-2,2)  node[xi] {} -- (-1,1) ; \draw[kernels2] (-1,1)  node[not] {} -- (0,0); 
 \draw (0,0)  -- (1,1) node[xic] {};
\draw[kernels2] (0,0) node[not] {} -- (0,1.5) node[xi] {}; 
}

\DeclareSymbol{Xi4eabbisc1perm}{1.5}{\draw[kernels2] (-1,2.5) node[xi] {} -- (-1,1) ; \draw[kernels2] (-2,2)  node[xic] {} -- (-1,1) ; \draw[kernels2] (-1,1)  node[not] {} -- (0,0); 
 \draw (0,0)  -- (1,1) node[xic] {};
\draw[kernels2] (0,0) node[not] {} -- (0,1.5) node[xi] {}; 
}

\DeclareSymbol{Xi4eabbisc2}{1.5}{\draw[kernels2] (-1,2.5) node[xi] {} -- (-1,1) ; \draw[kernels2] (-2,2)  node[xi] {} -- (-1,1) ; \draw[kernels2] (-1,1)  node[not] {} -- (0,0); 
 \draw (0,0)  -- (1,1) node[xic] {};
\draw[kernels2] (0,0) node[not] {} -- (0,1.5) node[xic] {}; 
}

\DeclareSymbol{Xi2cbis}{0}{\draw[kernels2] (0,1) -- (0.8,2.2) node[xi] {};\draw[kernels2] (0,-0.25) node[not] {} -- (0,1); \draw[kernels2] (0,1) node[not] {} -- (-0.8,2.2) node[xi] {};}

\DeclareSymbol{Xi2cbis1}{0}{\draw (0,1) -- (-0.8,2.2) node[xi] {};\draw[kernels2] (0,-0.25) node[not] {} -- (0,1) node[xi] {}; }

\DeclareSymbol{Xi2Xbis}{-2}{\draw[kernels2] (0,-0.25)  -- (-1,1) ; \draw (-1,1) node[xix] {};
\draw[kernels2] (0,-0.25) node[not] {} -- (1,1) node[xi] {};}

\DeclareSymbol{XXi2bis}{-2}{\draw[kernels2] (0,-0.25) -- (-1,1) node[xi] {};
\draw[kernels2] (0,-0.25) node[X] {} -- (1,1) node[xi] {};}

\DeclareSymbol{I1XiIXi}{0}{\draw[kernels2] (0,-0.25) -- (1,1) node[xi] {};
\draw (0,-0.25) node[not] {} -- (-1,1) node[xi] {};}

\DeclareSymbol{I1XiIXib}{0}{\draw  (0,-0.25) node[xi] {} -- (0,1) node[not] {};
\draw[kernels2] (0,1) -- (0,2.25) ; \draw (0,2.25) node[xi]{}; }

\DeclareSymbol{I1XiIXic}{0}{
\draw[kernels2] (0,0) -- (1,1) node[xi] {} ; 
\draw[kernels2] (0,0) node[not] {}  -- (-1,1) node[not] {} -- (0,2) node[xi] {};
}

\DeclareSymbol{thin}{1.4}{\draw[pagebackground] (-0.3,0) -- (0.3,0); \draw  (0,0) -- (0,2);}
\DeclareSymbol{thin2}{1.4}{\draw[pagebackground] (-0.3,0) -- (0.3,0); \draw[tinydots]  (0,0) -- (0,2);}

\DeclareSymbol{thick}{1.4}{\draw[pagebackground] (-0.3,0) -- (0.3,0); \draw[kernels2]  (0,0) -- (0,2);}

\DeclareSymbol{thick2}{1.4}{\draw[pagebackground] (-0.3,0) -- (0.3,0); \draw[kernels2,tinydots]  (0,0) -- (0,2);}

\DeclareSymbol{Xi4ind}{2}{\draw (0,0) node[xi,label={[label distance=-0.2em]right: \scriptsize  $ i $}]  { } -- (-1,1) node[xi,label={[label distance=-0.2em]left: \scriptsize  $ j $}] {} -- (0,2) node[xi,label={[label distance=-0.2em]right: \scriptsize  $ k $}] {} -- (-1,3) node[xi,label={[label distance=-0.2em]left: \scriptsize  $ \ell $}] {};}

\DeclareSymbol{Xi4c1}{2}{\draw (0,0) node[xic] {} -- (-1,1) node[xi] {} -- (0,2) node[xic] {} -- (-1,3) node[xi] {};} 
\DeclareSymbol{IXi2ex}{0}{\draw (0,-0.25) node[xie] {} -- (-1,1) node[xi] {} ; \draw (0,-0.25)-- (1,1) node[xi] {};}
\DeclareSymbol{IXi2ex1}{0}{\draw (0,-0.25) node[xie] {} -- (-1,1) node[xi] {} -- (0,2) node[xi] {};}

\DeclareSymbol{Xi4b1}{0}{\draw(0,1.5) node[xic] {} -- (0,0); \draw (-1,1) node[xic] {} -- (0,0) node[xi] {} -- (1,1) node[xi] {};}

\DeclareSymbol{Xi4ec1}{0}{\draw (0,2) node[xi] {} -- (-1,1) node[xic] {} -- (0,0) node[xic] {} -- (1,1) node[xi] {};}
\DeclareSymbol{Xi4ec2}{0}{\draw (0,2) node[xic] {} -- (-1,1) node[xi] {} -- (0,0) node[xic] {} -- (1,1) node[xi] {};}
\DeclareSymbol{Xi4ec3}{0}{\draw (0,2) node[xic] {} -- (-1,1) node[xic] {} -- (0,0) node[xi] {} -- (1,1) node[xi] {};}

\DeclareSymbol{I1Xi4ac1}{2}{\draw[kernels2] (0,0) node[not] {} -- (-1,1) ; \draw[kernels2] (0,0) node[not] {} -- (1,1) node[xic] {} ;
\draw (-1,1) node[xi] {} -- (0,2) node[xic] {} -- (-1,3) node[xi] {};}

\DeclareSymbol{I1Xi4ac2}{2}{\draw[kernels2] (0,0) node[not] {} -- (-1,1) ; \draw[kernels2] (0,0) node[not] {} -- (1,1) node[xic] {} ;
\draw (-1,1) node[xi] {} -- (0,2) node[xi] {} -- (-1,3) node[xic] {};}

\DeclareSymbol{I1Xi4bp}{2}{\draw (0,0) node[not] {} -- (-1,1) node[xi] {} -- (0,2) ; \draw[kernels2] (0,2) node[not] {} -- (-1,3) node[xi] {};\draw[kernels2] (0,2)  -- (1,3) node[xi] {};
}

\DeclareSymbol{I1Xi4bc1}{2}{\draw (0,0) node[xic] {} -- (-1,1) node[xi] {} -- (0,2) ; \draw[kernels2] (0,2) node[not] {} -- (-1,3) node[xi] {};\draw[kernels2] (0,2)  -- (1,3) node[xic] {};
}

\DeclareSymbol{I1Xi4bc2}{2}{\draw (0,0) node[xic] {} -- (-1,1) node[xi] {} -- (0,2) ; \draw[kernels2] (0,2) node[not] {} -- (-1,3) node[xic] {};\draw[kernels2] (0,2)  -- (1,3) node[xi] {};
}

\DeclareSymbol{I1Xi4cp}{2}{\draw (0,0) node[not] {} -- (-1,1) node[not] {}; \draw[kernels2] (-1,1) -- (0,2) ; 
\draw[kernels2] (-1,1) -- (-2,2) node[xi] {} ;
\draw (0,2) node[xi] {} -- (-1,3) node[xi] {};}

\DeclareSymbol{I1Xi4cc1}{2}{\draw (0,0) node[xic] {} -- (-1,1) node[not] {}; \draw[kernels2] (-1,1) -- (0,2) ; 
\draw[kernels2] (-1,1) -- (-2,2) node[xi] {} ;
\draw (0,2) node[xic] {} -- (-1,3) node[xi] {};}

\DeclareSymbol{I1Xi4cc2}{2}{\draw (0,0) node[xic] {} -- (-1,1) node[not] {}; \draw[kernels2] (-1,1) -- (0,2) ; 
\draw[kernels2] (-1,1) -- (-2,2) node[xi] {} ;
\draw (0,2) node[xi] {} -- (-1,3) node[xic] {};}

\DeclareSymbol{I1Xi4abc1}{2}{\draw[kernels2] (0,0) node[not] {} -- (-1,1) ; \draw[kernels2] (0,0) node[not] {} -- (1,1) node[xic] {};\draw (-1,1) node[xi] {} -- (0,2) ; \draw[kernels2] (0,2) node[not] {} -- (-1,3) node[xic] {};\draw[kernels2] (0,2)  -- (1,3) node[xi] {}; }

\DeclareSymbol{I1Xi4abc2}{2}{\draw[kernels2] (0,0) node[not] {} -- (-1,1) ; \draw[kernels2] (0,0) node[not] {} -- (1,1) node[xic] {};\draw (-1,1) node[xi] {} -- (0,2) ; \draw[kernels2] (0,2) node[not] {} -- (-1,3) node[xi] {};\draw[kernels2] (0,2)  -- (1,3) node[xic] {}; }

\DeclareSymbol{R1}{0}{\draw (-1,1) node[xi] {} -- (0,0) node[not] {};
\draw[kernels2] (0,1.5) node[xic] {} -- (0,0) -- (1,1) node[xic] {};}
\DeclareSymbol{R2}{0}{\draw (-1,1) node[xic] {} -- (0,0) node[not] {};
\draw[kernels2] (0,1.5)  {} -- (0,0) -- (1,1)  {};
\draw (0,1.5) node[xi] {};
\draw (1,1) node[xic] {};
}
\DeclareSymbol{R3}{1}{\draw[kernels2] (-1,1.5)  {} -- (0,0) node[not] {} -- (1,1.5);
\draw (-1,1.5) node[xi] {};
\draw[kernels2] (0,3) {} -- (1,1.5) -- (2,3)  {};
\draw  (0,3) node[xic] {} ;
\draw (2,3) node[xic] {};}
\DeclareSymbol{R4}{1}{\draw[kernels2] (-1,1.5) node[xic] {} -- (0,0) node[not] {} -- (1,1.5);
\draw[kernels2] (0,3) {} -- (1,1.5) -- (2,3) node[xic] {};
\draw (0,3) node[xi] {};}

\DeclareSymbol{I1Xi4bcp}{2}{\draw (0,0) node[not] {} -- (-1,1) node[not] {}; \draw[kernels2] (-1,1) -- (0,2) ; 
\draw[kernels2] (-1,1) -- (-2,2) node[xi] {} ; \draw[kernels2] (0,2) node[not] {} -- (-1,3) node[xi] {};\draw[kernels2] (0,2)  -- (1,3) node[xi] {};
}

\DeclareSymbol{I1Xi4bcc1}{2}{\draw (0,0) node[xic] {} -- (-1,1) node[not] {}; \draw[kernels2] (-1,1) -- (0,2) ; 
\draw[kernels2] (-1,1) -- (-2,2) node[xi] {} ; \draw[kernels2] (0,2) node[not] {} -- (-1,3) node[xi] {};\draw[kernels2] (0,2)  -- (1,3) node[xic] {};
}

\DeclareSymbol{I1Xi4bcc2}{2}{\draw (0,0) node[xic] {} -- (-1,1) node[not] {}; \draw[kernels2] (-1,1) -- (0,2) ; 
\draw[kernels2] (-1,1) -- (-2,2) node[xi] {} ; \draw[kernels2] (0,2) node[not] {} -- (-1,3) node[xic] {};\draw[kernels2] (0,2)  -- (1,3) node[xi] {};
} 

\DeclareSymbol{2I1Xi4bc1}{2}{\draw[kernels2] (0,0) node[not] {} -- (-1,1) ;
\draw[kernels2] (0,0) -- (1,1);
\draw (-1,1) node[xic] {} -- (-1,2.5) node[xi] {};
\draw (1,1)  node[xic] {} -- (1,2.5) node[xi] {};
}

\DeclareSymbol{2I1Xi4bc2}{2}{\draw[kernels2] (0,0) node[not] {} -- (-1,1) ;
\draw[kernels2] (0,0) -- (1,1);
\draw (-1,1) node[xi] {} -- (-1,2.5) node[xic] {};
\draw (1,1)  node[xic] {} -- (1,2.5) node[xi] {};
}

\DeclareSymbol{diff2I1Xi4bc2}{2}{\draw (0,0) node[diff] {} -- (-1,1) ;
\draw (0,0) -- (1,1);
\draw (-1,1) node[xi] {} -- (-1,2.5) node[xic] {};
\draw (1,1)  node[xic] {} -- (1,2.5) node[xi] {};
}

\DeclareSymbol{2I1Xi4bc3}{2}{\draw[kernels2] (0,0) node[not] {} -- (-1,1) ;
\draw[kernels2] (0,0) -- (1,1);
\draw (-1,1) node[xic] {} -- (-1,2.5) node[xic] {};
\draw (1,1)  node[xi] {} -- (1,2.5) node[xi] {};
}

\DeclareSymbol{Xi41}{0}{\draw (0,1) -- (0.8,2.2) node[xic] {};\draw (0,-0.25) node[xi] {} -- (0,1) node[xi] {} -- (-0.8,2.2) node[xic] {};} 

\DeclareSymbol{Xi42}{0}{\draw (0,1) -- (0.8,2.2) node[xi] {};\draw (0,-0.25) node[xic] {} -- (0,1) node[xi] {} -- (-0.8,2.2) node[xic] {};}

\DeclareSymbol{Xi4ca1}{0}{\draw (0,1) -- (-1,2.2) node[xic] {};\draw (0,-0.25) node[xi] {} -- (0,1) ; \draw[kernels2] (0,1) node[not] {} -- (1,2.2) node[xic] {};
\draw[kernels2] (0,1) {} -- (0,2.7) node[xi] {};
}

\DeclareSymbol{Xi4ca2}{0}{\draw (0,1) -- (-1,2.2) node[xi] {};\draw (0,-0.25) node[xi] {} -- (0,1) ; \draw[kernels2] (0,1) node[not] {} -- (1,2.2) node[xic] {};
\draw[kernels2] (0,1) {} -- (0,2.7) node[xic] {};
}

\DeclareSymbol{Xi4cap}{0}{\draw (0,1) -- (-1,2.2) node[xi] {};\draw (0,-0.25) node[not] {} -- (0,1) ; \draw[kernels2] (0,1) node[not] {} -- (1,2.2) node[xi] {};
\draw[kernels2] (0,1) {} -- (0,2.7) node[xi] {};
}

\DeclareSymbol{Xi3a}{0}{
 \draw (-1,1)  node[xi] {} -- (0,0); 
 \draw (0,0) node[xi] {}  -- (1,1) node[xi] {};
}

\DeclareSymbol{Xi4ebc1}{0}{
\draw[kernels2] (0,2) node[xi] {} -- (-1,1) ; \draw[kernels2] (-2,2)  node[xic] {} -- (-1,1) ; \draw (-1,1)  node[not] {} -- (0,0); 
 \draw (0,0) node[xic] {}  -- (1,1) node[xi] {};
}

\DeclareSymbol{Xi4ebc2}{0}{
\draw[kernels2] (0,2) node[xi] {} -- (-1,1) ; \draw[kernels2] (-2,2)  node[xi] {} -- (-1,1) ; \draw (-1,1)  node[not] {} -- (0,0); 
 \draw (0,0) node[xic] {}  -- (1,1) node[xic] {};
}

\DeclareSymbol{Xi2cbispex}{0}{\draw[kernels2] (0,1) -- (0.8,2.2) node[xi] {};\draw (0,-0.25) node[xie] {} -- (0,1); \draw[kernels2] (0,1) node[not] {} -- (-0.8,2.2) node[xi] {};}

\DeclareSymbol{Xi2cbis1p}{0}{\draw (0,1) -- (-0.8,2.2) node[xi] {};\draw (0,-0.25) node[not] {} -- (0,1) node[xi] {}; }

\DeclareSymbol{Xi2Xp}{-2}{\draw (0,-0.25) node[not] {} -- (-1,1) node[xix] {};} 

\DeclareSymbol{I1XiIXib}{0}{\draw  (0,-0.25) node[xi] {} -- (0,1) node[not] {};
\draw[kernels2] (0,1) -- (0,2.25) ; \draw (0,2.25) node[xi]{}; }

\DeclareSymbol{IXi2b}{0}{\draw  (0,-0.25) node[xi] {} -- (0,1) node[not] {};
\draw (0,1) -- (0,2.25) ; \draw (0,2.25) node[xi]{}; }

\DeclareSymbol{IXi2bex}{0}{\draw  (0,-0.25) node[xi] {} -- (0,1) node[xie] {};
\draw (0,1) -- (0,2.25) ; \draw (0,2.25) node[xi]{}; }

 \def\1{\mathbf{\symbol{1}}}

\def\one{\mathbf{1}}

\DeclareSymbol{diff}{0}{
\draw (0,0.5) node[diff] {};
}

\DeclareSymbol{diff1}{0}{
\draw (0,0.5) node[diff1] {};
}

\DeclareSymbol{diff2}{0}{
\draw (0,0.5) node[diff2] {};
}

\DeclareSymbol{geo}{0}{
\draw (0,0) node[diff] {};
\draw (0.3,0) node[diff] {};
}

\DeclareSymbol{generic}{0}{
\draw (0,0.6) node[xi] {};
}

\DeclareSymbol{g}{0}{
\draw (0,0.6) node[g] {};
}

\DeclareSymbol{Ito}{0}{
\draw (0,0.6) node[xies] {};
}

\DeclareSymbol{Itob}{0}{
\draw (0,0.6) node[xiesf] {};
}

\DeclareSymbol{greycirc}{0}{
\draw (0,0.3) node[xi] {};
}

\DeclareSymbol{not}{0}{
\draw (0,0.6) node[not] {};
\draw[tinydots] (0,0.6) circle (0.8);
}

\DeclareSymbol{genericb}{0}{
\draw (0,0.6) node[xic] {};
}

\DeclareSymbol{bluecirc}{0}{
\draw (0,0.3) node[xic] {};
}

\DeclareSymbol{genericxix}{0}{
\draw (0,0.6) node[xix] {};
}

\DeclareSymbol{genericX}{0}{
\draw (0,0.6) node[X] {};
}

\DeclareSymbol{diffIto}{1}{
\draw  (0,2.5) -- (0,0) ;
\draw (0,-0.1) node[diff] {};
\draw (0,2.5) node[xies] {};
}
\DeclareSymbol{Itodiff}{2}{
\draw(0,2.9) -- (0,-0.2);
\draw (0,2.9) node[diff] {};
\draw (0,-0.1) node[xies] {};
}

\DeclareSymbol{diffgeneric}{1}{
\draw  (0,2.5) -- (0,0) ;
\draw (0,-0.1) node[diff] {};
\draw (0,2.5) node[xi] {};
}

\DeclareSymbol{genericdiff}{2}{
\draw(0,2.9) -- (0,-0.2);
\draw (0,2.9) node[diff] {};
\draw (0,-0.1) node[xi] {};
}

\DeclareSymbol{diffdot}{2}{
\draw  (0,3) -- (0,-0.1) ;
\draw (0,3) node[not] {};
\draw (0,-0.1) node[diff] {};
}

\DeclareSymbol{diffdotmini}{0}{
\draw  (0,0) -- (0,1.2) ;
\draw (0,1.2) node[not] {};
\draw (0,0) node[diffmini] {};
}

\DeclareSymbol{dotdiff}{2}{
\draw[kernelsmod]  (0,3) -- (0,-0.1) ;
\draw (0,3) node[diff] {};
\draw (0,-0.1) node[not] {};
}

\DeclareSymbol{dotdiff1}{2}{
\draw[kernelsmod]  (0,3) -- (0,-0.1) ;
\draw (0,3) node[diff1] {};
\draw (0,-0.1) node[not] {};
}

\DeclareSymbol{dotdiff1mini}{0}{
\draw[kernelsmod]  (0,1.2) -- (0,0) ;
\draw (0,1.2) node[diffmini] {};
\draw (0,0) node[not] {};
}

\DeclareSymbol{dotdiff2}{2}{
\draw (0,3) -- (0,-0.1) ;
\draw (0,3) node[diff] {};
\draw (0,-0.1) node[not] {};
}

\DeclareSymbol{dotdiff2mini}{0}{
\draw (0,1.2) -- (0,0) ;
\draw (0,1.2) node[diffmini] {};
\draw (0,0) node[not] {};
}

\DeclareSymbol{dotdiffstraight}{0}{
\draw  (0,3) -- (0,-0.1) ;
\draw (0,3) node[diff] {};
\draw (0,-0.1) node[not] {};
}

\DeclareSymbol{arbre1}{0}{
\draw  (0,0) -- (1.5,1.5) ;
\draw (1.5,1.5) node[not] {};
\draw (0,0) node[not] {};
}

\DeclareSymbol{arbre2}{0}{
\draw  (0,0) -- (1.5,1.5) ;
\draw[kernelsmod] (0,0) -- (-1.5,1.5);
\draw (1.5,1.5) node[not] {};
\draw (0,0) node[not] {};
\draw (-1.5,1.5) node[xi] {};
}

\DeclareSymbol{arbre3}{0}{
\draw  (0,0) -- (1.5,1.5) ;
\draw[kernelsmod] (1.5,1.5) -- (0,3);
\draw (0,0) node[not] {};
\draw (1.5,1.5) node[not] {};
\draw (0,3) node[xi] {};
}

\DeclareSymbol{treeeval}{0}{
\draw (0,0) -- (1,1);
\draw (0,0) node[xi] {};
\draw (1.25,1.25) node[xi] {};
\draw (-0.6,0.6) node[]{\tiny{$i$}};
\draw (0.65,1.85) node[]{\tiny{$j$}};
}

\DeclareSymbol{testeval}{0}{
\draw (0,0) -- (1,1);
\draw (0,0) -- (-1,1);
\draw (0,0) node[xi] {};
\draw (1.25,1.25) node[xi] {};
\draw (-1.25,1.25) node[xi] {};
\draw (-0.6,-0.6) node[]{\tiny{$i$}};
\draw (0.65,1.85) node[]{\tiny{$j$}};
\draw (-1.95,1.85) node[]{\tiny{$k$}};
}

\DeclareSymbol{treeeval2}{0}{
\draw[kernelsmod] (-0.25,-1) -- (1,0.5) ;
\draw[kernelsmod] (1,0.5) -- (-0.25,2);
\draw (1,0.5) node[diff2] {};
\draw (-0.25,-1) node[not] {};
\draw (-0.25,2) node[xi] {};
\draw (-0.6,1.2) node[]{\tiny{1}};
}

\DeclareSymbol{arbreact}{1}{
\draw (0,0) node[not] {};
\draw[kernelsmod] (0,0) -- (1,1);
\draw[kernelsmod] (0,0) -- (-1,1);
\draw (-1,1) node[xic] {};
\draw  (0,2) -- (1,1) ;
\draw (0,2) node[xic] {};
\draw (1,1) node[xi] {};
}

\DeclareSymbol{arbreact1}{0}{
\draw (0,-1.5) -- (0,0);
\draw[kernelsmod] (0,0) -- (1,1);
\draw[kernelsmod] (0,0) -- (-1,1);
\draw  (0,2) -- (1,1) ;
\draw (0,-1.5) node[diff] {};
\draw (0,0) node[not] {};
\draw (-1,1) node[xic] {};
\draw (0,2) node[xic] {};
\draw (1,1) node[xi] {};
}

\DeclareSymbol{arbreact2}{0}{
\draw (0,-0.75) -- (-1,0.5); 
\draw (0,-0.75) -- (1,0.5);
\draw (0,1.5) -- (1,0.5);
\draw (0,1.5) node[xic] {};
\draw (1,0.5) node[xi] {};
\draw (-1,0.5) node[xic] {};
\draw (0,-0.75) node[diff] {};
}

\DeclareSymbol{arbreact3}{0}{
\draw[kernelsmod] (0,-0.75) -- (-1,0.5); 
\draw[kernelsmod] (0,-0.75) -- (1,0.5);
\draw (0,1.75) -- (1,0.5);
\draw (2,1.75) -- (1,0.5);
\draw (0,1.75) node[xic] {};
\draw (1,0.5) node[diff] {};
\draw (-1,0.5) node[xic] {};
\draw (2,1.75) node[xi] {};
\draw (0,-0.75) node[not] {};
}

\DeclareSymbol{pre_im_I1Xitwo}{0}{
\draw[kernels2] (0,-0.3) node[not] {} -- (-0.6,0.7) ;
\draw[kernels2] (0,-0.3) -- (0.6,0.7);
\draw (0,0.9) node[g] {};
}

\DeclareSymbol{pre_im_cI1Xi4ab}{2}{
\draw[kernels2] (0,-1) node[not] {} -- (-0.6,0) ;
\draw[kernels2] (0,-1) -- (0.6,0);
\draw (0,0.2) node[g] {};
\draw (0,0.6) -- (0,1.5);
\draw[kernels2] (0,1.5) node[not] {} -- (-0.6,2.5) ;
\draw[kernels2] (0,1.5) -- (0.6,2.5);
\draw (0,2.7) node[g] {};
}

\DeclareSymbol{pre_im_I1Xi4acc2}{0}{
\draw[kernels2] (-1,-0.5) node[not] {} -- (-1.6,0.5) ;
\draw[kernels2] (-1,-0.5) -- (-0.4,0.5);
\draw[kernels2] (-1,-0.5) -- (0.2,-1.5) node[not] {} ;
\draw (-1,1.1) -- (-1,2);
\draw[kernels2] (0.2,-1.5) -- (0.2,2);
\draw (-1,0.7) node[g] {};
\draw (-0.3,2.2) node[g] {};
}

\DeclareSymbol{pre_im_I1Xi4abcc2}{2}{
\draw[kernels2] (0,-1) node[not] {} -- (-1,0) node[not] {};
\draw[kernels2] (-1,1.2) node[not] {} -- (-1,0);
\draw[kernels2] (-1,1.2) -- (-1.5,2.5);
\draw[kernels2] (-1,1.2) -- (-0.5,2.5);
\draw[kernels2] (-1,0) -- (0.7,2.5);
\draw[kernels2] (0,-1) -- (1.5,2.5);
\draw (-1,2.7) node[g] {};
\draw (1,2.7) node[g] {};
}

\DeclareSymbol{pre_im_2I1Xi4c1}{2}{
\draw[kernels2] (0,-0.5) node[not] {} -- (-1,0.5) node[not] {};
\draw[kernels2] (0,-0.5) -- (1,0.5) node[not] {};
\draw[kernels2] (-1,0.5) node[not] {}-- (-1.7,2);
\draw[kernels2]  (-1,2) -- (1,0.5);
\draw[kernels2] (-1,0.5) -- (1,2);
\draw[kernels2] (1,0.5) -- (1.7,2);
\draw (-1.2,2.2) node[g] {};
\draw (1.2,2.2) node[g] {};
}

\DeclareSymbol{pre_im_Xi4eabisc2}{2}{
\draw[kernels2] (1.2,-0.5) node[not] {} -- (-0.7,0.8) ;
\draw[kernels2] (1.2,-0.5) -- (0.4,0.8);
\draw (0,1.4)  -- (0,2.2);
\draw (1.2,2.2) -- (1.2,-0.6);
\draw (0,1) node[g] {};
\draw (0.6,2.4) node[g] {};
}

\DeclareSymbol{pre_im_Xi4eabisc22}{2}{
\draw (1.2,-0.5) node[not] {} -- (-0.7,0.8) ;
\draw[kernels2] (1.2,-0.5) -- (0.4,0.8);
\draw (0,1.4)  -- (0,2.2);
\draw[kernels2] (1.2,2.2) -- (1.2,-0.6);
\draw (0,1) node[g] {};
\draw (0.6,2.4) node[g] {};
}

\DeclareSymbol{pre_im_Xi4eabisc222}{2}{
\draw[kernels2] (0.4,-0.5) node[not] {} -- (-0.6,1) ;
\draw[kernels2] (1.2,0) -- (0.3,1);
\draw (0,1.1)  -- (0,2.5);
\draw[kernels2] (1.2,2.5) -- (1.2,0) node[not] {} -- (0.4,-0.6);
\draw (0,1.2) node[g] {};
\draw (0.6,2.5) node[g] {};
}

\DeclareSymbol{pre_im_Xi4eabc2}{2}{
\draw (0,-0.5) node[not] {} -- (-1,0.5) node[not] {};
\draw[kernels2] (-1,0.5) -- (-1.5,2);
\draw[kernels2] (0,-0.5)  -- (0.7,2);
\draw[kernels2] (-1,0.5) -- (-0.5,2);
\draw[kernels2] (0,-0.5) -- (1.5,2);
\draw (-1,2.2) node[g] {};
\draw (1,2.2) node[g] {};
}

\DeclareSymbol{pre_im_Xi4eabbisc2}{2}{
\draw[kernels2] (0,-0.5) node[not] {} -- (-1,0.5) node[not] {};
\draw[kernels2] (-1,0.5) -- (-1.5,2);
\draw[kernels2] (0,-0.5)  -- (0.7,2);
\draw[kernels2] (-1,0.5) -- (-0.5,2);
\draw (0,-0.5) -- (1.5,2);
\draw (-1,2.2) node[g] {};
\draw (1,2.2) node[g] {};
}

\DeclareSymbol{pre_im_I1Xi4abcc1}{2}{
\draw[kernels2] (0,-1) node[not] {} -- (-1,0) node[not] {};
\draw[kernels2] (0,1.1) node[not] {} -- (-1,0);
\draw[kernels2] (-1,0) -- (-1.5,2.5);
\draw[kernels2] (0,1.1) node[not] {} -- (-0.5,2.5);
\draw[kernels2] (0,1.1) -- (0.5,2.5);
\draw[kernels2] (0,-1) -- (1.5,2.5);
\draw (-1,2.7) node[g] {};
\draw (1,2.7) node[g] {};
}

\DeclareSymbol{pre_im_Xi4eabc1}{2}{
\draw (0,-0.5) node[not] {} -- (-1,0.5) node[not] {};
\draw[kernels2] (-1,0.5) -- (-1.5,2);
\draw[kernels2] (0,-0.5)  -- (-0.5,2);
\draw[kernels2] (-1,0.5) -- (0.8,2);
\draw[kernels2] (0,-0.5) -- (1.5,2);
\draw (-1,2.2) node[g] {};
\draw (1,2.2) node[g] {};
}

\DeclareSymbol{pre_im_Xi4ba1b}{2}{
\draw[kernels2] (0,0) node[not] {}  -- (1.8,1.5);
\draw[kernels2] (0,0) -- (0.8,1.5);
\draw (0,-0.1) -- (-1.8,1.5);
\draw (0,-0.1) -- (-0.8,1.5);
\draw (-1,1.7) node[g] {};
\draw (1,1.7) node[g] {};
}

\DeclareSymbol{pre_im_Xi4ba2}{2}{
\draw (0,-0.1) node[not] {}  -- (1.8,1.5);
\draw[kernels2] (0,0) -- (0.8,1.5);
\draw (0,-0.1) -- (-1.8,1.5);
\draw[kernels2] (0,0) -- (-0.8,1.5);
\draw (-1,1.7) node[g] {};
\draw (1,1.7) node[g] {};
}

\DeclareSymbol{pre_im_Xi4cabc2}{2}{
\draw[kernels2] (0,-0.5) node[not] {} -- (-1,0.5) node[not] {};
\draw[kernels2] (-1,0.5) -- (-1.5,2);
\draw (-1,0.5)  -- (0.7,2);
\draw[kernels2] (-1,0.5) -- (-0.5,2);
\draw[kernels2] (0,-0.5) -- (1.5,2);
\draw (-1,2.2) node[g] {};
\draw (1,2.2) node[g] {};
}

\DeclareSymbol{pre_im_Xi4cabc1}{2}{
\draw[kernels2] (0,-0.5) node[not] {} -- (-1,0.5) node[not] {};
\draw (-1,0.5) -- (-1.5,2);
\draw[kernels2] (-1,0.5)  -- (0.7,2);
\draw[kernels2] (-1,0.5) -- (-0.5,2);
\draw[kernels2] (0,-0.5) -- (1.5,2);
\draw (-1,2.2) node[g] {};
\draw (1,2.2) node[g] {};
}

\DeclareSymbol{pre_im_Xi4eabbisc1}{2}{
\draw[kernels2] (0,-0.5) node[not] {} -- (-1,0.5) node[not] {};
\draw[kernels2] (-1,0.5) -- (-1.5,2);
\draw[kernels2] (0,-0.5)  -- (-0.5,2);
\draw[kernels2] (-1,0.5) -- (0.8,2);
\draw (0,-0.5) -- (1.5,2);
\draw (-1,2.2) node[g] {};
\draw (1,2.2) node[g] {};
}

\DeclareSymbol{pre_im_1}{0}{
\draw[kernels2] (0,-0.5) node[not] {} -- (-0.6,0.5) ;
\draw[kernels2] (0,-0.5) -- (0.6,0.5);
\draw (0,1.1)  -- (-0.55,2);
\draw (0,1.1)  -- (0.55,2);
\draw (0,0.7) node[g] {};
\draw (0,2.2) node[g] {};
}

\DeclareSymbol{disconnect}{0}{
\draw[kernels2] (0,-0.5) node[not] {} -- (-0.6,0.5) ;
\draw[kernels2] (0,-0.5) -- (0.6,0.5);
\draw (-0.55,1.1)  -- (-0.55,2.3);
\draw (0.55,2.3) -- (0.55,1.5) -- (1.2,1.5) -- (1.2,3.5) -- (0.55,3.5) -- (0.55,2.7);
\draw (0,0.7) node[g] {};
\draw (0,2.5) node[g] {};
}

\DeclareSymbol{pre_im_2}{2}{\draw[kernels2] (0,0) node[not] {} -- (-1,1) node[not] {};
\draw[kernels2] (0,0) -- (1,1) node[not] {};
\draw[kernels2] (-1,1) -- (-1.5,2.5);
\draw[kernels2] (-1,1) -- (-0.5,2.5);
\draw[kernels2] (1,1) -- (0.5,2.5);
\draw[kernels2] (1,1) -- (1.5,2.5);
\draw (-1,2.7) node[g] {};
\draw (1,2.7) node[g] {};
}

\DeclareSymbol{CX_rec}{0}{
\draw [black] (-0.3,1) to (-0.3,-0.3);
\draw [black] (0.3,1) to (0.3,-0.3);
\draw [black] (-0.3,1) to (-0.3,2.3);
\draw [black] (0.3,1) to (0.3,2.3);
\draw (0,1) node[rec] {};
}

\DeclareSymbol{CX_cerc}{0}{
\draw [black] (0,1) to (0,-0.3);
\draw (0,1) node[cerc] {};
}

\DeclareMathAlphabet{\mathpzc}{OT1}{pzc}{m}{it}

\def\eqref#1{(\ref{#1})}

\makeatletter 
\newcommand*{\bigcdot}{}
\DeclareRobustCommand*{\bigcdot}{%
  \mathbin{\mathpalette\bigcdot@{}}%
}
\newcommand*{\bigcdot@scalefactor}{.5}
\newcommand*{\bigcdot@widthfactor}{1.15}
\newcommand*{\bigcdot@}[2]{%
  \sbox0{$#1\vcenter{}$}
  \sbox2{$#1\cdot\m@th$}%
  \hbox to \bigcdot@widthfactor\wd2{%
    \hfil
    \raise\ht0\hbox{%
      \scalebox{\bigcdot@scalefactor}{%
        \lower\ht0\hbox{$#1\bullet\m@th$}%
      }%
    }%
    \hfil
  }%
}
\makeatother

\def\act{\bigcdot}

\tcbset
{colframe=boxcolor,colback=symbols!7!pagebackground,coltext=pageforeground,
fonttitle=\bfseries,nobeforeafter,center title,size=fbox,boxsep=1.5pt,
top=0mm,bottom=0mm,boxsep=0mm,tcbox raise base}

\def\two{{\<generic>\kern0.05em\<genericb>}}
\def\twoI{{\<Ito>\kern0.05em\<Itob>}}

\def\mail#1{\burlalt{#1}{mailto:#1}}

\usepackage{thmtools}

\begin{document}

\title{Diagram-free approach for convergence of trees based model in Regularity Structures}

\author{Yvain Bruned$^1$, Usama Nadeem$^2$}
\institute{ 
 IECL (UMR 7502), Université de Lorraine
 \and University of Edinburgh \\
Email:\ \begin{minipage}[t]{\linewidth}
\mail{yvain.bruned@univ-lorraine.fr}
\\ \mail{M.U.Nadeem@sms.ed.ac.uk}
\end{minipage}}

\maketitle 

\begin{abstract}
In this work, we translate at the level of decorated trees some of the crucial arguments which have been used by P. Linares et al. in their recent paper to propose a diagram-free approach for the convergence of the model in regularity structures. This allows us to broaden the perspective and enlarge the scope of singular SPDEs covered by this approach. It also sheds new light on algebraic structures introduced in the foundational paper of Martin Hairer on regularity structures which was used later for recursively described renormalised models. 
\end{abstract}
\setcounter{tocdepth}{2}
\setcounter{secnumdepth}{4}
\tableofcontents

\section{Introduction}
 
 Solutions of singular stochastic partial differential equations (SPDEs) are given as expansions with stochastic iterated integrals coming from the iteration of Duhamel's formula. Both regularity structures introduced in \cite{reg} and Para-controlled calculus in \cite{GIP13} are based on this idea. Then, the main task that remains is to construct such iterated integrals as containing distributional products of the singular SPDEs considered. This is a challenging problem but the advent of algebraic structures like Hopf algebras that organise the tremendous computations involved, have made the task easier. In the context of regularity structures, one works with local objects, by which we mean recentred iterated integrals that need to be renormalised. Recentring and renormalisation have since been understood via two Hopf algebras in cointeraction in \cite{BHZ}. This paved the way for a general convergence theorem in \cite{CH16} and together with \cite{BCCH} it provides a general black box for treating singular SPDEs. In the discrete setting, one is far from a general convergence result, but progress (see \cite{Mat18,MH21,BN}) has been made by building upon the convergence result obtained in \cite{HQ15} with the use of discrete regularity structures in \cite{EH17}. Let us mention that the most general result is given in \cite{BN} where the authors treat the generalised KPZ equation - one of the most challenging singular SPDEs.
All of these approaches are based on renormalising Feynman diagrams obtained from decorated trees which are combinatorial codings for iterated integrals. The renormalisation implemented is close in spirit to the  BPHZ renormalisation \cite{BP57,KH69,WZ69}. 
 
Recently, a different approach \cite{LOTP} has been proposed in the context of regularity structures, which is based on coding that replaces decorated trees with multi-indices \cite{OSSW,LOT,LO22}. Multi-indices can be understood as a way of compressing information and collecting decorated trees sharing some features under the umbrella of the same symbol. With such an approach, one has no direct access to a specific iterated integral and therefore the authors developed in \cite{LOTP} a diagrammatic-free approach for the convergence which relies on the inductive construction of multi-indices. Another key idea in their construction is to use a spectral gap inequality that allows them to connect the $p$-th moment of a stochastic integral to one of its Malliavin derivative. This approach is quite promising but thus far has only been affected on the stochastic quasilinear parabolic equation.
 
In this work, we want to push forward this convergence result to a high degree of generality by applying these ideas to decorated trees. The scope will now be any subcritical equation as described in \cite{reg}, made possible by providing general algebraic formulae at the level of decorated trees that were written in a specific case for multi-indices in \cite{LOTP}. We reiterate that deriving these formulae is the main contribution of our paper which when combined with the analytic estimates from \cite{LOTP}, one can obtain a general convergence result which provides an alternative Black Box for singular SPDEs based on recursive arguments, where the model is constructed with a preparation maps introduced in \cite{BR18} and the renormalised equation is as given in \cite{BB21}. This reveals a major simplification, as a lighter combinatorial structure is used in comparison to \cite{BHZ}.
Let us mention a different approach for decorated trees in \cite{HS23} introduced a few months after our work. The authors prove a general convergence result inspired also by \cite{LOTP}. They redo entirely the analytical part by introducing pointed Besov modelled distributions. For the algebraic part, they do not use preparation maps and there is not an obvious correspondence with the key algebraic identities of  \cite{LOTP}. It is also not clear whether the quasi-linear case is covered by their approach, whereas our work the identities obtained are in the quasi-linear context.
Just before the work \cite{HS23}, another work by the first author \cite{BB23} presented a simple proof of the convergence of the renormalised model for the generalised KPZ equation. It uses also the spectral gap inequality but it based on diagrammatic approach following recent advances in \cite{BN} and also using some algebraic identities of this work such as  Proposition~\ref{malliavin_derivative_model1}.

 The other impact of this work is the unification of the various combinatorial approaches such as decorated trees and multi-indices by showing ideas in one structure appear also naturally in the other. This programme has been started in \cite{BK} where it has been shown that a common post-Lie structure is at the origin of the Hopf algebras used for singular SPDEs. The result can be described in both languages: decorated trees and multi-indices.
 
Let us be more specific about the main ideas of this paper. Solutions $ u $ of local subcritical SPDEs are locally described by
\begin{equs}
u(y) -  u(x) = \sum_{\tau \in \mathcal{T}} u_{\tau}(x)(\Pi_x \tau)(y),
\end{equs}
where $ (\Pi_x \tau)(y) $ are stochastic iterated integrals recentred around the point $ x $, $ \mathcal{T} $ is a combinatorial set parametrising the iterated integrals - that can comprise either decorated trees \cite{reg,BHZ} or multi-indices \cite{OSSW,LOT} - and the $ u_{\tau}(x) $ are some kind of derivatives. Then, the theory of regularity structures provides a reexpansion map $ \Gamma_{xy} $ that allows us to move the recentring:
\begin{equs}
\Pi_y = \Pi_x \Gamma_{xy}.
\end{equs}
The collection of these two maps $ (\Pi_x,\Gamma_{xy}) $ is what is referred to as a model \cite[Def. 3.1]{reg}, and it defines the topology upon which one wants to pass to the limit for constructing solutions of singular SPDEs. Indeed, the rough noise in the equation is replaced by a mollified version depending on a small parameter $ \varepsilon $, and one considers the mollified model $ (\Pi^{(\varepsilon)}_x, \Gamma^{(\varepsilon)}_{xy})$ with the goal of eventually repealing the mollification. This model reflects the ill-defined distributional products of a singular SPDE and therefore fails to converge without an appropriate renormalisation. This renormalisation is implemented recursively by the choice of a preparation map $ R $ satisfying some suitable local properties, which allows one to consider the renormalised model denoted by $ (\Pi_x^R, \Gamma_{xy}^R) $. 
In \cite{LOTP}, the authors solve quasilinear parabolic equations of the form:
\begin{equs}\label{eq:quas}
(\partial_ t - \partial_x^{2})u = a(u) \partial_{x}^2 u + \xi,
\end{equs}
driven by a stationary noise $\xi$ that is assumed to satisfy the spectral gap inequality - see \cite[Ass. 2.1]{LOTP}. The central quantity of interest in their approach is $\Pi^{-}_{x}$, which is identified as the negative-degree part of their model. One can refer to \cite[Eq. 2.18]{LOTP} for the proper expression, but on the regularity structures side, this corresponds to applying $\Pi_x^R$ to the non-linearity in \eqref{eq:quas}. We may use these interchangeably, being mindful of the correspondence between decorated trees and the multi-index that encodes them. The significance of this object in their argument lies in the fact that a Schauder estimate of $\Pi^{-}_{x}$ in base point $x$ - which they represent by $\Pi_x$ and corresponds to $\Pi^{R}_x\mathcal{I}(\cdot)$ on the regularity structures side - can solve the renormalised version of \eqref{eq:quas}, provided that one has control of the $p$-th moments of $\Pi^{-}_{x}$. To gain this control of the $p$-th moments of $\Pi^-_x$ the authors appeal to the spectral gap inequality, which reduces the problem to controlling the first moment of $\Pi^-_x$ and the $p$-th moment of its Malliavin derivative $\delta\Pi^-_x[\xi]$. 
Indeed, one has 
\begin{equs}
	\left(\mathbb{E}|\Pi_x^{-}|^p\right)^{\frac{1}{p}} \lesssim \left|\mathbb{E}\Pi_x^{-}\right| + \mathbb{E}\left(\Vert \delta\Pi_x^{-}\Vert_{\star}^p\right)^{\frac{1}{p}},
	\end{equs}
where $ \Vert \cdot \Vert_{\star} $ is a suitable functional norm.
Control of $\left|\mathbb{E}\Pi_x^{-}\right|$ is easier of the two tasks; it is a consequence of the implementation of the BPHZ renormalisation procedure. For the control of the latter, the task is made difficult however by the fact that one sees divergent constants in $\Pi_x^{-}$. One hopes that $\delta c = 0$ should kill the constants but the configuration in which $c$ appear keeps this from happening. The way out is to evaluate at the diagonal and then taking the Malliavin derivative. This extricates $\delta\Pi^{-}$ from the divergent constants but has the unwanted effect of restricting it to the diagonal, motivating the need to recentre. The natural candidate map to affect this is $\Gamma_{xy}$, which in their paper is written on the dual side as $\Gamma_{xy}^{*}$. This leads the authors to the identity
\begin{equs}
\delta \Pi^{-}_y(x) -  \delta \Gamma^{*}_{yx}&\Pi^{-}_{x} \\
&=\sum_{k\ge 0}z_k\Pi_y^{k}(x)\partial_1^{2}\left(\delta\Pi_y-\delta\Pi_y(x)-(\delta\Gamma^{*}_{yx}\Pi_x)\right) + \delta\xi_\epsilon(x)\one,
\end{equs}
where the left-hand side is just an application of Leibniz rule to $\delta\left(\Gamma^{*}_{xy}\Pi^{-}_{y}\right)$. The expression $\delta\Pi_y-\delta\Pi_y(x)-(\delta\Gamma^{*}_{yx}\Pi_x)$ is recognised as having the potential to locally describe $\delta \Pi^{-}_y(x) -  \delta \Gamma^{*}_{yx}\Pi^{-}_{x}$ but it is involved in a possibly ill-defined product with $\Pi_y^{k}$. To circumvent this, the authors replace $\delta\Gamma^{*}_{yx}$ with an endomorphism $d\Gamma^{*}_{yx}$ which enforces a particular bound on $\delta\Pi_y-\delta\Pi_y(x)-(\delta\Gamma^{*}_{yx}\Pi_x)$, leading to the identity:
\begin{equs} \label{renormalisation_identity}
\delta \Pi^{-}_y(x) -  d \Gamma^{*}_{yx}&\Pi^{-}_{x} \\
&=\sum_{k\ge 0}z_k\Pi_y^{k}(x)\partial_1^{2}\left(\delta\Pi_y-\delta\Pi_y(x)-(d\Gamma^{*}_{yx}\Pi_x)\right) + \delta\xi_\epsilon(x)\one.
\end{equs}
The map $d\Gamma^{*}_{yx}$ is algebraically similar to $\delta\Gamma^{*}_{yx}$, and in fact they are chosen so that they agree on planted trees. In a similar vein the authors derive the following expression
\begin{equs} \label{magic_formula}
	\begin{aligned}
	Q \left( \delta \Pi^{-}_x -  d \Gamma^{*}_{xz} Q \Pi^{-}_{z} \right)(z) &= Q \sum_{k \geq 0} z_k \Pi_x^k (z) \partial_x^2 \left( \delta  \Pi_x -  d \Gamma^{*}_{xz} Q  \Pi_{z}  \right)(z) \\ &+ \delta \xi_{\varepsilon} (z),
	\end{aligned}
\end{equs}
where $ \Pi_x  $ corresponds to the model applied to terms from the right hand side of the equation which have been integrated,  $ \xi_{\varepsilon} $ is the mollified space-time white noise, $ Q $ is a projection that considers only the most singular terms in the expansion, $ z_k $ is an abstract variable that keeps track of the exponent $ k $ in $ \Pi^k_x (z) $. A crucial observation about the formula \eqref{magic_formula} is the absence of any renormalisation term: the authors of \cite{LOTP} described it as a $ c $-free formula. The main contribution of this work is to recover this formula at the level of decorated trees and to extend it to a large class of models.
This is done by considering two vector spaces of decorated trees, one $ \mathcal{T}_0 $ coming from the right hand side of a singular SPDEs and the other $ \mathcal{T}_1 $ is such that one instance of the noise $ \xi $ inside an element of $ \mathcal{T}_0 $ is replaced by an infinitesimal perturbation $ \delta \xi $. Formally, we define a Malliavin derivative $ D_{\Xi} : \mathcal{T}_0 \rightarrow   \mathcal{T}_1$ at the level of decorated trees (see Definition~\ref{malliavin_derivative}). Then, we introduce two renormalised models on decorated trees in \eqref{recursive_model}: $ \Pi_x^{R,0} $ taking into account $ \delta \xi $ when computing the length of Taylor expansion for the recentering and  $ \Pi_x^{R,1} $ that treats $\delta\xi$ as it would be $\xi$, and as a result produces shorter Taylor expansions. The map $D_\Xi$, we prove in Theorem~\ref{malliavin_derivative_model1} allows us to give a combinatorial meaning to the fundamentally analytical Malliavin derivative $\delta\Pi^{R,1}$:
\begin{equs}
\delta\Pi^{R,1} = \Pi^{R,1}D_\Xi{\color{blue}.}
\end{equs}
We will further show that the two models are connected via a map $ \hat{\Delta}_0 $ which is close to the way the renormalisation was handled in \cite{reg}.
It is defined recursively following ideas developed in 
\cite{BR18} in Proposition~\ref{prop:magicformula}. 
The following identity encapsulates this relationship:
\begin{equation} \label{magicformula-intro}
	\Pi_x^{R,1} = \left( \Pi_x^{R,0}\otimes f_x^{R,0} \right) \hat{\Delta}_0.
\end{equation}
The map $\hat\Delta_0$ here links the two models essentially by curtailing the excess terms in the Taylor expansions under $ \Pi_x^{R,0} $ and $ f_x^{R,0} $ is a map into reals, defined via $\Pi_x^{R,0}$. From this identity, one is able to define $ d \Gamma_{yx} $ on decorated trees in Definition~\ref{def:dGamma}:
\begin{equs}\label{eq:dGammadef-intro}
	d \Gamma_{yx}^R  = \mathcal{Q}_0 \left( \Gamma_{yx}^{R,0} \otimes f_x^{R,0} \right) \hat{\Delta}_0 D_{\Xi}, 
\end{equs}
where $ \mathcal{Q}_0 $ projects to zero decorated trees with at least one noise $ \delta \xi $ and $ \Gamma_{yx}^{R,0} $ is the reexpansion map for $ \Pi_x^{R,0} $. With this definition and by putting an assumption on the renormalisation (Assumption~\ref{assumpt2}), we provide a version of \eqref{magic_formula} on decorated trees in Theorem~\ref{main_result} together with Corollary~\ref{coro_identi}. The model for the quasilinear equation is a peculiar one because in it one sees terms containing $ \partial_x^2 K $ - where $K$ is the heat kernel. These do not lend well to the analysis, because there is no gain of regularity by convolving with the kernel $\partial_x^2 K$. However, for other models satisfying Assumption~\ref{assumpt1} like the generalised KPZ equation and the $ \varphi^4_3 $, we are able to provide a stronger statement in Theorem~\ref{th:local without R}, for every $ \tau \in \mathcal{T}_0 $
\begin{equs} \label{local R_intro}
	\left( \Pi^{R,1}_y  d \Gamma_{yx}^R \tau  \right)(y)  =   \left(\delta \Pi_x^{R,1} \tau \right)(y).
\end{equs}
To gain control of \eqref{renormalisation_identity}, the authors appeal to a reconstruction argument, which involves constructing a distribution $F$ for a family of distributions $\{F_{z}\}_z$, indexed by spacetime, such that near the diagonal $F(z)\approx F_z(z)$ in a suitable sense. This reconstruction necessitates (amongst other things) control of the so-called "continuity expression" identity:
\begin{equs} \label{conti_Gamma}
d\Gamma_{xy}^{*} - d\Gamma_{xz}^{*}\Gamma^{*}_{zy}.
\end{equs}
In Proposition~\ref{control_gamma}, we rewrite this identity in terms of $ \gamma^{R,0}_{xy} $ that characterise $ \Gamma^{R,0}_{xy} $.
Once control of the \eqref{renormalisation_identity} is had, one is able to construct the rough path increment of $\delta\Pi_{x}$ and control it. In particular, they prove that the rough path increment of $\delta\Pi_x$:
\begin{equs} \label{other_ident}
\delta \Pi_x - \delta \Pi_x (z) -  d \Gamma^{*}_{xz}\Pi_{z},
\end{equs}
relates to the rough path increment of $\delta\Pi_x^{-}$ via an integral representation. In Proposition~\ref{ident_Pi}, we explain how this identity translates to decorated trees. Control of these two rough path increments can then be leveraged to get control of $p$-th moments of $\delta\Pi^{-}_{x}$.

Let us now outline the paper by summarising the content of its sections. In Section~\ref{Section::2}, we recall the basics of decorated trees and illustrate the rules associated with singular SPDEs for constructing those trees. Then, we mention one crucial assumption for the renormalised quasilinear equation (see Assumption~\ref{assumpt2}) that will be used in the sequel.  We also consider decorated trees with two different noises for encoding Malliavin derivatives. We introduce this derivative $ D_{\Xi} $ at the level of the trees - see Definition~\ref{malliavin_derivative}. We conclude the section with the important Assumption~\ref{assumpt1} on the positive degree of branches containing a Malliavin derivative. This is crucial in establishing the aforementioned algebraic identities. In Section~\ref{Section::3}, we introduce renormalised models based on two different degrees which differ on whether they take into account the Malliavan derivative on the noise or not. The renormalisation is introduced via a preparation map in Definition~\ref{DefnPreparationMap} that provides a recursive definition of the model. Here, we follow the construction given in \cite{BR18,BB21}. The main novelty of the section is to explain how to move from one model to the other given in Proposition~\ref{prop:magicformula}. The proof is based on a recursive map $ \hat{\Delta}_0 $ which has been originally introduced in the context of renormalisation see \cite{reg,BR18}. Finally, in Section~\ref{sec:AlgIde} we show how the previously built machinery can be used to bridge the gap between the approaches \cite{reg} and \cite{LOTP}. Namely, we propose algebraic counterparts on decorated trees, to the objects employed in \cite{LOTP} - see Proposition~\ref{malliavin_derivative_model1} and Definition~\ref{def:dGamma} - before showing how they can be used to reproduce some of the most important identities used in \cite{LOTP}. These identities are presented in Proposition~\ref{control_gamma}, Theorem~\ref{th:local without R} Theorem~\ref{main_result}, Corollary~\ref{coro_identi} and Proposition~\ref{ident_Pi}.
\subsection*{Acknowledgements}

{\small
	The authors thank Pablo Linares, Felix Otto, Markus Tempelmayr, and Pavlos Tsatsoulis for interesting discussions on the topic of multi-indices. They also thank the reviewer for their meticulous double-checking and insight.
Y. B. thanks the Max Planck Institute for Mathematics in the Sciences (MiS) in Leipzig for having supported his research via a long stay in Leipzig from January to June 2022. Y. B.
is funded by the ANR via the project LoRDeT (Dynamiques de faible régularité via les arbres décorées) from the projects call T-ERC\_STG. U. N. thanks the Max Planck Institute for Mathematics in the Sciences (MiS) for a short stay in Leipzig.
}

\section{Decorated Trees}\label{sec:DecTrMI}

\label{Section::2}

We begin naturally by explicating on decorated trees, the main combinatorial structure used in SPDEs, which were first systematically described in this context in \cite{BHZ}.

We pick three symbols $\mcI$, $\Xi_0$, and $\Xi_1$ and define with it, a set of edge decorations $ \mathcal{D} \coloneqq \lbrace \mcI,\,\Xi_0,\,\Xi_1 \rbrace \times \mathbb{N}^{d+1}$. The first of the symbols $\mcI$ represents convolution with a kernel that comes from the differential operator of the equation one is interested in, and the symbols $\Xi_0,\,\Xi_1$ represent the noise term and an infinitesimal perturbation of it. When working with a system of SPDEs with more than one kernel and noise, one can augment $\mathcal{D}$ with new symbols (and symbol pairs) for the extra kernels and noises.

\begin{definition}
A \textbf{decorated tree} over $\mathcal{D}$ is a $3$-tuple of the form  $\tau_{\Labe}^{\Labn} =  (\tau,\Labn,\Labe)$ where $\tau$ is a non-planar rooted tree with node set $N_\tau$ and edge set $E_\tau$. The maps $\Labn : N_{\tau} \rightarrow \mathbb{N}^{d+1}$ and $\Labe : E_\tau{\tiny } \rightarrow \mathcal{D}$ are node and edge decorations, respectively.
\end{definition}

We use $T$ to denote the set of decorated trees and $ \mathcal{T} $ is its linear span.  We define a binary tree product by 
\begin{equation}  \label{treeproduct}
 	(\tau,\Labn,\Labe) \cdot  (\tau',\Labn',\Labe') 
 	= (\tau \cdot \tau',\Labn + \Labn', \Labe + \Labe')\;, 
\end{equation} 
where $\tau \cdot \tau'$ is the rooted tree obtained by identifying the roots of $ \tau$ and $\tau'$. The sums $ \Labn + \Labn'$ mean that decorations are added at the root and extended to the disjoint union by setting them to vanish on the other tree. Each edge and vertex of both trees keeps its decoration, except the roots which merge into a new root decorated by the sum of the previous two decorations.

\begin{enumerate}
   \item[(i)] An edge decorated by  $ (\mcI,a) \in \mathcal{D} $  is denoted by $ \mcI_{a} $. The symbol $  \mcI_{a} $ is also viewed as the operation that grafts a tree onto a new root via a new edge with edge decoration $ a $. The new root at hand remains decorated with $0$. 
   
   \item[(ii)]  An edge decorated by $ (\Xi_i,0) \in \mathcal{D} $ is denoted by $  \Xi_i $, for $i\in\{0,1\}$.

   \item[(iii)] A factor $ X^k$   encodes a single node  $ \bullet^{k} $ decorated by $ k \in \mathbb{N}^{d+1}$. We write $ X_i$, $ i \in \lbrace 0,1,\ldots,d\rbrace $, to denote $ X^{e_i}$. Here, we have denoted by $ e_0,...,e_d $ the canonical basis of $ \mathbb{N}^{d+1} $. The element $ X^0 $ which encodes $\bullet^0$, will be denoted by $\one$. The space of all the monomials $X^{k}$ will be denoted by $\bar{T}$, and its linear span by $\bar{\CT}$.
 \end{enumerate}
 Using this symbolic notation any decorated tree $ \tau \in T $ can be represented as:
 \begin{equs}\label{eq:treeform}
 \tau = X^k\Xi_0^l \Xi_1^m  \prod_{i=1}^n \mcI_{a_i}(\tau_i),
 \end{equs}
 where $ \prod_i $ is the tree product, $ k \in \mathbb{N}^{d+1} $, $l$ , $m \in \mathbb{N} $. In relevant applications, and particularly for the one we have in mind for this article, a product of noises is not allowed and one can only consider the cases for which $ l + m \le 1$. A planted tree is a tree of the form $ \mcI_a(\tau) $ meaning that there is only one edge connecting the root to the rest of the tree. The decomposition \eqref{eq:treeform} can be reinterpreted by saying that any decorated tree admits a unique decomposition into a product of planted trees. We also introduce abstract derivatives $\CD_p$, for $p\in\mathbb{N}^{d+1}$ by requiring that
\begin{equs}
\CD_p\CI_\alpha\tau=\CI_{(\alpha+p)}\tau,\quad \CD_pX^q = \begin{cases}
\frac{q!}{(q-p)!}X^{q-p} & q \ge p \\
0 & \mbox{otherwise},
\end{cases}
\end{equs}
and finally extending to all of $T$, by Leibniz rule.
Upon the noises we define the degree maps $\textsf{deg}_0$ and  $\textsf{deg}_1$ by setting
\begin{equs}
 \textsf{deg}_0(\Xi_0)= \textsf{deg}_1(\Xi_0) = \textsf{deg}_1(\Xi_1) =  \alpha, \quad \textsf{deg}_0(\Xi_1) = \alpha + \frac{d+2}{2},
 \end{equs}
where $ \alpha $ is the space-time regularity of the noise $ \xi $ encoded by $ \Xi_0 $ and the presence of $\tfrac{d+2}{2}$ in $\textsf{deg}_0(\Xi_1)$ is to keep track of the gain in regularity that comes from taking a Malliavin Derivative - this point is codified in \eqref{eq:malliavin_derivative_model1}. From there the degree maps are extended to all of the $T$ by decreeing that for $j\in\{0,1\}$:
\begin{alignat*}{2}
&\textsf{deg}_j(X_0)= 2,&&\textsf{deg}_j(X_i) = 1,\,\,\,\text{for }i\neq 0, \\
&\textsf{deg}_j\left(\prod_{i}\tau_i\right)=\sum_i\textsf{deg}_j(\tau_i),\qquad&&\textsf{deg}_j\left(\mcI_\alpha\left(\tau\right) \right)= \textsf{deg}_j(\tau) + 2 - |\alpha|_{\s_{\text{\tiny par}}},
\end{alignat*}
where for a fixed scaling vector $\s =  (\s_0,\hdots,\s_d) \in\mathbb{N}_{> 0}^{d+1}$, and some multi-index $\mbn\in\mb{N}_{\ge 0}^{d+1}$ we define the quantity
 \begin{equs}
|\mathbf{n}|_{\s}\coloneqq \sum_{i=0}^d \s_i \mbn_i,
\end{equs}
and $\s_\text{par} = (2,\hdots,1)$. Notice that this means for a multi-index $\mbn$, $\textsf{deg}_j(X^\mbn) = |\mbn|_{\s_\text{par}}$. We consider parabolic scaling in this paper - for which reason we fix $\s = \s_{\text{par}}$ - but everything works for a general scaling, with the obvious modification of the degree maps on the monomials. In the sequel, we will use a short hand notation replacing $ |\cdot|_{\s} $ by $ |\cdot| $. We remark that the "$+2$" in the definition of $\textsf{deg}_i(\mcI_\alpha(\tau))$ is due to Schauder estimates. In addition to the degree maps, we also define $|\cdot|_{\Xi}$ as the mapping that outputs the number of noises in a decorated tree, without any discrimination between $\Xi_0$ and $\Xi_1$.

We suppose we are given a subspace of trees $ T_0 \subset T $ associated with the right hand side of a singular SPDE defined on the space-time $\mathbb{R}_+\times\mathbb{R}^{d}$ with the mild formulation:
 \begin{equs} \label{main_equation}
 u = K * \left( F(u, \nabla u,...) \xi + G(u,\nabla u,...) \right),
 \end{equs}
where $ * $ is the space-time convolution, $ \xi  $ is  space-time white noise, $ F $ and $ G $ are nonlinearities depending on the solution $ u $ and its derivatives, and $ K $ is the heat kernel. The linear span of $ T_0 $ is denoted by $ \mathcal{T}_0 $.
  Due to the affine structure of the noise, the term $ \Xi_0  $ that encodes $ \xi $ does not appear in any meaningful formal expansion of a potential solution to equation \eqref{main_equation}; this puts a constraint on the way the decorated trees for this equation are constructed. Such constraints are formalised through the notion of a normal, complete rule. In lieu of giving a complete exposition of such rules, we illustrate them on concrete examples, as well as referring the reader to Section 5 of \cite{BHZ}. So if one considers the quasilinear parabolic equation, written in a manner commensurate with \cite{BHZ}:
\begin{equs}
\partial_ t u - \partial_x^{2} u = f(u)\partial_x^2 u + \xi, \quad (t,x) \in \mathbb{R}_+\times\mathbb{R},
\end{equs}
then decorated trees associated with this equation can be generated by the following rules $ \mathcal{R}_{\text{\tiny{qua}}} $:
\begin{equs}
\mathcal{R}_{\text{\tiny{qua}}} = \left\lbrace \left.\left( \prod_{i\in I} \mathcal{I}(\cdot)\right) \mathcal{I}_2(\cdot), \, \Xi \, \right|\, I \text{ finite set} \right\rbrace,
\end{equs}
where $ \mathcal{I}_2 $ is a short hand notation for $ \mathcal{I}_{(0,2)} $.
This means that given $ (\tau_i)_{i \in I}, \tau $ with $ I $ finite set and $ \tau_i, \tau \in T_0 $, one has
\begin{equs}
\left(\prod_{i\in I} \mathcal{I}(\tau_i)\right) \mathcal{I}_2(\tau) \in T_0, \quad \Xi \in T_0.
\end{equs}
In Section~\ref{Section::3}, we will introduce renormalisation maps $ R : \mathcal{T}_0 \rightarrow \mathcal{T}_0$ that will ask to enlarge the set of decorated trees $ T_0 $ that we started with. For quasilinear equations, we will find that the following assumptions are sufficient:
\begin{assumption}\label{assumpt3}
We assume that the preparation map R is such that
\begin{equs}
R\CQ_0 = \CQ_0R.
\end{equs}
\end{assumption}
\begin{assumption} \label{assumpt2}
For every $ \tau  \in \mathcal{T}_0$, we assume that $ R $ is such that:
\begin{equs} \label{decomposition_R}
(R - \id) \left[\left(\prod_{i\in I} \mathcal{I}(\tau_i)\right) \mathcal{I}_2(\tau)\right] = 
 \sum_{i\in J} c_i \left(\prod_{j \in K_{J}} \mathcal{I}(\tau_{i,j})\right),
\end{equs}
where $J, K_J$ are finite sets and the trees $ \tau_{i,j}$ belong to $ \mathcal{T}_0 $.
\end{assumption}
The prototypical example of a preparation map satisfying these assumptions, will be given in \eqref{eq:prepmap}.
This implies that one has to consider the new set of rules in order to consider terms of the form $ \prod_{i \in I} \mathcal{I}(\tau_i) $:
\begin{equs}
\mathcal{R}^c_{\text{\tiny{qua}}} = \left\lbrace \prod_{i\in I} \mathcal{I}(\cdot), \, \left.\left( \prod_{i\in I} \mathcal{I}(\cdot)\right) \mathcal{I}_2(\cdot),  \, \Xi\, \right| \, I \text{ is a finite set} \right\rbrace.
\end{equs}
The fact that now $ \mathcal{T}_0 $ is stable by $ R $ means that $ \mathcal{R}^c_{\text{\tiny{qua}}} $ is complete. Normality ensures that $ \mathcal{T}_0 $ is stable under the action of the coaction and coproduct introduced in Section~\ref{Section::3}. Let us also provide normal and complete rules for other known models. The generalised KPZ equation is given by:
\begin{equs}
\partial_t u - \partial_x^2 u = f(u) (\partial_x u)^2 + g(u) \xi, \quad (t,x) \in \mathbb{R}_+\times\mathbb{R}.
\end{equs}
Then, the rules $ \mathcal{R}^{c}_{\text{\tiny{gKPZ}}} $ are:
\begin{equs}
\mathcal{R}^{c}_{\text{\tiny{gKPZ}}} = \Biggl\lbrace \prod_{i\in I} \mathcal{I}(\cdot), \, &\left( \prod_{i\in I} \mathcal{I}(\cdot)\right) \mathcal{I}_1(\cdot),\, \left( \prod_{i\in I} \mathcal{I}(\cdot)\right) \mathcal{I}_1(\cdot) \mathcal{I}_1(\cdot), \\
&\left.\left( \prod_{i\in I} \mathcal{I}(\cdot)\right) \Xi\, \right| \, I \text{ is a finite set} \Biggr\rbrace,
\end{equs}
where $ \mathcal{I}_{1} $ is a short hand notation for $ \mathcal{I}_{(0,1)} $.
For the $ \varphi^{4}_3 $ equation
\begin{equs}
\partial_t u - \Delta u = u^3 +  \xi, \quad (t,x) \in \mathbb{R}_+\times\mathbb{R}^3,
\end{equs}
one has
\begin{equs}
\mathcal{R}^c_{\tiny{\varphi^{4}_3}} = \lbrace  \mathcal{I}(\cdot), \, \mathcal{I}(\cdot) \mathcal{I}(\cdot),  \mathcal{I}(\cdot) \mathcal{I}(\cdot)\mathcal{I}(\cdot), \, \Xi \rbrace.
\end{equs}
 We denote by  $ T_1 $ the set generated by decorated trees in $ T_0 $ where at most one instance of the noise $ \Xi_0 $ has been replaced by $ \Xi_1 $. We set $ \mathcal{T}_1 $ to be the linear span of $ T_1 $.
 Given a subset $ E \subset T $, we define the following set:
 \begin{equs}
 	E^{+,i} := \left\lbrace X^k \prod_{j} \mathcal{I}^{+,i}_{a_j}(\tau_j) \, \middle| \, \textsf{deg}_i(  \mathcal{I}_{a_j}(\tau_j) ) > 0, \, \tau_j \in E, \, k \in \mathbb{N}^{d+1}   \right\rbrace. 
 \end{equs}
The new symbol $ \mathcal{I}^{+,i}_{a_j} $ has been chosen in order to stress the difference between the set above and $ E $, which occurs especially when $ E = T_0 $. There is no constraint on the product at the root for the decorated trees in $ T_0^{+,i} $. The $i$ in the superscript of $ \CI^{+,i}_a $ is meant to emphasise the dependence of the projection on $\textsf{deg}_i $. The symbol can also be extended to any element of $ T $ by sending to zero the trees of negative degree:
\begin{equs}
	\CI^{+,i}_a(\tau) = 0, \quad \textsf{deg}_i(  \mathcal{I}_{a}(\tau) ) \le 0.
\end{equs}
 As there is no noise of type $ \Xi_1 $ in $ T_0 $, one can perform the following identification $ T_0^{+,0} = T_0^{+,1} $.
We denote by $ \mathcal{T}^{+,i} $ the linear span  of $ T^{+,i} $.

 We finish this section by defining a natural derivative $ D_{\Xi} $ that allows us to move from $ \mathcal{T}_0 $ into $ \mathcal{T}_1 $. It also comes with a natural assumption on the degree of the trees that contain extra instances of the noise of the form $ \Xi_1 $.
\begin{definition} \label{malliavin_derivative}
The linear map $ D_{\Xi} \,: \, \mathcal{T}_0 \rightarrow \mathcal{T}_1 $ is defined as the derivation that turns $ \Xi_0 $ of a given decorated tree into $ \Xi_1 $. It can be defined inductively on the tree structure by requiring
\begin{equs}
D_{\Xi} \left( X^k \Xi_0 \prod_{i=1}^n \mathcal{I}_{a_i}(\tau_i) \right) = X^k \Xi_1 \prod_{i=1}^n &\mathcal{I}_{a_i}(\tau_i) \\
&+ X^k \Xi_0  \sum_{j=1}^n \mathcal{I}_{a_j}(D_{\Xi} \tau_j) \prod_{i \neq j} \mathcal{I}_{a_i}(\tau_i),
\end{equs}
and also $D_{\Xi}\left(\Xi_0\right) = \Xi_1$, and $D_{\Xi}\left(X^\mbn\right) = D_{\Xi}\left({\mathbf{1}}\right)=\mathbf 0$.
\end{definition}
 
We make the following crucial assumption on the space $ \mathcal{T}_0 $:
\begin{assumption} \label{assumpt1}
For every $ \tau  \in \mathcal{T}_0$, we assume that:
\begin{equs} \label{decomposition_D_Xi}
D_{\Xi} \tau = \sum_{i} \mathcal{I}_{a'_i}(D_\Xi \tau_i') \tau_i + \Xi_1 \tau', 
\end{equs}
where the $ \tau_i, \tau_i', \tau' $ belongs to $ \mathcal{T}_0 $ and $ \tau' $ does not have any noise at the root. We suppose that
\begin{equs} \label{degree_positive}
\textnormal{\textsf{deg}}_0( \mathcal{I}_{a'_i}(D_\Xi \tau_i') ) > 0,
\end{equs}
where in the later identity, we have made an abuse of notation as $ \textnormal{\textsf{deg}}_0 $ is not defined for a linear combination of decorated trees. Here, we know that $ D_\Xi \tau_i'  $ is a linear combination of decorated trees of the same degree. Therefore,  we denote by $ 
\textnormal{\textsf{deg}}_0( \mathcal{I}_{a'_i}(D_\Xi \tau_i') )  $ their degree.
\end{assumption}
\section{Renormalised Models}

\label{Section::3}

The first step towards being able to reproduce the arguments in $\cite{LOTP}$ at the level of trees, is to construct algebraic formulation of the Malliavin Derivatives $\delta\Pi$ and $\delta\Gamma$ employed there, commensurate to decorated trees. For this, we will need the linear, multiplicative coactions $ \Delta_i : \CT_i \rightarrow \CT_i \otimes \CT^{+,i} $ that can be traced back to Hairer's seminal paper \cite{reg}, where these co-actions allowed for the construction of a regularity structure to which (a modified version) of the noise is lifted.
\begin{equation} \label{coaction}
\begin{split}
&\Delta_i \mathbf{1} \coloneqq \mathbf{1} \otimes \mathbf{1},\qquad \Delta_i X_j = X_j \otimes \mathbf{1} + \mathbf{1}\otimes X_j, \\
& \Delta_0\Xi_0 = \Xi_0\otimes\one,\hspace{5.5mm}\Delta_1 \Xi_i \coloneqq \Xi_i \otimes \mathbf{1},\\
&\Delta_i(\mcI_\alpha\tau) \coloneqq (\mcI_\alpha\otimes \Id )\Delta_i \tau + \hspace{-4mm}\sum_{\vert\ell +m\vert < \textsf{deg}_i(\mcI_\alpha\tau)} \frac{X^\ell}{\ell!}\otimes \frac{X^m}{m!}\mcI^{+,i}_{\alpha+\ell+m}(\tau).
\end{split}
\end{equation}
 By construction we have on it a following triangular structure 
\begin{equs}\label{eq:sweedlernotation}\Delta_i\tau = \tau\otimes\mathbf{1} + \sum_{(\tau)}\tau^{(1)}\otimes\tau^{(2)}\end{equs} such that $\textsf{deg}_i(\tau^{(1)})<\textsf{deg}_i(\tau)$.

We also supplement these coactions with coproducts $\Delta^{+}_{i}:\CT^{+,i}\rightarrow \CT^{+,i}\otimes \CT^{+,i}$.
\begin{equs} \label{def_delta_+} \begin{aligned}
\Delta^{+}_{i}(\tau\bar\tau)&=\Delta^{+}_{i}(\tau)\Delta^{+}_{i}(\bar\tau) \\
\Delta^{+}_{i}\left(\CI_\alpha^{+,i}\tau\right)&=\sum_{|\ell|<\textsf{deg}_i(\CI_\alpha(\tau))}\left(\CI^{+,i}_{\alpha+\ell}\otimes\frac{(-X)^{\ell}}{\ell!}\right)\Delta\tau+\one\otimes\CI^{+,i}_{\alpha}\tau.
\end{aligned}
\end{equs}
\color{black}
 In the sequel, we will denote $ \mathcal{T}^{+,0} $ by $ \mathcal{T}^+ $.
\begin{remark}
An alternative definition of \eqref{coaction} has been introduced in \cite{BHZ} where the last identity is replaced by
\begin{equs}
\Delta_i(\mcI_a\tau) \coloneqq (\mcI_a\otimes \Id )\Delta_i \tau + \hspace{-4mm}\sum_{\vert\ell \vert < \textsf{deg}_i(\mcI_a\tau)} \frac{X^\ell}{\ell!}\otimes \mcI^{+,i}_{a+\ell}(\tau).
\end{equs}
which corresponds to a change of basis: $ \tilde{\mcI}^{+,i}_{a+\ell}(\tau) =\sum_{m} \frac{X^m}{m!} \mcI^{+,i}_{a+\ell+m}(\tau)$. We have chosen the original formulation coming from \cite{reg} because one gets simple formulae for the map $ \hat{\Delta}_0 $ that we introduce in the sequel.
\end{remark}
The following proposition will be useful later.

\begin{proposition}\label{prop:comDdelta_0}
    One has for all $p\in\mathbb{N}^{d+1}$
    \begin{equs}\label{eq:comDdelta_0}
        \left(\CD_{p}\otimes\Id\right)\Delta_0 = \Delta_0\CD_{p}.
    \end{equs}
\end{proposition}
\begin{proof}
The proposition is easily confirmed for $\{\one,\Xi_0,X_i\}$ via elementary computations. The next thing to notice is that the expressions on either side of \eqref{eq:comDdelta_0} are multiplicative:
\begin{equs}
\Delta_0\CD_{p}(\tau\tau') = \left(\CD_{p}\otimes\Id\right)\Delta_0(\tau\tau') &=\left(\CD_{p}\otimes\Id\right)\Delta_0(\tau)\left(\CD_{p}\otimes\Id\right)\Delta_0(\tau') \\
 &= \Delta_0\CD_{p}(\tau)\Delta_0\CD_{p}(\tau')  .
\end{equs}
From here one can extend the result to $T$ once one has the result for planted trees. For that one checks:
\begin{equs}
	\left(\CD_{p}\otimes\Id\right)\Delta_0\CI_\alpha\tau &= \left(\CD_{p}\otimes\Id\right)\Biggl[(\mcI_\alpha\otimes \Id )\Delta_0 \tau \\
	&\hphantom{aaaaaa}+\hspace{-3mm}\sum_{\vert\ell +m\vert < \textsf{deg}_0(\mcI_\alpha\tau)} \frac{X^\ell}{\ell!}\otimes \frac{X^m}{m!}\mcI^{+,0}_{\alpha+\ell+m}(\tau)\Biggr] \\
	&= (\CI_{(\alpha + p)}\otimes\Id)\Delta_0\tau \\
	&\hphantom{aaaaaa}+\hspace{-3mm}\sum_{\vert\ell +m\vert < \textsf{deg}_0(\mcI_\alpha\tau)} \frac{X^{\ell-p}}{(\ell-p)!}\otimes \frac{X^m}{m!}\mcI^{+,0}_{\alpha+p+\ell-p+m}(\tau)
	\\
	&\overset{(\ell'\coloneqq\ell - p)}{=} (\CI_{(\alpha + p)}\otimes\Id)\Delta_0\tau \\
	&\hphantom{aaaaaa}+\hspace{-3mm}\sum_{\vert\ell' +m\vert < \textsf{deg}_0(\mcI_{(\alpha+p)}\tau)} \frac{X^{\ell'}}{(\ell')!}\otimes \frac{X^m}{m!}\mcI^{+,0}_{\alpha+p+\ell'+m}(\tau) \\
	&=\Delta_0\CI_{(\alpha+p)}\tau = \Delta_0\CD_p\CI_\alpha\tau.
\end{equs}
\end{proof}
The next component needed in our construction of the aforementioned quantities is that of a preparation map $ R $ for the degree $  \textsf{deg}_i $. Such maps are fundamental integrants in the construction of renormalisation maps, and as such their utility here, should not come as a surprise. We recall the relevant definition from \cite{BR18}, where it was first introduced.
\begin{definition}
 \label{DefnPreparationMap}
A \textbf{preparation map} is a linear map $
R : \CT \rightarrow \CT $ 
that fixes polynomials, noises, planted trees, and such that 
\begin{itemize}
   \item for each $ \tau \in \mathcal{T} $ there exist finitely many $\tau_j \in \mathcal{T}$ and constants $\lambda_j$ such that
\begin{equation} \label{EqAnalytical}
R \tau = \tau + \sum_j \lambda_j \tau_j, \quad\textrm{with}\quad \textsf{deg}_1(\tau_j) \geq \textsf{deg}_1(\tau) \quad\textrm{and}\quad |\tau_j|_{\Xi} < |\tau|_{\Xi},
\end{equation} 
   \item one has 
 \begin{equation} \label{EqCommutationRDelta}
 ( R \otimes \Id) \Delta_i = \Delta_i R.
 \end{equation}
 \item The map $ R $ commutes with $ D_{\Xi} $ in the sense that one has on $ \mathcal{T}_0 $:
 \begin{equs} \label{commute_D_Xi}
 R D_{\Xi} = D_{\Xi} R.
 \end{equs}
\end{itemize}
\end{definition}

The commutation property \eqref{commute_D_Xi} is where we differ from the original definition in the literature. It is meant to reflect the fact that by changing a noise $ \xi $ in some tree into an infinitesimal perturbation of itself $ \delta \xi $, it becomes less singular and therefore does not require any renormalisation. 
One natural choice for the preparation map $ R $ is given in \cite{BR18,BB21} by:
\begin{equs}\label{eq:prepmap}
R_{\ell}^{*} \tau =  \sum_{\sigma \in T^-_0} \frac{\ell(\sigma)}{S(\sigma)} \sigma \star \tau,
\end{equs}
where $ S(\sigma) $ is a symmetry factor associated to $ \sigma $, $ T_{0}^-  $ are elements in $ T_{0} $ with negative degree using $ \textsf{deg}_0 $, the associative product $ \star $ is related to the dual of $ \Delta_i $ and $ \ell : T_0^{-} \rightarrow \mathbb{R} $, which provides the renormalisation constants, is chosen in such a manner so that $ R_{\ell}$ has the various properties of a preparation map. The map $ R_{\ell} $ whose dual map is $ R_{\ell}^{*} $ has been first introduced in \cite{BR18} as an extraction/contraction procedure of a subtree with negative degree happening at the root.
With a preparation map $R$ fixed, one is able to define the corresponding renormalised model $ \Pi^{R,i}_x $ for $ i \in \lbrace 0,1 \rbrace $ inductively as follows:
\begin{alignat}{2} \label{recursive_model}
&\Pi_x^{R,i} \tau  = \hat{ \Pi}^{R,i}_x R \tau, &&\hat{{ \Pi}}^{R,i}_x (\tau \bar \tau) = (\hat{{ \Pi}}^{R,i}_x \tau) \, (\hat{{\Pi}}^{\!R,i}_x\bar \tau), \nonumber \\ 
&\big(\hat { \Pi}_x^{R,i}(\mathcal{I}_a\tau)\big)(y)= \big(D^{a} K &&* { \Pi}_x^{R,i} \tau\big)(y)\nonumber  \\ &  &&- \sum_{|k| < \textsf{deg}_i(\mathcal{I}_a\tau) }\frac{(y-x)^k}{k!}  \big(D^{a +k}  K * { \Pi}_x^{R,i} \tau\big)(x),
\end{alignat}
where $ K $ is now a compactly support function obtained from the kernel associated with the singular SPDEs considered, with the seed given by:

\begin{equs}
\Pi_x^{R,0}\Xi_0 = \xi,\quad&\Pi_x^{R,1}\Xi_0=\xi,\quad\Pi_{x}^{R,1}\Xi_1=\delta\xi,\quad\Pi_x^{R,i}X^{k} = (\cdot - x)^{k}, \\
&\hat{\Pi}^{R,i}_{x}\tau = \Pi_x^{R,i}\tau,\quad\text{for }\tau\in\{\Xi_0,\Xi_1\}\cup\bar{T}.
\end{equs}
\color{black}
This recursive definition was first introduced in \cite{BR18}. The main idea is that one needs to cure first a local divergence with $ R $. Then by multiplicativity one iterates $ R $ deeper inside the iterated integral.
Notice that the models depend on the choice of the degree $ \textsf{deg}_i $ so a difference manifests in the length of the Taylor expansions, in particular when the noise $ \Xi_1 $ comes into account, one gets longer Taylor expansions.
One has the following analytical bound:
\begin{equs}
	\left|\left(\Pi^{R,i}_x \tau \right)(y)\right| \lesssim |y-x|^{\textsf{deg}_i(\tau)},
\end{equs}
that depends on the choice of degree. The proof is verbatim the same as the one given in \cite{BR18}.

\begin{remark}\label{rem:derivativecommutativity}
From \eqref{recursive_model}, one learns that the multi-index subscript $a$ in $\mathcal{I}_a(\tau)$ effectively keeps a count of the derivatives on kernel on the abstract side. In light of this, and the distributivity of derivatives on convolutions, we have the commutativity property:
\begin{equs}
\hat\Pi^{R,i}_{x}\mathcal{I}_a(\tau)=\hat\Pi^{R,i}_{x}\CD^\alpha\mathcal{I}(\tau)=\partial^{a}\hat\Pi^{R,i}_{x}\mathcal{I}(\tau),
\end{equs}
where the second equality is obvious from the properties of the derivative and \eqref{recursive_model}. Furthermore, due to invariance of the preparation map on planted trees, we have that
\begin{equs}
\Pi^{R,i}_{x}\mathcal{I}_a(\tau)&=\hat\Pi^{R,i}_{x}R\mathcal{I}_a(\tau) = \hat\Pi^{R,i}_{x}\mathcal{I}_a(\tau)\\
&= \partial^{a}\hat\Pi^{R,i}_{x}\mathcal{I}(\tau)\label{eq:derivativecommutativity} = \partial^{a}\hat\Pi^{R,i}_{x}R\mathcal{I}(\tau) = \partial^{a}\Pi^{R,i}_{x}\mathcal{I}(\tau).
\end{equs}

\end{remark}
\begin{remark} \label{invariant_x}
The use of a preparation map also allows us to treat the case of a local renormalisation meaning that we are able to construct the renormalisation for a model based on kernels $ K $ that are non-translation invariant. This was first noticed in \cite{BB21b} where one can define a preparation map $ R(\cdot) $ depending on a given space-time point and check the same properties listed in Definition~\ref{DefnPreparationMap}. In this case, we define the renormalised model as:
\begin{equs}
(\Pi_x^{R,i} \tau)(y)  = (\hat{ \Pi}^{R,i}_x R(y) \tau)(y), \quad (\PPi^R \tau)(y)  = (\hat{\PPi} \,  R(y) \tau)(y).
\end{equs}
The only difference with the translation invariant case is the replacement of $ R $ by $ R(y) $. In the sequel, all the algebraic identities are robust to this new formulation. Therefore, the identities that we derive in the end are similar in this framework.
\end{remark}

An alternate construction of the renormalised model is due to the relationship between the (renormalised) model $\Pi^{R,i}_{x}$ and the pre-model $\PPi^R$, and the manner in which $\Delta_i$ are defined:
\begin{equs}\label{eq:equality1}
\Pi_x^{R,i} = \left( \PPi^R \otimes f_x^{R,i} \right) \Delta_i,
\end{equs}
where $ \PPi^R $ the pre-model and $f_x^{R,i}$ are defined by the following recursive formulae:
\begin{equs} \label{def_pre_model}
\PPi^R \tau &= \hat{\PPi}^R \,  R \tau, \quad \hat{\PPi}^{R} \tau \bar{\tau} = \hat{\PPi}^{R} \tau  \, \hat{\PPi}^{R} \bar{\tau}, \\
&\big(\hat{ \PPi}^R(\mathcal{I}_a\tau)\big) = \big(D^{a} K * {\PPi}^{R} \tau\big).
\end{equs}
with the seed
\begin{equs}
\PPi^{R}\Xi_0 = \hat\PPi^{R}\Xi_0 &= \xi,\quad\PPi^{R}\Xi_1 = \hat\PPi^{R}\Xi_1 = \delta\xi, \\
&\PPi^R X^{k} = \hat\PPi^{R} X^k= \bullet^{k},
\end{equs} \color{black}
 and
\begin{equation}\label{eq:recf}\begin{split}
f_x^{R,i}(\mathbf{1})&=1,\qquad f_x^{R,i}(X_j) = -x_j\mbox{ for }j\in\{0,\hdots,d\},\\
&f_x^{R,i}(\CI_a^{+,i}(\tau))= - (D^a K * (\Pi_x^{R,i} \tau))(x),
\end{split}
\end{equation}
respectively, and then are extended multiplicatively. The reader may refer to \cite[Rem. 8.31]{reg} for a more thorough explanation.

\begin{remark}
It should be noted that a model in the regularity structure paradigm is a pair of continuous linear maps $(\Pi,\Gamma)$ where the components satisfy certain axioms (see \cite[Def. 2.1]{reg}) including $ \Pi_y = \Pi_x\Gamma_{yx}$. $\Pi^{R,i}$ is a renormalised \textit{model} exactly in the sense of being the first component in a pair that satisfies the previously cited axioms. As the exact construction of the corresponding $\Gamma^{R,i}_{xy}$ is not necessary we do not expend any energy in its description, but we do note the utility of the coactions $\Delta_i$ in its definition:
\begin{equs}\label{eq:gammacoaction}
\Gamma^{R,i}_{yx} = (\Id\otimes\gamma^{R,i}_{yx})\Delta_i,
\end{equs}
where $\gamma_{yx}^{R,i}$ is a character on a particular collection of trees. For a complete exposition, we refer to the discussion surrounding \cite[Eq. 8.17]{reg}.
\end{remark}
\begin{remark}
Notice that \eqref{eq:gammacoaction} suggests that the commutativity that the coactions enjoy with the preparation map, filters through to commutativity with $\Gamma^{R,i}$. Indeed we have due to \eqref{EqCommutationRDelta}
\begin{equs}\label{eq:Gammacommute}
\Gamma^{R,i}_{yx}R &= (\Id\otimes\gamma^{R,i}_{yx})\Delta_i R \\
&= (\Id\otimes\gamma^{R,i}_{yx})(R\otimes\Id)\Delta_i \\
&= (R\otimes\Id)(\Id\otimes\gamma^{R,i}_{yx})\Delta_i \\
& = R\Gamma^{R,i}_{yx}.
\end{equs}
It is easy to see that this commutativity and $\Pi^{R,i}_y = \Pi^{R,i}_x\Gamma^{R,i}_{yx}$ gives us $\hat\Pi^{R,i}_y = \hat\Pi^{R,i}_x\Gamma^{R,i}_{yx}$.
\end{remark}

Given that the definitions of $\Pi_x^{R,1}$, $\Pi_x^{R,0}$ differ only in the length of Taylor expansions, it is natural to suspect the existence of some mapping that connects $\Pi_x^{R,1}$, $\Pi_x^{R,0}$. In the following proposition, we show how a linear coaction $ \hat{\Delta}_0:\CT\rightarrow\CT\otimes\CT^{+}$ from the renormalisation literature can be co-opted to produce a curtailing procedure that achieves this.
\begin{proposition}\label{prop:magicformula} The following identity holds on $ \mathcal{T}_1 $
\begin{equation} \label{magicformula}
\Pi_x^{R,1} = \left( \Pi_x^{R,0}\otimes f_x^{R,0} \right) \hat{\Delta}_0,
\end{equation}
for the map $\hat{\Delta}_0$ defined recursively as follows:
	\begin{align}
\hat{\Delta}_{0} (\bullet) &\coloneqq \bullet\otimes \mathbf{1}, \quad \textrm{ for }\bullet\in\big\{\mathbf{1}, X_j, \Xi_i\big\} \label{EqDefnDeltaHatCirc}, \\
\hat{\Delta}_{0} (\mcI_a\tau) &\coloneqq (\mcI_a\otimes\textnormal{Id})\hat{\Delta}_{0} (\tau) \nonumber\\
   &- \sum_{\textsf{deg}_1(\mcI_a\tau) \leq \vert\ell\vert} \frac{X^\ell}{\ell!}\otimes\mcM^+\big(\mcI^{+,0}_{a+\ell}\otimes\textnormal{Id}\big)\hat{\Delta}_{0} (\tau), \label{EqDefnDeltaHatCirc}
   \end{align}
\noindent for $i\in\{0,1\}$ and $j\in\{0,\hdots,d\}$, where $ \mathcal{M}^+ $ is the multiplication map on $ \CT^+ $ and finally extended multiplicativity to all decorated trees in $ \mathcal{T}_1 $.
\end{proposition}
To prove this proposition, we have to collect a lemma first, and to prove this requisite lemma, we need the following result which explains the relationship between $\Delta_0$ and $\hat\Delta_0$.
\begin{proposition}\label{lem:delhatdel} There exists a linear, multiplicative map $\hat\Gamma_0$ on the decorated trees such that
\begin{equation}\label{eq:delhatdel}
\hat \Delta_0 = (\textnormal{Id}\otimes\hat\Gamma_0)\Delta_0.
\end{equation}
\end{proposition}
\begin{proof}
We begin with the following limited characterisation of $\hat \Gamma_0$,
\begin{equs}
\hat{\Gamma}_0 X_i \coloneqq \mathbf{0},\quad\hat\Gamma_0 \tau = \tau,\text{ for }\tau\in\{\one,\Xi_0,\Xi_1\}.
\end{equs}
It is easy to see then that \eqref{eq:delhatdel} holds true for $\tau\in\{\mathbf{1},X_i,\Xi_0,\Xi_1\}$. We then extend the definition above by
\begin{equs}
&\hat{\Gamma}_0 \left( \sum_{\ell} \frac{X^{\ell}}{\ell !} \mathcal{I}^{+,0}_{a + \ell}(\tau) \right) =  - \mathbf{1}_{\lbrace \textsf{deg}_1(\mcI_a\tau) \le 0 \rbrace} \mcM^+\big(\mcI^{+,0}_{a}\otimes\Id\big)\hat{\Delta}_{0} (\tau).
\end{equs}
In particular the above definition tells us how $\hat\Gamma_0$ maps $\mcI_a^{+,0}(\tau)$, via the change of base
\begin{equation} \label{change_bais}
\sum_{\ell} \frac{X^{\ell}}{\ell !} \mathcal{I}^{+,0}_{a + \ell}(\tau) = \tilde{\mathcal{I}}^{+,0}_{a}(\tau).
\end{equation}
Finally, we extend $\hat\Gamma_0$ multiplicatively to all of $\mathcal{T}_0$. This multiplicative property of $\hat\Gamma_0$ is crucial because it endows the right hand side of \eqref{eq:delhatdel} with multiplicativity.

Now the result can be achieved by induction on the size of the trees. We have already checked the base case on the trees $\{\mathbf{1},\Xi_0,\Xi_1,X_i\}$ and due to the multiplicativity discussed above, we need only to check the equality on $\mcI_a(\tau)$ for $\tau\in \mathcal{T}_0$. To that end we compute
\begin{equs}
\left( \Id \otimes \hat{\Gamma}_0 \right) \Delta_0 \CI_a(\tau) & =  
 (\mcI_a\otimes\Id)(\Id\otimes \hat{\Gamma}_0 )\Delta_{0}  \tau   \\ & + \left( \id \otimes \hat{\Gamma}_0 \right)  \sum_{\vert\ell +m\vert < \textsf{deg}_0(\mcI_a\tau)} \frac{X^\ell}{\ell!}\otimes \frac{X^m}{m!}\mcI^{+,0}_{a+\ell+m}(\tau).
\end{equs}
We can conclude by applying the induction hypothesis on the first term on the right hand side and the definition of $ \hat{\Gamma}_0 $ on the second term.
\end{proof}

\begin{remark}
The algebraic property given by Proposition~\ref{lem:delhatdel} is quite remarkable as it says that $ \hat{\Delta}_0 $ is of the same nature as $ \Delta_0 $ which is a variant of the deformed Butcher-Connes-Kreimer coproduct. This shows how powerful this formalism is for encoding various analytical operations. This property has been overlooked in \cite{BR18} as one can write a similar formula for $ \Delta^{M^{\circ}} $. Indeed, one has 
\begin{equs}
\Delta^{\!M^{\circ}} (\mcI_a\tau)   &= (\mcI_a\otimes\Id) \Delta^{\!M}  \tau - \hspace{-3mm}\sum_{\textsf{deg}_0(\mcI_a\tau) \leq \vert\ell\vert} \frac{X^\ell}{\ell!}\otimes\mcM^+\big(\mcI^{+}_{a+\ell}\otimes\Id\big)\Delta^{\!M}  \tau,
\\ 
\Delta^{\!M}  \tau & = \Delta^{\!M^{\circ}} R \tau.
\end{equs}
From this we can write:
\begin{equs} \label{new_formula_delta_M}
\Delta^{\!M^{\circ}} = ( M^{\circ} \otimes\hat\Gamma_M) \Delta_0, \quad    
\end{equs}
where 
\begin{equs}
&\hat{\Gamma}_M \left( \sum_{\ell} \frac{X^{\ell}}{\ell !} \mathcal{I}^{+,0}_{a + \ell}(\tau) \right) = -\mathbf{1}_{\lbrace \textsf{deg}_0(\mcI_a\tau) \le 0 \rbrace} \mcM^+\big(\mcI^{+,0}_{a}\otimes\Id\big)\Delta^{\!M} (\tau).
\end{equs}
Indeed, one has:
\begin{equs}
( M^{\circ} \otimes\hat\Gamma_M) \Delta_0 \mathcal{I}_{a}(\tau) &  =  (\mcI_a\otimes\textnormal{Id})(  M^{\circ} R \otimes \hat{\Gamma}_M )\Delta_{0}  \tau   \\ & + \left( \id \otimes \hat{\Gamma}_M \right)  \sum_{\vert\ell +m\vert < \textsf{deg}_0(\mcI_a\tau)} \frac{X^\ell}{\ell!}\otimes \frac{X^m}{m!}\mcI^{+,0}_{a+\ell+m}(\tau).
\end{equs}
We conclude by using the induction hypothesis and the fact that $ R $ commutes with $ \Delta_0 $:
\begin{equs}
(  M^{\circ} R \otimes \hat{\Gamma}_M )\Delta_{0}  = (  M^{\circ}  \otimes \hat{\Gamma}_M )\Delta_{0} R = \Delta^{\! M^{\circ}} R. 
\end{equs}
\end{remark}
Given the commutativity that $ R $ enjoys with $\Delta_0$ and the result we have connecting $ \hat{\Delta}_0 $ to $\Delta_0$, one might suspect that the commutativity of $R$ extends to $\hat\Delta_0$. The following lemma, which is needed in the proof of Proposition~\ref{prop:magicformula}, proves this suspicion to be correct.
\begin{lemma}\label{lem:req2} One has 
\begin{equs} \label{commutation_Delta1_R} 
\left(R \otimes \id \right) \hat{\Delta}_0 = \hat{\Delta}_0 R.
\end{equs}
\end{lemma}
\begin{proof} This is an easy consequence of \eqref{eq:delhatdel} and \eqref{EqCommutationRDelta}. Indeed one has
\begin{equation*}\begin{split}
\left(R\otimes \Id\right)\hat\Delta_0 &= \left(R\otimes \Id\right)\left( \id \otimes \hat{\Gamma}_0 \right) \Delta_0 \\
&= \left( \id \otimes \hat{\Gamma}_0 \right)\left(R\otimes \Id\right) \Delta_0 \\
&= \left( \id \otimes \hat{\Gamma}_0 \right) \Delta_0 R \\
&= \hat{\Delta}_0 R.
\end{split}
\end{equation*}
\end{proof}
\begin{remark}
The commutation property between $ R $ and $ \Delta_0 $ provides a new formula for $ \Delta^{\!M} $ coming from \eqref{new_formula_delta_M}
\begin{equs}
\Delta^{\! M} = \Delta^{\! M^{\circ}} R = \left( M^{\circ} \otimes \hat{\Gamma}_M \right) \Delta_0 R = \left( M^{\circ} R \otimes \hat{\Gamma}_M \right) \Delta_0  = \left( M \otimes \hat{\Gamma}_M \right) \Delta_0. 
\end{equs}
\end{remark}
We are now in position to be able to prove the main result of this section.
\begin{proof}[of Proposition~\ref{prop:magicformula}]
For clarity of exposition, let us explain how this proof differs from other of its sort in this article. As before we follow the programme of establishing the result on the base symbols, followed by the multiplicity of the results, from where the result comes inductively once it has been established for planted trees. What is novel in this particular proof is that due to the way $\Pi^{R,i}_x$ and $\hat\Pi^{R,i}_x$ are intertwined, one has to concurrently establish \eqref{magicformula}, and:
\begin{equs}
\hat\Pi_x^{R,1} = \left( \hat\Pi_x^{R,0}\otimes f_x^{R,0} \right) \hat{\Delta}_0,
\end{equs}
using one to prove the other. The base case is settled by fixing first $\tau\in\{\mathbf{1},X_j,\Xi_i\}$ for $i\in\{0,1\},j\in\{0,1,\hdots,d\}$ and noticing that:
\begin{equation*}\begin{split}
\left( \hat\Pi_x^{R,0}\otimes f_x^{R,0} \right) \hat{\Delta}_0\tau &= \left( \hat\Pi_x^{R,0}\otimes f_x^{R,0} \right)(\tau\otimes\mathbf{1})\\
 &= \hspace{2mm}\hat\Pi_x^{R,0}(\tau)\otimes f_x^{R,0}(\mathbf{1})\\
 &= \hspace{2mm}\hat\Pi_x^{R,0}(\tau) \\
 &= \hspace{2mm}\hat\Pi_x^{R,1}(\tau),
\end{split}
\end{equation*}
where in the second equality we used $f_x^{R,0}(\mathbf{1})=1$, in the third we have identified the tensor product with a scalar, with multiplication, and finally that $\hat\Pi^{R,i}$ agree on this choice of $\tau$. Futhermore, as $R$ is invariant on this choice of $\tau$, we also have
\begin{equs}\left(\Pi^{R,0}_{x}\otimes f_x^{R,0}\right)\hat\Delta_0\tau = {\Pi}^{R,1}_{x}(\tau).
\end{equs}
For multiplicativity we take $\tau,\bar\tau\in\{\one,X_j,\Xi_i\}$ with $i\in\{0,1\}$ and $j\in\{0,1,...,d\}$ and check that
\begin{equs}
\hat{\Pi}^{R,1}_{x}(\tau\bar\tau) =\hat{\Pi}^{R,1}_{x}(\tau) \hat{\Pi}^{R,1}_{x}(\bar\tau) &= \left(\hat\Pi^{R,0}_{x}\otimes f_x^{R,0}\right)\hat\Delta_0\tau\left(\hat\Pi^{R,0}_{x}\otimes f_x^{R,0}\right)\hat\Delta_0\bar\tau \\
&=\left(\hat\Pi^{R,0}_{x}\otimes f_x^{R,0}\right)\hat\Delta_0\left(\tau\bar\tau\right),
\end{equs}
where the last inequality is due to the fact that all the maps considered are multiplicative. It only remains to contend with planted trees $\mathcal{I}_{a}(\tau)$.
 \begin{equation*}
 \begin{aligned}
&  \left( \hat{\Pi}_x^{R,0} \otimes f_x^{R,0} \right) \hat{\Delta}_{0} \mathcal{I}_a(\tau)  = \left( \hat{\Pi}_x^{R,0} \otimes f_x^{R,0} \right)(\mcI_a\otimes\Id)\hat{\Delta}_{0} \tau \\ &  -  \sum_{ \textsf{deg}_1(\mcI_a\tau) \leq \vert\ell\vert}\left( \hat{\Pi}_x^{R,0} \otimes f_x^{R,0} \right) \left( \frac{X^\ell}{\ell!}\otimes\mcM^+\big(\mcI^{+,0}_{a+\ell}\otimes\Id\big)\hat{\Delta}_{0} \tau \right).
 \end{aligned}
 \end{equation*}
\noindent The second term on the right hand side above, contributes:
 \begin{equs}
 - \sum_{ |\ell| \geq \textsf{deg}_1(\mcI_a\tau) }\frac{(\cdot-x)^\ell}{\ell!}   \left( f_x^{R,0} \CI^{+,0}_{a+\ell}  \otimes f_x^{R,0} \right) \hat{\Delta}_0 \tau.
 \end{equs}
Whereas the first term on the right hand side, owing to the upper triangular structure of $\hat{\Delta}_0$, gives:
 \begin{equs}
 & \left( \hat{\Pi}_x^{R,0} \otimes f_x^{R,0} \right)(\mcI_a\otimes\Id)\hat{\Delta}_{0} \tau = D^a K *  \left(\Pi_x^{R,0} \otimes f_x^{R,0}  \right) \hat{\Delta}_0 \tau 
 \\ &  + \sum_{ |\ell| < \textsf{deg}_0(\mcI_a\tau) }\frac{(\cdot-x)^\ell}{\ell!}   \left( f_x^{R,0} \CI^{+,0}_{a+\ell}  \otimes f_x^{R,0} \right) \hat{\Delta}_0 \tau.
 \end{equs}	
 By the induction hypothesis, one has:
 \begin{equs}
  D^a K *  \left(\Pi_x^{R,0} \otimes f_x^{R,0}  \right) \hat{\Delta}_0 \tau  = D^a K * \Pi_x^{R,1} \tau.
 \end{equs}
Also using the induction hypothesis in conjunction with the upper triangular structure of the map $\hat{\Delta}_0$, we see that:
 \begin{equs}
 & \sum_{ |\ell| \geq \textsf{deg}_1(\mcI_a\tau) }\frac{(\cdot-x)^\ell}{\ell!}   \left( f_x^{R,0} \CI^{+,0}_{a+\ell}  \otimes f_x^{R,0} \right) \hat{\Delta}_0 \tau \\ & = -  \sum_{ |\ell| \geq \textsf{deg}_1(\mcI_a\tau) }\frac{(\cdot-x)^\ell}{\ell!}   
 \left( \left(D^{a + \ell} K * \Pi_x^{R,0}\tau\right) (x) \right),
 \end{equs}
and similarly for the summation over $|\ell| < \textsf{deg}_0(\mcI_a\tau)$. By construction we know that $\textsf{deg}_0(\sigma) \ge \textsf{deg}_1(\sigma)$ which implies $\sum_{|\ell| <\textsf{deg}_0(\sigma)} (\bullet) - \sum_{|\ell| \geq\textsf{deg}_1(\sigma)} (\bullet) = \sum_{|\ell| <\textsf{deg}_1(\sigma)} (\bullet) $ for any $\sigma\in\mathcal{T}$. Putting all of these observations together, we get:
\begin{equation*}\begin{split}
\left( \hat{\Pi}_x^{R,0} \otimes f_x^{R,0} \right) \hat{\Delta}_{0} \mathcal{I}_a(\tau) &= \big(D^{a} K * { \Pi}_x^{R,1} \tau\big) \\ & - \sum_{|k| < \textsf{deg}_1(\mathcal{I}_a\tau) }\frac{(\cdot-x)^k}{k!}  \left(\left(D^{a +k}  K * { \Pi}_x^{R,1} \tau\right)(x)\right) \\
&= \hat{\Pi}_x^{R,1}(\mathcal{I}_a(\tau)).
\end{split}
\end{equation*}
To complete the induction, we argue on the strength of Lemma~\ref{lem:req2} as follows:
\begin{equs}\begin{split}
\Pi_x^{R,1}=\hat{\Pi}_x^{R,1} R &= \left( \hat{\Pi}_x^{R,0} \otimes f_x^{R,0} \right) \hat{\Delta}_{0}R
\\ &=\left( \hat{\Pi}_x^{R,0} \otimes f_x^{R,0} \right)(R \otimes \text{Id}) \hat{\Delta}_{0}
\\ &=\left( \hat{\Pi}_x^{R,0}R \otimes f_x^{R,0} \right)\hat{\Delta}_{0}
\\ &=\left( {\Pi}_x^{R,0} \otimes f_x^{R,0} \right)\hat{\Delta}_{0}. 
\end{split}
\end{equs}
\end{proof}
\begin{remark} Such factorisation was initiated in \cite{reg} for finding a map $ \Delta^M $ given a renormalisation map $ M $ such that:
\begin{equs}
\Pi_x^M = \left( \Pi_x\otimes f_x \right) \Delta^{\!M}.
\end{equs}
This map exists for any linear map $ M : \mathcal{T}_0 \rightarrow \mathcal{T}_0 $ that can be understood as choosing a renormalisation map. The difficulty was to show that given a map $ M $, the map $ \Delta^{\!M} $ is upper triangular. This property guarantees that we get the bounds for a model in the end.
As one has to construct this map by hand, \cite{BHZ} proposes a different route by introducing extended decorations that allow one to write directly a formula of the type:
\begin{equs}
\Pi_x^M = \Pi_x M,
\end{equs}
which is not true in general with the normal setting. It was in \cite{BR18} where a recursive formula is given for $ \Delta^{\! M} $ when $M$ is constructed via a preparation map in the following way:
\begin{equs} \label{recursive_M}
M = M^{\circ} R, \quad M^{\circ} \tau \bar{\tau} = M^{\circ} \tau M^{\circ} \bar{\tau}, \quad M^{\circ} \mathcal{I}_{a}(\tau) 
= \mathcal{I}_a(M \tau).
\end{equs}
One defines $ \Delta^{\!M} $ as in \cite{BR18}:
\begin{equs} \label{EqDefnDeltaMCirc_M} 
&\Delta^{\!M^{\circ}} (\bullet) \coloneqq \bullet\otimes \mathbf{1}, \quad \textrm{ for }\bullet\in\big\{\mathbf{1}, X_i, \Xi_i\big\},   \\
&\Delta^{\!M^{\circ}} (\mcI_a\tau) \coloneqq (\mcI_a\otimes\Id)\Delta^{\!M} (\tau)\hspace{3mm} - \sum_{\textsf{deg}_0(\mcI_a\tau) \leq \vert\ell\vert} \frac{X^\ell}{\ell!}\otimes\mcM^+\big(\mcI^{+,0}_{a+\ell}\otimes\Id\big)\Delta^{M} (\tau),
\\ &\Delta^{\!M} \coloneqq \Delta^{\!M^{\circ}} R, \quad 
\Delta^{\!M^{\circ}} \tau \bar \tau = \left( \Delta^{\!M^{\circ}} \tau \right)\left( \Delta^{\!M^{\circ}} \bar \tau \right).
\end{equs}

Let us mention that the recursive formula \eqref{recursive_M} of $ M $  does not make sense when $ R $ depends on a space-time point. One has to use instead the notation $ \Delta^{\! R} $ (resp. $ \hat{\Delta}^{R} $) instead of $ \Delta^{\! M} $ (resp. $ \Delta^{M^{\circ}} $). The rest of the definition remains the same. The recursive formulation is a consequence of shortening Taylor expansions in the context of renormalisation. In principle, renormalisation does not commute with recentring as $ R $ applied to $ \tau  $ produces a sum of trees with a higher degree. Therefore, it is not surprising that if such a map can be used in the context of Malliavin derivatives, it should increase the degree of the noise.
\end{remark}
One sees from the definition that like $\Delta_i$, the map $ \hat{\Delta}_0 $ is upper triangular in the sense that: 
\begin{equs}\label{eq:hatdelta0tri}
\hat{\Delta}_0 \tau - \tau \otimes \mathbf{1} = \sum_{\widehat{(\tau)}} \hat\tau^{(1)} \otimes \hat\tau^{(2)}, \quad \textsf{deg}_1(\hat\tau^{(1)}) \geq \textsf{deg}_1(\tau),
\end{equs}
where the sweedler notation above, is chosen so to avoid confusion with \eqref{eq:sweedlernotation}. Contrasting this triangular structure with \eqref{EqDefnDeltaHatCirc} suggests that the effect of the co-action on $\tau\in\mathcal{T}_1$, is such that $\Xi_1$ always ends up in $\tau^{(2)}$ except of course for $\tau\otimes\mathbf{1}$. This motivates the following result with which we conclude this section:
\begin{proposition}
One has for every $ \tau \in \mathcal{T}_0 $:
\begin{equs} \label{id_T_0}
\hat{\Delta}_0 \tau = \tau \otimes \mathbf{1}.
\end{equs}
\end{proposition}
\begin{proof}
One proceeds by induction. The base case is trivial. Furthermore, as a consequence of the multiplicativity of $\hat\Delta_0$, one only needs to consider $ \mathcal{I}_a(\tau) \in \mathcal{T}_0 $. Then assuming $\eqref{id_T_0}$ true for $\tau$:
\begin{equs}
\hat{\Delta}_{0} (\mcI_a\tau) & = (\mcI_a\otimes\Id)\hat{\Delta}_{0} (\tau)  - \sum_{\textsf{deg}_1(\mcI_a\tau) \leq \vert\ell\vert} \frac{X^\ell}{\ell!}\otimes\mcM^+\big(\mcI^{+,0}_{a+\ell}\otimes\Id\big)\hat{\Delta}_{0} (\tau)
\\ & = \mcI_a(\tau)\otimes \mathbf{1}  
- \sum_{\textsf{deg}_1(\mcI_a\tau) \leq \vert\ell\vert} \frac{X^\ell}{\ell!}\otimes\mcM^+\big(\mcI^{+,0}_{a+\ell}(\tau)\otimes  \mathbf{1} \big)
\\ & = \mcI_a(\tau)\otimes \mathbf{1} - \sum_{\textsf{deg}_1(\mcI_a\tau) \leq \vert\ell\vert} \frac{X^\ell}{\ell!}\otimes \mcI^{+,0}_{a+\ell}(\tau)
\\ & =  \mcI_a(\tau)\otimes \mathbf{1},
\end{equs}
where we have used the induction hypothesis and the last line comes from the fact that $ \textsf{deg}_1(\mcI_a\tau) = \textsf{deg}_0(\mcI_a\tau) $ because $ \tau \in \mathcal{T}_0 $.
\end{proof}	
\section{Algebraic identities}\label{sec:AlgIde}
The purpose of this section is to employ the constructions in the previous section to relate the combinatorial approach of our decorated trees to that of the fundamentally analytical multi-index approach of \cite{LOTP}. On the stochastic estimates end, this analyticity is manifest most prominently in their use of the spectral gap inequality to reduce the control of $p$-th moment of $\Pi_x^{-}$, to the control of the $p$-th moment of its Malliavin derivative and absolute control of its first moment. We note that the control of the latter, in either of the approaches, is due to BPHZ renormalisation and it is the former term that is the actual novelty here. To define the derivative, the authors in \cite{LOTP} use the standard definition in the Malliavin Calculus literature for cylindrical functionals of the noise
$F[\xi]\coloneqq\bar{F}((\xi,\eta_1),\hdots,(\xi,\eta_N))$, i.e. the Malliavin derivative is the Fr\'{e}chet Derivative in the noise:
\begin{equs}\label{eq:malderiv}\frac{\partial F}{\partial\xi}[\xi]=\sum_{n=1}^N\partial_n\bar{F}((\xi,\eta_1),\hdots,(\xi,\eta_N))\eta_n.
\end{equs}
Pursuant to our goal of reconciling the two approaches, we first contend with the question of how to relate the analytical object $\delta\Pi^{R,1}_{x}$ to our decorated tree based approach. For this we may think of the Fr\'{e}chet derivative \eqref{eq:malderiv} as a directional derivative in some arbitrary infinitesimal perturbation $\delta\xi$, which owing to the linearity of $\Pi^{R,1}_{x}$ is equivalent to exchanging some instance of $\xi$ with that of $\delta\xi$. That is what $D_\Xi$ does on the level of trees and hence we have the following result.

\begin{proposition}
\label{malliavin_derivative_model1}
One has on $\mathcal{T}_0 $
\begin{equs}\label{eq:malliavin_derivative_model1}
\delta\Pi_x^{R,1} = \Pi_x^{R,1} D_{\Xi}.
\end{equs}
\end{proposition}
\begin{proof} One can proceed by induction on the size of the decorated trees using the recursive formula of the model. On the basic symbols $\{\mathbf{1},\Xi_0,X_i\}$ the property is seen to be true. Suppose then we have two trees $\tau,\,\bar\tau$ that satisfy the induction hypothesis; we check that
\begin{equs}
\delta \left(  \hat{\Pi}^{R,1} (\tau \bar{\tau}) \right) &= (\delta \hat{\Pi}^{R,1} \tau)( \hat{\Pi}^{R,1}\bar{\tau}) + (\hat{\Pi}^{R,1} \tau)( \delta\hat{\Pi}^{R,1}\bar{\tau}) \\
&=  (\hat{\Pi}^{R,1}D_\Xi\tau)(\hat{\Pi}^{R,1}\bar{\tau}) + (\hat{\Pi}^{R,1} \tau)(\hat{\Pi}^{R,1}D_\Xi\bar{\tau})\\
&=\hat\Pi^{R,1}\left(D_\Xi\tau\,\bar\tau + \tau\,D_\Xi\bar\tau\right)\\
&=\hat\Pi^{R,1}(D_\Xi(\tau\bar\tau)),
\end{equs}
where we have used the product rule for the Malliavin derivative, linearity, and multiplicativity of $\hat\Pi^{R,1}$, and the fact that $D_\Xi$ is a derivation. With multiplicativity settled we need only settle the case of planted trees $\bar\tau=\mathcal{I}_a(\tau)$. To that end, we can argue
\begin{equs}
\delta\hat\Pi_x^{R,1} (\mathcal{I}_a\tau) &= \delta\big(D^{a} K * { \Pi}_x^{R,1}\tau\big) \\ 
&\hphantom{aaaa}- \delta\left(\sum_{|k| < \textsf{deg}_1(\mathcal{I}_a\tau) }\frac{(y-x)^k}{k!}  \big(D^{a +k}  K * { \Pi}_x^{R,1} \tau\big)(x)\right)\\
&= (D^{a} K * \delta{ \Pi}_x^{R,1}\tau\big) \\
&\hphantom{aaaa}- \sum_{|k| < \textsf{deg}_1(\mathcal{I}_a\tau) }\frac{(y-x)^k}{k!}  \big(D^{a +k}  K * \delta{ \Pi}_x^{R,1} \tau\big)(x)\\
&= \hat\Pi_x^{R,1} \mathcal{I}_a D_\Xi\tau \\
&= \hat\Pi_x^{R,1} D_\Xi\mathcal{I}_a\tau,
\end{equs}
where we have used the induction hypothesis, the fact that Frechet derivatives are distributive over convolutions, and the definition of $D_\Xi$. To complete the induction we argue:
\begin{equs}\label{eq:proof10}
\delta\left(\Pi^{R,1}_{x}\tau\right) &= \delta\left(\hat\Pi^{R,1}_{x}\left(R\tau\right)\right) 
= \hat\Pi^{R,1}_{x}D_\Xi \left(R\tau\right) \\
&= \hat\Pi^{R,1}_{x}R \left(D_\Xi\tau\right) = \Pi^{R,1}_{x} \left(D_\Xi\tau\right).
\end{equs}
\end{proof}
One can remark that the commutativity property is not true when one considers the model $ \Pi_{x}^{R,0} $ that takes into account the specific degree of the perturbation $ \delta \xi $. This breakdown is attributable to the fact that one has to change the length of the Taylor expansions after applying $ \delta $.

The next object we will contend with is $\pxyo$. We recall from \cite[Eq. 2.64]{LOTP} that
\begin{equs}
\pxyo=\Pi_x(y),
\end{equs}
which corresponds to applying the model to certain planted trees. The following proposition then reveals that on $\CI(T)$  for the equation that the authors contend with in \cite{LOTP}, $\gi{0}{xy}$ is essentially the same as $\pxyo$.
\begin{proposition}\label{prop:gammapiequiv}
For $\tau$ such that $\CI(\tau)$ is of positive degree, one has
\begin{equs} \label{gamma_Pi}
\gi{0}{xy}\left(\tilde{\CI}^{+,0}(\tau)\right) = \left(\Pi^{R,0}_y\CI(\tau)\right)(x),
\end{equs}
where we recall the change of basis given in $\eqref{change_bais}$:
\begin{equation*} 
	\sum_{\ell} \frac{X^{\ell}}{\ell !} \mathcal{I}^{+,0}_{a + \ell}(\tau) = \tilde{\mathcal{I}}^{+,0}_{a}(\tau).
\end{equation*}
\end{proposition}
\begin{proof}
We recall from \cite{reg} that
\begin{equs}
	\gamma^{R,0}_{xy}\left( \tilde{\CI}^{+,0}(\tau) \right) = \left( f^{R,0}_x \CA \otimes f^{R,0}_y \right) \Delta^{\! +}_0\left(\tilde{\CI}^{+,0}(\tau)\right),
\end{equs}
where  $\CA$ is the antipode for $\CT^{+}$, the existence of which is the content of \cite[Th. 8.16]{reg}. In particular, the antipode satisfies:
\begin{equs}
\quad \CA \one = \one, \quad \CA X_i = - X_i, \quad	\CA \mathcal{I}^{+,0}(\tau) = - \sum_{\ell}\CM^+ \left[ \frac{X^{\ell}}{\ell !} \left( \CI^{+,0}_{\ell} \otimes \CA \right) \right]\Delta_0 \tau.
\end{equs}
 One has by using \eqref{def_delta_+} 
\begin{equs}
	\Delta^{\! +}_0\left(\tilde{\CI}^{+,0}(\tau)\right) & = 	\Delta^{\! +}_0 \left( 	\sum_{\ell} \frac{X^{\ell}}{\ell !} \mathcal{I}^{+,0}_{ \ell}(\tau) \right)
	\\ & = \sum_{\ell,k} \left( \frac{X^k}{k!} \otimes \frac{X^{\ell}}{\ell !} \right)	\Delta^{\! +}_0 \left( 	  \mathcal{I}^{+,0}_{ k+\ell}(\tau) \right)
	\\ & = \sum_{\ell,k,r} \left( \frac{X^k}{k!} \mathcal{I}^{+,0}_{ k+\ell + r} \otimes \frac{X^{\ell}}{\ell !} \frac{(-X)^r}{r!}  \right)	\Delta_0 	  \tau 
	\\ & \hphantom{aaaaaaaaa} + \sum_{\ell,k} \left( \frac{X^k}{k!}   \otimes \frac{X^{\ell}}{\ell !} \mathcal{I}^{+,0}_{ k+\ell}(\tau)  \right)	
	\\ & = \sum_{k} \left( \frac{X^k}{k!} \mathcal{I}^{+,0}_{ k} \otimes \one \right)	\Delta_0 	  \tau 
	\\ & \hphantom{aaaaaaaaa} + \sum_{\ell,k} \left( \frac{X^k}{k!}   \otimes \frac{X^{\ell}}{\ell !} \mathcal{I}^{+,0}_{ k+\ell}(\tau)  \right),
\end{equs}
where we have used 
\begin{equs}
 \sum_{\ell,r}	\frac{X^{\ell}}{\ell !} \frac{(-X)^r}{r!} = \mathbf{1}.
\end{equs}
Furthermore one has $f^{R,0}_x\left(\CA\tilde{\CI}^{+,0}(\tau)\right) = \left(\PPi^{R}\CI(\tau)\right)(x)$ which we can check as follows
\begin{equs}
	f^{R,0}_x ( \CA \tilde{\CI}^{+,0}(\tau)) & = f^{R,0}_x \left( - \sum_{\ell,r} \CM^+\left[\frac{(-X)^{\ell}}{\ell !} \frac{(X)^{r}}{r !} \left( \mathcal{I}^{+,0}_{\ell +r} \otimes \CA \right)\right] \Delta_0 \tau \right) \\
	& = f^{R,0}_x \left( -  \CM^+\left[ \left( \mathcal{I}^{+,0}\otimes \CA \right)\right] \Delta_0 \tau \right)
	\\ & = \left( (K * \Pi^{R,0}_x)(\cdot)(x) \otimes f_x^{R,0} \CA \right) \Delta \tau
	\\ & = \left(K * \PPi^{R} \tau\right)(x)
	\\ & = \left(\PPi^{R} \mathcal{I}(\tau)\right)(x),
\end{equs}
where we have used the facts
\begin{equs}
	\left( \Pi^{R,0}_x \otimes f^{R,0}_x \CA \right) \Delta_0 = \PPi^{R},
\end{equs}
and
\begin{equs}
	f^{R,0}_x\left(\mathcal{I}^{+,0}(\tau)\right) = -  (K * \Pi^{R,0}_x \tau)(x).
\end{equs}
The last identity is true because $ \CI(\tau) $ is always of positive degree due to our assumption.
On the other hand, one has from \eqref{eq:equality1} that
\begin{equs}
(\Pi^{R,0}_y \mathcal{I}(\tau))(x) &=  \left(  (\PPi^{R} \cdot)(x) \otimes f^{R,0}_y \right) \Delta_0 \CI(\tau)
\\ &  = \left(  (\PPi^{R} \CI \cdot)(x) \otimes f^{R,0}_y \right) \Delta_0 \tau \\
&\hphantom{aaaaaaaa}+ \sum_{k,\ell} \left(  \left(\PPi^{R} \frac{X^{k}}{k !} \right)(x) \otimes f^{R,0}_y\left( \frac{X^{\ell}}{\ell!} \CI^{+,0}(\tau)  \right) \right).
\end{equs}
We conclude by using the fact that:
\begin{equs}
	 (\PPi^{R} X^{k})(x) =  f^{R,0}_x( \CA X^k) = x^k.
\end{equs}
	\end{proof}

\noindent As explained in the introduction, the presence of divergent constants in the formulation of $\Pi^{-}_x$ \cite[Eq. 2.18]{LOTP} prompted the authors to instead first establish control of its so-called rough-path increment $\delta \Pi_x^{R,1} - \Pi^{R,1}_y  d \Gamma_{yx}^R$. Here $d\Gamma_{yx}^{R}$ is a modelled distribution in Hairer's regularity structures sense \cite[Def. 3.1]{reg} that is a structurally similar analytical quantity to $\delta\Gamma^{R}_{yx}$. In our context $d\Gamma_{yx}^{R}$ takes the following form.
\begin{definition}\label{def:dGamma}
For $\tau\in\mathcal{T}_0$ one has
\begin{equs}\label{eq:dGammadef}
d \Gamma_{yx}^R \tau = \mathcal{Q}_0 \left( \Gamma_{yx}^{R,0} \otimes f_x^{R,0} \right) \hat{\Delta}_0 D_{\Xi} \tau,
\end{equs}
where $ \mathcal{Q}_0 $ projects to zero, all trees containing the noise $\Xi_1$.
\end{definition}

\begin{remark}
It should be noted here that in \cite{LOTP} the authors work with the algebraic transpose $(\Gamma^{R,0}_{yx})^*$ and hence their $d\Gamma^{*}$ corresponds to the dual of \eqref{eq:dGammadef}.
\end{remark}
That $d\Gamma^R_{yx}$ is linear, is obvious and due to the multiplicity of $\Gamma_{yx}^{R,0},\,f_x^{R,0},\text{ and }\hat{\Delta}_0$, and the Leibniz rule that $D_\Xi$ satisfies, we have the following identity
\begin{equs}\label{eq:dGammamult}
d\Gamma_{yx}^{R}(\tau_1\tau_2) = d\Gamma_{yx}^{R}(\tau_1)d\bar\Gamma_{yx}^{R}(\tau_2) + d\Gamma_{yx}^{R}(\tau_2)d\bar\Gamma_{yx}^{R}(\tau_1),
\end{equs}
where $d\bar\Gamma_{yx}^{R}\coloneqq \left( \Gamma_{yx}^{R,0} \otimes f_x^{R,0} \right) \hat{\Delta}_0$. $d\bar\Gamma^{R}_{yx}$ is easily seen to multiplicative.
With these characterisations of $\delta\Pi_x^{R,1}$ and $d\Gamma^R_{yx}$, we are now able to turn our attention to the indicated identities.
We begin here with the so-called continuity expression $d\Gamma^{R}_{yx} - \Gamma^{R,1}_{yz}d\Gamma^{R}_{zx}$. To understand the need for this expression, we return to the rough-path increment $\delta \Pi_x^{R,1} - \Pi^{R,1}_y  d \Gamma_{yx}^R$. Particularly, to gaining control of this rough-path increment, where the authors in \cite{LOTP} employ a reconstruction argument, which requires as a hypothesis the control of the continuity expression $d\Gamma^{R}_{yx} - \Gamma^{R,1}_{yz}d\Gamma^{R}_{zx}$. For the details of this kind of reconstruction argument, we refer the readers to the discussion around \cite[Lem. 4.8]{LO22}.

The next theorem we present will establish the relationship between $\delta\Pi^{R,1}$ and $d\Gamma^{R}$ with a relatively lax assumption on $R$. This theorem is of central importance in this work because it can be used to extend our arguments here, to non-linearities coming from equations beyond the quasilinear parabolic one the authors consider in \cite{LOTP}.

\begin{theorem}\label{th:local without R} For every $\tau\in\mathcal{T}_0 $ satisfying Assumption~\ref{assumpt1}, and $R$ satisfying Assumption~\eqref{assumpt3} one has
\begin{equs} \label{local without R}
\left( \hat{\Pi}^{R,1}_y  d \Gamma_{yx}^R  \tau \right)(y)  =   \left(\delta \hat{\Pi}_x^{R,1} \tau \right)(y),
\end{equs}
and hence one has for $\tau\in\mathcal{T}_{0}$
\begin{equs} \label{local R}
\left( \Pi^{R,1}_y  d \Gamma_{yx}^R  \tau \right)(y)  =   \left(\delta \Pi_x^{R,1} \tau \right)(y).
\end{equs}
\end{theorem}

\begin{proof}
Instead of the usual inductive argument, we prove the \eqref{local without R} for all $\tau\in\mathcal{T}_0$ and derive \eqref{local R} therefrom. It is easy to see that \eqref{local without R} holds for $\tau\in\{\mathbf{1},X_i,\Xi_0\}$. To establish its "multiplicativity" we note first that for $\tau\in\mathcal{T}_{0}$ we have
\begin{equs}
\hat{\Pi}^{R,1}_y  d \bar\Gamma_{yx}^R  \tau &=  \left(\hat{\Pi}_y^{R,1}\Gamma_{yx}^{R,0}\otimes f_x^{R,0} \right) \hat{\Delta}_0 \tau  \\
&= \left(\hat{\Pi}_y^{R,0}\Gamma_{yx}^{R,0}\otimes f_x^{R,0} \right) \hat{\Delta}_0 \tau \label{eq:proof3}\\
&= \left(\hat{\Pi}_x^{R,0}\otimes f_x^{R,0} \right) \hat{\Delta}_0 \tau \\
&= \hat{\Pi}_x^{R,1}\tau.
\end{equs}
With \eqref{eq:proof3} in hand, if we are given $\tau,\bar\tau\in\mathcal{T}_0$, for which \eqref{local without R} holds, we see using the multiplicativity of $\hat\Pi^{R,1}$, the product rule and \eqref{eq:dGammamult}:
\begin{equs}
\left( \hat{\Pi}^{R,1}_y  d \Gamma_{yx}^R  \left[\tau\bar\tau\right] \right)(y) &= \Bigl( \left(\hat{\Pi}^{R,1}_y  d \Gamma_{yx}^R \tau\right) \hat{\Pi}^{R,1}_yd \bar\Gamma_{yx}^R\bar\tau \\
&\hphantom{aaaaaaaaa}+ \left(\hat{\Pi}^{R,1}_y  d \bar\Gamma_{yx}^R \tau\right) \hat{\Pi}^{R,1}_yd\Gamma_{yx}^R\bar\tau\Bigr)(y)\\
&= \left( \left(\delta \hat{\Pi}_x^{R,1} \tau\right) \hat{\Pi}_x^{R,1}\bar\tau + \left(\hat{\Pi}_x^{R,1} \tau\right) \delta \hat{\Pi}_x^{R,1}\bar\tau\right)(y) \label{eq:proofmult} \\
&= \delta\left(\hat{\Pi}_x^{R,1}\left[\tau\bar\tau\right]\right)(y).
\end{equs} 
This means that we need only prove \eqref{local without R} for decorated trees of the form $ \mathcal{I}_{a}(\tau) $. We know from Assumption~\ref{assumpt1} that
\begin{equs}
\textnormal{\textsf{deg}}_0( \mathcal{I}_{a}(D_\Xi \tau) ) > 0,
\end{equs}
which implies that one has
\begin{equs}
\left( \hat{\Pi}^{R,0}_y \mathcal{I}_a(\tau) \right) (y) = 0.
\end{equs}
Then, one has
\begin{equs}
	\label{computation}
	\begin{aligned}
\left( \hat{\Pi}^{R,1}_y  d \Gamma_{yx}^R  \mathcal{I}_a(\tau) \right)(y) & = \left( \hat{\Pi}^{R,1}_y \mathcal{Q}_0 \left( \Gamma_{yx}^{R,0} \otimes f_x^{R,0} \right) \hat{\Delta}_0 \mathcal{I}_a(D_{\Xi} \tau) \right)(y)
\\ & = \left( \left(  \hat{\Pi}^{R,0}_y \Gamma_{yx}^{R,0} \cdot \right)(y) \otimes f_x^{R,0} \right) \hat{\Delta}_0 \mathcal{I}_a(D_{\Xi} \tau)
\\ & =\left( \left(  \hat{\Pi}^{R,0}_x  \cdot \right)(y) \otimes f_x^{R,0} \right) \hat{\Delta}_0 \mathcal{I}_a(D_{\Xi} \tau)
\\ & = \left( \hat{\Pi}^{R,1}_x \mathcal{I}_a(D_{\Xi} \tau) \right) (y)
\\ & = \left(\delta \hat{\Pi}_x^{R,1} \mathcal{I}_a{\tau} \right)(y),
\end{aligned}
\end{equs}
where the Assumption~\ref{assumpt1} is required for doing away with $ \mathcal{Q}_0 $ and transforming $ \hat{\Pi}^{R,1}_y $ into  $ \hat{\Pi}^{R,0}_y $. Indeed, one has
\begin{equs}
	(\hat{\Pi}^{R,1}_y \mathcal{Q}_0  \Gamma_{yx}^{R,0} \mathcal{I}_a(D_{\Xi} \tau))(y) = 	(\hat{\Pi}^{R,0}_y   \Gamma_{yx}^{R,0} \mathcal{I}_a(D_{\Xi} \tau))(y).
	\end{equs}
This due to the fact that 
\begin{equs}
	\textnormal{\textsf{deg}}_0( \mathcal{I}_{a}(D_\Xi \tau') ) > 0,
\end{equs}
for every $ \tau' $ subtree of $ \tau $ including the root of $ \tau $ and therefore
\begin{equs}
	(\hat{\Pi}^{R,0}_y  \mathcal{I}_a(D_{\Xi} \tau'))(y) = 0.
\end{equs}
The evaluation at the point $ y $ acts as the projection $ \CQ_0 $. In \eqref{computation}, Proposition~\ref{prop:magicformula} is used in the fourth equality and Proposition~\ref{malliavin_derivative_model1} is used for the last equality. These propositions are used with $ \Pi $ replaced by $ \hat{\Pi} $ which is justified because in proving the proposition for the former we proved it for the latter as well.
 This proves \eqref{local without R} on the whole of $\mathcal{T}_{0}$.
Now to see how \eqref{local without R} implies \eqref{local R} we recall the commutative properties of $R$ - as noted in \eqref{EqCommutationRDelta}, \eqref{commute_D_Xi}, \eqref{commutation_Delta1_R}, and \eqref{eq:Gammacommute} - and the fact that $R$ maps into $\mathcal{T}_{0}$. Then with \eqref{local without R} at hand, we can argue that
\begin{equs}
\left( \Pi^{R,1}_y  d \Gamma_{yx}^R  \tau \right)(y) &= \left( \hat\Pi^{R,1}_y  R d \Gamma_{yx}^R  (\tau) \right)(y) \\
&= \left( \hat{\Pi}^{R,0}_y \mathcal{Q}_0 (R\otimes\text{id})\left( \Gamma_{yx}^{R,0} \otimes f_x^{R,0} \right) \hat{\Delta}_0 D_\Xi\,(\tau) \right)(y) \\
&=\left( \hat{\Pi}^{R,0}_y \mathcal{Q}_0 \left( \Gamma_{yx}^{R,0} \otimes f_x^{R,0} \right)(R\otimes\text{id}) \hat{\Delta}_0 D_\Xi\,(\tau) \right)(y) \\
&=\left( \hat{\Pi}^{R,0}_y \left( \Gamma_{yx}^{R,0} \otimes f_x^{R,0} \right)\hat{\Delta}_0 D_\Xi\,(R\tau) \right)(y) \\
&= \left(\hat\Pi^{R,1}_{x} \left(D_\Xi\,R\tau\right)\right)(y) \\
&= \left(\hat\Pi^{R,1}_{x}R\left(D_\Xi\tau\right)\right)(y) \\
&= \left(\delta \Pi_x^{R,1} \tau \right)(y).
\end{equs}
\end{proof}

We have remarked before that it is the lack of regularity in $\Pi^{-}_x$ that prompts the authors to exploit the rough-path increment, but it should be emphasised that the reason they are able to do so, is that they see no renormalisation when evaluating the rough-path increment at the diagonal. In particular, they have a divergence-free formula \cite[Eq. 4.50]{LOTP}, that they use in their reconstruction argument. In our next theorem, we derive the same formula in our context and exhibit the fact that we indeed see no renormalisation on the diagonal.
\begin{theorem} \label{main_result} Given a preparation map $ R $ such that it satisfies Assumption~\ref{assumpt2}, one has for $ \tau,\tau_1,\hdots,\tau_i \in \mathcal{T}_0 $, that
\begin{equs}
\left( F_{xy} \left[\left(\prod_i \mathcal{I}(\tau_i)\right) \mathcal{I}_2(\tau)\right] \right)(y) =  \left(\prod_i \left( \hat\Pi^{R,1}_x \mathcal{I}(\tau_i) \right) (y)\right)  (\hat F_{xy} \mathcal{I}_2(\tau))(y)
\end{equs}
where $F_{xy} = \delta \Pi_x^{R,1} - \Pi^{R,1}_y  d \Gamma_{yx}^R$ and $\hat F_{xy}=\delta\hat\Pi^{R,1}_{x} - \hat\Pi_{y}^{R,1}d\Gamma_{yx}^{R}$.
\end{theorem}
\begin{proof} We remark for future reference that $F_{xy}$ is linear but fails to be multiplicative because $d\Gamma^{R}_{yx}$ fails to be so. Then we have the following relationship:
\begin{equs}\label{eq:FandFhat}
\hat F_{xy}R &= (\delta\hat\Pi^{R,1}_{x} - \hat\Pi_{y}^{R,1}d\Gamma_{yx}^{R})R \\
&= (\hat\Pi^{R,1}_{x}D_\Xi - \hat\Pi_{y}^{R,1}d\Gamma_{yx}^{R})R \\
&= \hat\Pi^{R,1}_{x}D_\Xi R - \hat\Pi_{y}^{R,1}d\Gamma_{yx}^{R}R \\
&= F_{xy},
\end{equs}
where we have used the following observations that follow easily from definitions and the commutativity properties of $R$
\begin{equs}
\Pi^{R,1}_y  d \Gamma_{yx}^R &= \hat{\Pi}^{R,1}_y  d \Gamma_{yx}^R R,\qquad \delta\Pi_{x}^{R,1} &= \hat\Pi^{R,1}_{x}D_\Xi R.
\end{equs}
From \eqref{local without R} and in particular the multiplicativity we saw in \eqref{eq:proofmult}, we see that
\begin{equs}\label{eq:proof2}
\hat{F}_{xy}\left(\prod_i\mathcal{I}(\tau_i)\right)(y) = 0.
\end{equs}
Now given the trees $\tau, \tau_i \in \mathcal{T} $, we define the quantity
\begin{equs}
A = \left(\prod_i \mathcal{I}(\tau_i)\right) \mathcal{I}_2(\tau),
\end{equs}
which due to our assumption on the renormalisation map $R$, is such that
\begin{equs}
(R-\id) A = \sum_{j} c_j \prod_{i \in I_j} \mathcal{I}(\tau_{i,j}).
\end{equs}	  
Using the equality above, the linearity of $\hat F_{xy}$  and \eqref{eq:proof2} gives us that
\begin{equs}
\left(\hat F_{xy}(R - \text{id})A\right)(y) &= \left(\hat F_{xy}\sum_{j} c_j \prod_{i \in I_j}\mathcal{I}(\tau_{i,j})\right)(y) \\
&= \left(\sum_{j} c_j \hat F_{xy}\left(\prod_{i \in I_j} \mathcal{I}(\tau_{i,j})\right)\right)(y) = 0.
\end{equs}
Due to the equality above and \eqref{eq:FandFhat}, it is enough to prove that 
\begin{equs}
\left(\hat F_{xy} \prod_i \mathcal{I}(\tau_i) \mathcal{I}_2(\tau) \right)(y) =  \prod_i \left( \hat\Pi^{R,1}_x \mathcal{I}(\tau_i) \right) (y)  (\hat F_{xy} \mathcal{I}_2(\tau))(y).
\end{equs}
This can be proved inductively, indeed the base case follows from \eqref{eq:proof3} and \eqref{eq:proof2}:
\begin{equs}
\left(\hat F_{xy}\mathcal{I}(\tau) \mathcal{I}_2(\bar\tau) \right)(y) &= \left(\delta\hat\Pi_x^{R,1}\left(\mathcal{I}(\tau)\mathcal{I}_2(\bar\tau)\right) - \hat\Pi^{R,1}_yd\Gamma^{R}_{yx}\left(\mathcal{I}(\tau)\mathcal{I}_2(\bar\tau)\right)\right)(y) \\
&=  \Bigl(\delta\hat\Pi_x^{R,1}\mathcal{I}(\tau)\hat\Pi_x^{R,1}\mathcal{I}_2(\bar\tau)+\delta\hat\Pi_x^{R,1}\mathcal{I}_2(\bar\tau)\hat\Pi_x^{R,1}\mathcal{I}(\tau) \\
&-  \hat\Pi^{R,1}_y\left(d\Gamma^{R}_{yx}\mathcal{I}(\tau)d\bar\Gamma^{R}_{yx}\mathcal{I}_2(\bar\tau)\right) - \hat\Pi^{R,1}_y\left(d\Gamma^{R}_{yx}\mathcal{I}(\tau)d\bar\Gamma^R_{yx}\mathcal{I}_2(\bar\tau)\right)\Bigr)(y) \\
&= \left(\hat\Pi^{R,1}_x\mathcal{I}(\tau)\right)(y)\left[\delta\hat\Pi_x^{R,1}\mathcal{I}_2(\bar\tau)-\hat\Pi_{y}^{R,1}d\Gamma_{yx}^{R}\CI_2(\bar\tau)\right](y)\\
&+\left(\hat\Pi^{R,1}_x\mathcal{I}_2(\bar\tau)\right)(y)\left[\delta\hat\Pi_x^{R,1}\mathcal{I}(\tau)-\hat\Pi^{R,1}_yd\Gamma^{R}_{yx}\mathcal{I}(\tau)\right](y) \\
&= \left(\hat\Pi^{R,1}_x\mathcal{I}(\tau)\right)(y)\hat{F}_{xy}\mathcal{I}_2(\bar\tau)(y) + \left(\hat\Pi^{R,1}_x\mathcal{I}_2(\bar\tau)\right)(y)\hat{F}_{xy}\mathcal{I}(\tau)(y) \\
&= \left(\hat\Pi^{R,1}_x\mathcal{I}(\tau)\right)(y)\hat{F}_{xy}\mathcal{I}_2(\bar\tau)(y).
\end{equs}
From here it is only a matter of extending inductively to $\prod_i\mathcal{I}(\tau_i)$, which is done in the same vein as the base step above.
\end{proof}

In the next corollary, we provide a more explicit formula for the term involving $ \mathcal{I}_2(\tau)  $. It basically says that one has a commutative identity with derivatives and the map described above.

\begin{corollary} \label{coro_identi} For every decorated tree $ \tau $, one has
\begin{equs}
\left(\hat F_{xy}  \mathcal{I}_2(\tau)\right)(y) = \left(\partial_{x_1}^2 \left( \delta \hat\Pi_x^{R,1} - \hat\Pi^{R,1}_y d \Gamma_{yx}^R  \right)(\mathcal{I}(\tau))\right)(y).
\end{equs} 
\end{corollary}
\begin{proof}
This is an easy consequence of the commutativity properties, as noted in Proposition~\ref{prop:comDdelta_0} and \eqref{eq:derivativecommutativity}. Indeed one checks
\begin{equs}
\delta\hat\Pi^{R,1}_{x}\mathcal{I}_{2}(\tau) &= \hat\Pi^{R,1}_{x}D_\Xi\mathcal{I}_{2}(\tau) 
= \hat\Pi^{R,1}_{x}\mathcal{I}_{2}(D_\Xi\tau) \label{eq:malderivativecommutativity}\\
&= \partial^{2}_{x_1}\hat\Pi^{R,1}_{x}\mathcal{I}(D_\Xi\tau) 
= \partial^{2}_{x_1}\delta\hat\Pi^{R,1}_{x}\mathcal{I}(\tau),
\end{equs}
and that:
\begin{equs}
\hat\Pi^{R,1}_{y}d\Gamma^{R}_{yx}\mathcal{I}_{2}(\tau) &= \hat\Pi^{R,1}_{y}d\Gamma^{R}_{yx}D_{e_1}^{2}\CI(\tau)\\
&=\hat\Pi^{R,1}_{y}D^{2}_{e_1}d\Gamma^{R}_{yx}\CI(\tau)\\
&=\partial^{2}_{x_1}\hat\Pi^{R,1}_{y} d\Gamma^{R}_{yx}\mathcal{I}(\tau),
\end{equs}
where the second inequality is justified by noting that
\begin{equs}
\CD_p\hat{\Delta}_0 &= \left(\CD_p\otimes\Id\right)(\Id\otimes\hat{\Gamma}_0)\Delta_0 \\
&= (\Id\otimes\hat{\Gamma}_0)\left(\CD_p\otimes\Id\right)\Delta_0 \\
&\overset{\eqref{eq:comDdelta_0}}{=} (\Id\otimes\hat{\Gamma}_0)\Delta_0\CD_p = \hat{\Delta}_0\CD_p,
\end{equs} and similarly
\begin{equs}
\CD_p\Gamma^{R,0}_{yx} &= \left(\CD_p\otimes\Id\right)(\Id\otimes\gamma^{R,0}_{yx})\Delta_0 \\
&=(\Id\otimes\gamma^{R,0}_{yx})\left(\CD_p\otimes\Id\right)\Delta_0 \\
&=(\Id\otimes\gamma^{R,0}_{yx})\Delta_0\CD_p = \Gamma^{R,0}_{yx}\CD_p,
\end{equs}
from where we get:
\begin{equs}
\CD_pd \Gamma_{yx}^R &= \left(\CD_p\otimes\Id\right)\mathcal{Q}_0 \left( \Gamma_{yx}^{R,0} \otimes f_x^{R,0} \right) \hat{\Delta}_0 D_{\Xi} \\
&= \mathcal{Q}_0 \left(\CD_p\Gamma_{yx}^{R,0} \otimes f_x^{R,0} \right) \hat{\Delta}_0 D_{\Xi} \\
&= \mathcal{Q}_0 \left(\Gamma_{yx}^{R,0} \otimes f_x^{R,0} \right) \CD_p\hat{\Delta}_0 D_{\Xi} \\
&= \mathcal{Q}_0 \left(\Gamma_{yx}^{R,0} \otimes f_x^{R,0} \right) \hat{\Delta}_0 D_{\Xi} \CD_p \\
&= d\Gamma^R_{yx}\CD_p
\end{equs}
Now putting the two terms together and evaluating at the diagonal to kill positive degree trees in the second term, we get:
\begin{equs}
\left(\hat{F}_{xy} \mathcal{I}_2(\tau)\right)(y) & = \left(\left( \delta \hat\Pi_x^{R,1} - \hat\Pi^{R,1}_y d \Gamma_{yx}^R \right) \mathcal{I}_2(\tau)\right)(y) \\ 
& =  \left(\left(\partial_{x_1}^2 \delta\hat\Pi_x^{R,1} - \partial_{x_1}^2 \hat\Pi^{R,1}_y d \Gamma_{yx}^R \right) \mathcal{I}(\tau)\right)(y)\\
& = \left(\partial_{x_1}^2 \left( \delta  \hat\Pi_x^{R,1}  -  \hat\Pi^{R,1}_y d \Gamma_{yx}^R \right)  \mathcal{I}(\tau)\right)(y).
\end{equs}
\end{proof}
We finish the section by giving our version of \cite[Eq. 4.84]{LOTP}, which should be read in the context of the equation they have in \cite{LOTP}.
\begin{proposition} \label{ident_Pi}
One has for every $ \tau \in \mathcal{T}_0 $, such that  $\CI(\tau)$ is of positive degree, that
\begin{equs}
 \left( \delta \Pi_x^{R,1} - \Pi^{R,1}_y d \Gamma_{yx}^R \right) \mathcal{I}(\tau)  =  \left( \delta \Pi_x^{R,1} - \left( \delta \Pi_x^{R,1} \cdot \right)(y) -  {\Pi}^{R,1}_y P_{\mcI}d \Gamma_{yx}^R \right) \mathcal{I}(\tau),
\end{equs}
where $ P_{\mathcal{I}} $ is the projection onto planted trees of the form $ \mathcal{I}(\tau) $.
\end{proposition}
\begin{proof}
This is a consequence of the elementary decomposition:
\begin{equs}
\Pi^{R,1}_y d \Gamma_{yx}^R = {\Pi}^{R,1}_y P_{\mcI}d \Gamma_{yx}^R + \Pi^{R,1}_y (\id - P_{\mathcal{I}}) d \Gamma_{yx}^R.
\end{equs}
Then, 
\begin{equs}
\left(\Pi^{R,1}_y (\id - P_{\mathcal{I}}) d \Gamma_{yx}^R \mathcal{I}(\tau)\right)(y) & = (\Pi^{R,0}_y  d \Gamma_{yx}^R \mathcal{I}(\tau))(y)
\\ & = \left(  (\Pi_{x}^{R,0} \cdot)(y) \otimes f_x^{R,0} \right) \hat{\Delta}_0 D_{\Xi} \mathcal{I}(\tau)
\\ & = \left( \Pi_x^{R,1} D_{\Xi} \mathcal{I}(\tau) \right)(y)
\\ & = \left( \delta \Pi_x^{R,1}  \mathcal{I}(\tau) \right)(y),
\end{equs}
where we have used Proposition~\ref{prop:magicformula} and Proposition~\ref{malliavin_derivative_model1}. The first equality is true because any tree of the form $ \CI(\tau') $ with $ \tau' $ a subtree of $ \tau $ is of positive degree as a consequence of our assumption. The term $ d \Gamma_{yx}^R \mathcal{I}(\tau) $ produces terms of this form that are either killed by the projection $ \id - P_{\mathcal{I}} $ or by the evaluation of the model on the diagonal:
\begin{equs}
	(\Pi^{R,0}_y \mathcal{I}(\tau'))(y) = 0,
\end{equs}
which allows us to conclude.
\end{proof}

\begin{proposition} \label{control_gamma}
	One has for $\tau\in\mathcal{T}_0$
	\begin{equs}
		\left(d\Gamma^{R}_{yx}-\Gamma_{yz}^{R,1}d\Gamma^{R}_{zx}\right) \tau= \sum_{(D_\Xi\tau)}\tau^{(1)}&\gamma^{R,0}_{yx}(\tau^{(2)}) \\
		&- \sum_{(D_\Xi\tau)}\gamma^{R,0}_{zx}(\tau^{(2)})\sum_{(\tau^{(1)})}\tau^{(1,1)}\gamma^{R,0}_{yz}(\tau^{(1,2)}),
	\end{equs}
	where $\sum_{(\cdot)}$ refers to the sweedler notation as in \eqref{eq:sweedlernotation}.
\end{proposition}
\begin{proof}
	Let $\tau\in\mathcal{T}_{0}$. One can check then
	\begin{equs}
		\left(d\Gamma^{R}_{yx} - \Gamma^{R,1}_{yz}d\Gamma^{R}_{zx}\right)(\tau)\!&= \mathcal{Q}_0 \left( \Gamma_{yx}^{R,0} \otimes f_x^{R,0} \right)\hat{\Delta}_0 D_{\Xi}\tau \\
		&\hphantom{aaaaaa}-\Gamma^{R,1}_{yz}\mathcal{Q}_0 \left( \Gamma_{zx}^{R,0} \otimes f_x^{R,0} \right) \hat{\Delta}_0D_\Xi\tau \\
		&= \mathcal{Q}_0\Gamma^{R,0}_{yx}D_\Xi\tau - \Gamma^{R,1}_{yz}\mathcal{Q}_0\Gamma^{R,0}_{zx}D_\Xi\tau \label{eq:proof4} \\
		&=\mathcal{Q}_0\Gamma^{R,0}_{yx}D_\Xi\tau - \Gamma^{R,0}_{yz}\mathcal{Q}_0\Gamma^{R,0}_{zx}D_\Xi\tau.
	\end{equs}
The third equality above is true because $\mathcal{Q}_0$ maps into $\mathcal{T_0}$ and $\Gamma^{R,0}$, $\Gamma^{R,1}$ agree on $\mathcal{T_0}$. To see why the second inequality is true, we recall, to begin with, that in the triangular structure of $\hat\Delta_0$, the noise $\Xi_1$ ends up in $\hat\tau^{(2)}$ in each component $\hat\tau^{(1)}\otimes\hat\tau^{(2)}$ of the Sweedler notation \eqref{eq:hatdelta0tri}, save for the leading term $D_\Xi\tau\otimes\mathbf{1}$. As $\hat\tau^{(2)}$ feeds into $f^{R,0}_{x}$ and $\hat\tau^{(1)}\in\mathcal{T}_0$, we are able to do away with the projection and	\begin{equs}
		&\mathcal{Q}_0\left(\Gamma_{yx}^{R,0}(\hat\tau^{(1)})\otimes f_x^{R,0}(\hat\tau^{(2)})\right)-\Gamma^{R,1}_{yz}\mathcal{Q}_0\left(\Gamma_{zx}^{R,0}(\hat\tau^{(1)}) \otimes f_x^{R,0}(\hat\tau^{(2)}) \right)\\
		&\hphantom{3mm}=\Gamma_{yx}^{R,0}(\hat\tau^{(1)}) f_x^{R,0}(\hat\tau^{(2)})-\Gamma^{R,1}_{yz}\Gamma_{zx}^{R,0}(\hat\tau^{(1)})f_x^{R,0}(\hat\tau^{(2)})\\
		&\hphantom{3mm}=f_x^{R,0}(\hat\tau^{(2)})\left(\Gamma_{yx}^{R,0}\hat\tau^{(1)}-\Gamma^{R,1}_{yz}\Gamma_{zx}^{R,0}\hat\tau^{(1)}\right) \\
		&\hphantom{3mm}=f_x^{R,0}(\hat\tau^{(2)})\left(\Gamma_{yx}^{R,0}\hat\tau^{(1)}-\Gamma^{R,0}_{yz}\Gamma_{zx}^{R,0}\hat\tau^{(1)}\right)= 0,
	\end{equs}
For the leading term $D_\Xi\tau\otimes\mathbf{1}$ we are unable to remove the projection in \eqref{eq:proof4}. Now with reference with to \eqref{eq:gammacoaction}, we may rewrite \eqref{eq:proof4} as:
	\begin{equs}
		\CQ_0\Gamma^{R,0}_{yx}D_\Xi\tau - \Gamma^{R,0}_{yz}\mathcal{Q}_0\Gamma^{R,0}_{zx}D_\Xi\tau &= \mathcal{Q}_0(\text{id}\otimes\gamma_{yx}^{R,0})\Delta_0D_\Xi\tau \\
		&- (\Id\otimes\gamma^{R,0}_{yz})\Delta_0\mathcal{Q}_0(\Id\otimes\gamma^{R,0}_{zx})\Delta_0D_\Xi\tau.
	\end{equs}
	It only remains to analyse each of the terms on the right hand side in the previous equality. Appealing to Sweedler notation for $\Delta_0$ and noting that $\mathcal{Q}_0$ does away with the leading term with $\Xi_1$ on the left side of the tensor product, we get
	\begin{equs}
		\mathcal{Q}_0(\text{id}\otimes\gamma_{yx}^{R,0})\Delta_0D_\Xi\tau = \sum_{(D_\Xi\tau)}\tau^{(1)}\gamma^{R,0}_{yx}(\tau^{(2)}).
	\end{equs}
	A similar calculation is possible for the remaining term
	\begin{equs}
		(\Id\otimes\gamma^{R,0}_{yz})\Delta_0\mathcal{Q}_0(\Id\otimes\gamma^{R,0}_{zx})\Delta_0D_\Xi\tau &= (\Id\otimes\gamma^{R,0}_{yz})\Delta_0\sum_{(D_\Xi\tau)}\tau^{(1)}\gamma^{R,0}_{zx}(\tau^{(2)}) \\
		&= \sum_{(D_\Xi\tau)}\gamma^{R,0}_{zx}(\tau^{(2)})(\Id\otimes\gamma^{R,0}_{yz})\Delta_0\tau^{(1)} \\
		&= \sum_{(D_\Xi\tau)}\gamma^{R,0}_{zx}(\tau^{(2)})\sum_{(\tau^{(1)})}\tau^{(1,1)}\gamma^{R,0}_{yz}(\tau^{1,2})
	\end{equs}
	which allows us to conclude.
\end{proof}

\end{document}